\documentclass[10pt,a4paper]{amsart}

\usepackage{amsmath,amsthm,amssymb,mathrsfs,amsfonts,verbatim,enumitem,color}
\usepackage{etoolbox}
\usepackage{bbm}
\usepackage[all,tips]{xy}
\usepackage{graphicx,ifpdf}
\usepackage[T1]{fontenc}
\usepackage{hyperref}
\usepackage[right,displaymath,mathlines]{lineno}



\date{\today}

\title[Genericity on curves and applications]{Genericity on curves and applications: pseudo-integrable billiards, Eaton lenses and gap distributions}

\author[K.\ Fr\k{a}czek]{Krzysztof Fr\k{a}czek}
\address{Faculty of Mathematics and Computer Science, Nicolaus
Copernicus University, ul. Chopina 12/18, 87-100 Toru\'n, Poland}
\email{fraczek@mat.umk.pl}

\author[R.\ Shi]{Ronggang Shi}
\address{Shanghai Center for Mathematical Sciences, Fudan University, Shanghai 200433, PR China}
 \email{ronggang@fudan.edu.cn}

\author[C.\ Ulcigrai]{Corinna Ulcigrai}
\address{School of Mathematics, University of Bristol, Howard House, Queens Ave, BS8 1SD, Bristol, United Kingdom}
 \email{corinna.ulcigrai@bris.ac.uk}

\newtheorem{thm}{Theorem}[section]
\newtheorem{lem}[thm]{Lemma}
\newtheorem{cor}[thm]{Corollary}
\newtheorem{prop}[thm]{Proposition}

\theoremstyle{definition}
\newtheorem{de}[thm]{Definition}

\theoremstyle{remark}
\newtheorem{rem}[thm]{Remark}

\newtheorem*{nota}{Notation}
\newtheorem*{out}{Structure of the paper}

\newcommand{\rmnum}[1]{\romannumeral #1}
\newcommand{\Rmnum}[1]{\expandafter\@slowromancap\romannumeral #1@}
\newcommand{\R}{{\mathbb{R}}}

\newcommand{\Q}{{\mathbb{Q}}}
\newcommand{\C}{{\mathbb{C}}}
\newcommand{\Z}{{\mathbb{Z}}}
\newcommand{\N}{{\mathbb{N}}}
\newcommand{\T}{{\mathbb{T}}}
\newcommand{\vep}{\varepsilon}
\newcommand{\Int}{\operatorname{Int}}
\newcommand{\Fix}{\operatorname{Fix}}
\newcommand{\dd}{\, d}

\renewcommand{\xi}{v}

\begin{document}

\begin{abstract}
In this paper we prove results on  Birkhoff and Oseledets
genericity along certain curves in the space of affine lattices
and in moduli spaces of translation surfaces. We also prove
 applications of
these results to
dynamical billiards, mathematical physics and  number theory.
In the space of affine lattices $ASL_2(
\mathbb{R})/ASL_2( \mathbb{Z})$, we prove that almost every point
on a curve with some non-degeneracy assumptions is Birkhoff
generic for the geodesic flow. This
implies almost
everywhere genericity for some curves in the locus of branched
covers of the torus inside the stratum $\mathcal{H}(1,1)$ of
translation surfaces. For these curves (and more in general curves
which are well-approximated by horocycle arcs and satisfy almost
everywhere Birkhoff genericity) we also prove that almost every
point is Oseledets generic for the Kontsevitch-Zorich cocycle,
generalizing a recent result by Chaika and Eskin. As applications,
we first consider a class of pseudo-integrable billiards,
billiards in ellipses with barriers, which was recently explored
by Dragovi\'c and Radnovi\'c, and prove that for almost every
parameter, the billiard flow is uniquely ergodic within the region
of phase space in which it is trapped. We then consider any
periodic array of Eaton retroreflector lenses, placed on vertices
of a lattice, and prove that in almost every direction light rays
are each confined to a band of finite width. This generalizes a
phenomenon recently discovered by Fr\k{a}czek and Schmoll which
could so far only be proved for random periodic configurations.
Finally, a result on the gap distribution of fractional parts of
the sequence of square roots of positive integers, which extends previous work by
Elkies and McMullen, is also obtained.
\end{abstract}

\maketitle

\section{Introduction}
\begin{sloppypar}In this paper we prove three quite different results  which answer
recent open questions in dynamical systems, mathematical physics
and number theory. These three results all turn out to be
applications of two more  results on genericity  in homogeneous and
Teichm\"uller dynamics, which constitute the heart of this paper.
The three applications, which are explained in the following
sections of the introduction, concern more precisely the chaotic
properties (specifically ergodicity) of a recently discovered
class of \emph{pseudo-integrable billiards} (see
\S~\ref{sec:pseudo_integrable}), the behaviour of \emph{light rays
in periodic arrays of Eaton lenses} (see \S~\ref{sec:Eaton}) and
the \emph{gap distribution} of the sequence of \emph{fractional
    parts of square roots of positive integers} (see \S~\ref{sec:gap_intro}). The common
 result they exploit, which is stated in
\S~\ref{sec:Birkhoff},  concerns Birkhoff genericity (under the
geodesic flow) for almost every point on certain curves in the
space of affine lattices and in the moduli space of certain
translation surfaces.  Furthermore, for the application to Eaton
lenses, we also need a result on Oseledets genericity (for the
Kontsevich-Zorich cocycle) for almost every parameter describing
certain curves in the moduli space of translation surfaces, which
is stated in \S~\ref{sec:Oseledets}. The final section
\S~\ref{sec:outline} of this introduction provides an outline of
how the three main applications are related to these two
genericity results and describes the structure of the rest of the
paper.
\end{sloppypar}

\subsection{Pseudo-integrable billiards in ellipses}\label{sec:pseudo_integrable}
In this section we answer an open question 
on the ergodic properties of  billiards in ellipses
with barriers, a new class of pseudo-integrable billiards
recently described by Dragovi\'c and Radnovi\'c in \cite{Dra-Ra}.

A (planar) \emph{mathematical billiard} is a dynamical  system in
which a point-mass moves inside a \emph{billiard table} $T \subset
\mathbb{R}^2$, i.e.\ a bounded domain $T\subset\R^2$ whose
boundary $\partial T$ consists of finite number of smooth curves.
A \emph{billiard trajectory} is the trajectory described by the
point-mass which moves freely inside the table on segments of
straight lines and undergoes elastic collisions (angle of
incidence equals to the angle of reflection) when it hits the
boundary of the table.  The billiard flow $\{b_t\}_{t \in \R}$ is
defined on a subset of the phase space $S^1T$  that consists of
the points $(x, \xi ) \in T \times S^1$, where $S^1=\{ \xi \in \mathbb{C}:|\xi|=1\}$, such that if $x$ belongs
to the boundary of $T$ then $\xi $ is an inward unit tangent vector. For
$t \in \R$ and  $(x,\xi  )$ in the domain of $\{b_t\}_{t\in\R}$,
$b_t$ maps $ (x, \xi )$ to $b_t(x,\xi ) = (x_t,\xi_t)$,
where $x_t$ is the point reached after time $t$ by flowing at unit
speed along the billiard trajectory starting at $x$ in direction
of the unit vector $\xi $ and $\xi _t$ is the unit tangent vector to the trajectory
at $x_t$.

Billiards are sometimes divided into \emph{convex}, \emph{chaotic}
and \emph{polygonal}   billiards (see for example the survey by Tabachnikov
\cite{Tab}). The billiard in $T$ is called
\emph{integrable} if an open  subset of $S^1T$ is filled by
invariant sets so that the billiard flow restricted to each such
set is isomorphic to a linear flow on a two-dimensional torus
(only convex or polygonal billiards can be integrable). The
billiard  system inside any ellipse is integrable. Each invariant
set is determined by a confocal ellipse or hyperbola (called a
\emph{caustic}) and consists of all trajectories tangent the
caustic.
In \emph{chaotic billiards}, such as the famous Sinai billiard (a
square with a convex scatterer), no such invariant sets exist and
the billiard flow exhibits strong chaotic properties and in
particular is ergodic on the whole phase space (with respect to
the billiard invariant measure, see \cite{Tab}), i.e.\ there are
no billiard flow invariant sets of positive measure. An
intermediate behaviour is exhibited by  \emph{rational polygonal
billiards} (when $T$ is a polygon whose angles are rational
multiples of $\pi$), sometimes referred to as
\emph{pseudo-integrable} billiards: for every direction $v \in
S^1$ the billiard flow in direction $v$ on $ S^1T$ is confined to an invariant
surface in the phase space and the billiard flow restricted to
this invariant surface is typically ergodic, but in contrast with
integrable billiards, the invariant surface is not a torus but has
higher genus.  The study of rational polygonal billiards is
intimately connected to the rich area of research in translation
surfaces and Teichm\"uller dynamics.

Recently, Dragovi\'c and Radnovi\'c discovered a new class of
pseudo-integrable  billiards, see \cite{Dra-Ra}, given by
\emph{billiards in ellipses with barriers}, that we now describe.
Let $0<b<a$. Denote by $\{\mathcal{C}_{\lambda}: 0<\lambda<a\}$ the
family of confocal ellipses
(for $ \lambda\le b$ with $\lambda=b$ the set of two focal points) and hyperbolas (for
$b<\lambda<a$)
\[\frac{x^2}{a-\lambda}+\frac{y^2}{b-\lambda}=1.\]
Let
us consider the billiard flow inside the ellipse $\mathcal{C}_{0}$
with one linear vertical obstacle of length
$\sqrt{b}-\sqrt{b-\lambda_0}$, $0<\lambda_0<b$, which is
positioned as shown in Figure~\ref{EllipticBilliardEandH}. This
billiard table is denoted by $\mathcal{D}_{\lambda_0}$.

\begin{figure}[h]
\includegraphics[width=0.8\textwidth]{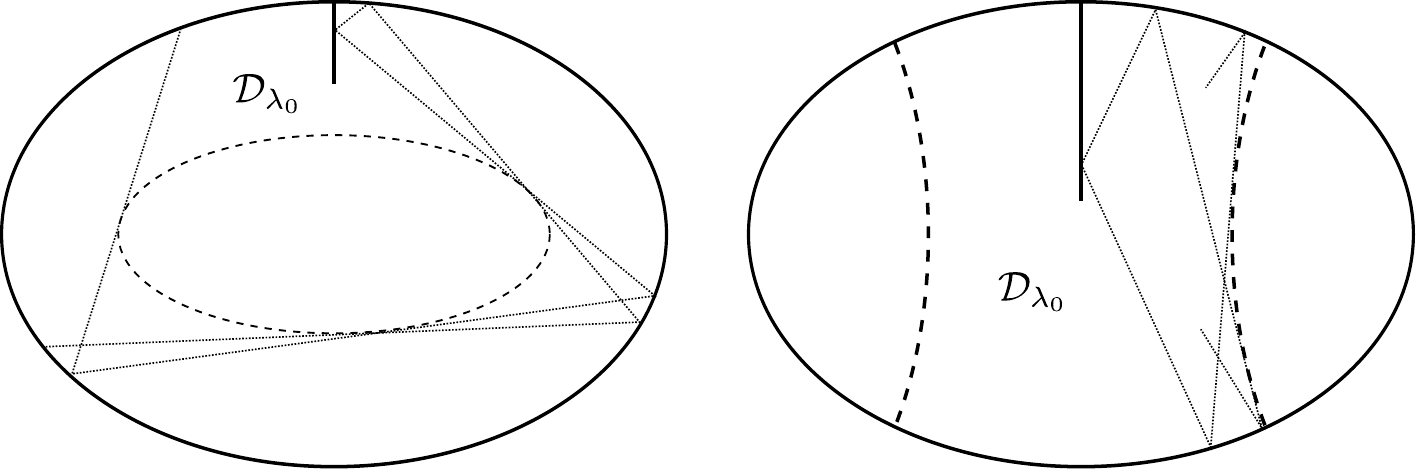}
\caption{}\label{EllipticBilliardEandH}
\end{figure}

Dragovi\'c and Radnovi\'c  observed in \cite{Dra-Ra} that the
phase space of the billiard flow on $\mathcal{D}_{\lambda_0}$
splits into invariant subsets $\mathcal{S}_{\lambda}$,
$0<\lambda<a$ so that the ellipse $\mathcal{C}_{\lambda}$ for
$0<\lambda\leq b$  or the hyperbola $\mathcal{C}_{\lambda}$ for
$b<\lambda<a$  is a caustic of all billiard trajectories in
$\mathcal{S}_{\lambda}$ (see Figure~\ref{EllipticBilliardEandH}).
Thus, the ellipse with a barrier display the typical trapping
phenomenon of a pseudo-integrable billiard.
In \cite{Dra-Ra} they construct examples of regions where
the flow is minimal but not uniquely ergodic, i.e.\  every trajectory is
 dense but not every trajectory is uniformly distributed. It is natural to expect
that this behaviour is exceptional.

We answer affirmatively the natural conjecture (raised e.g.~by Zorich)
that  for typical parameters $\lambda \in (0,a)$ of
the caustic parameterizing the invariant region
$\mathcal{S}_{\lambda}$   billiard trajectories are typically
dense in the invariant subset and furthermore, the billiard flow
restricted to  $\mathcal{S}_{\lambda}$ is uniquely ergodic.
Let us recall that a flow is \emph{uniquely ergodic} if it admits
a unique invariant probability measure, in which case it is also
automatically ergodic with respect to this measure.

\begin{thm}\label{thm:ellbill}
For almost every $\lambda\in(0,a)$ the billiard flow on $\mathcal{D}_{\lambda_0}$ restricted to $\mathcal{S}_\lambda$ is uniquely ergodic.
\end{thm}
This implies in particular that for typical parameters
$\lambda\in(0,a)$, every billiard trajectory inside
$\mathcal{S}_{\lambda}$   is dense and uniformly distributed. This
result is proved in \S~\ref{sec:ellipses_proof} and will be
deduced by the Birkhoff genericity results in
\S~\ref{sec:Birkhoff}.

\subsection{Periodic systems of Eaton lenses}\label{sec:Eaton}
In this section we study the behavior of light trajectories in  a
plane on which a lattice system of round retroreflector lenses of
the same size (called \emph{Eaton lenses}) is arranged.
An \emph{Eaton lens} is a circular lens 
which acts as a perfect retroreflector, i.e.\ so
that each ray of light after passing through the Eaton lens is
directed back toward its source,  see Figure~\ref{Eaton_lens}.
Let   $R>0$   denote the radius of the lens. The refractive index (RI for short) in an Eaton lens varies from
$1$ (outside the lens) to infinity (at the center of the lens, where it is not defined) according to the formula $RI=\sqrt{{2R}/{r}-1}$,  where $0<r\leq R$. As it was observed in \cite{Ha-Ha}, a
light ray entering the Eaton lens at a point $x_e$ moves (inside
the lens) in an elliptic orbit whose focal point coincides with
the center of the lens and then it leaves the lens at a point
$x_l$ so that the points $x_e$, $x_l$ are the ends of the minor
axis of the ellipse. Therefore, the direction of the light ray is
reversed after passing through the lens. There is only one
exception when the light ray hits the center of the lens and
disappears. We adopt the convention that when the light ray hits
the center, it turns back at the center and continues its motion
backwards.
\begin{figure}[h]
\includegraphics[width=0.4\textwidth]{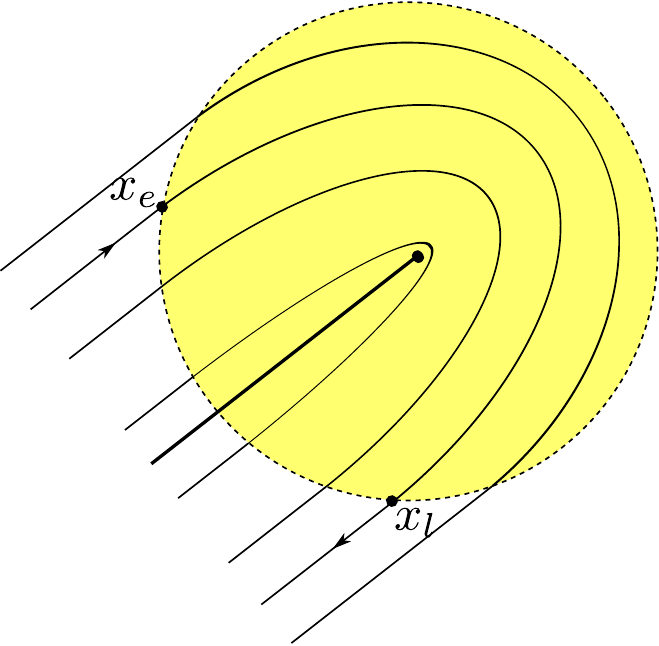}
\caption{Eaton lens and a parallel family of light rays}
\label{Eaton_lens}
\end{figure}

Denote by $L(\Lambda,R)$ the system of identical Eaton lenses of
radius $R>0$ arranged on the plane  $\R^2$ so that their
centers are placed at the points of a unimodular lattice
$\Lambda\subset\R^2$, see Figure~\ref{lattice_eaton}.
\begin{figure}[h]
\includegraphics[width=0.5\textwidth]{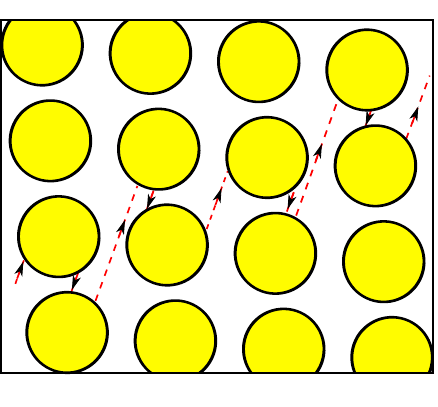}
\caption{The system of lenses $L(\Lambda,R)$}
\label{lattice_eaton}
\end{figure}
We deal only with pairs $(\Lambda,R)$ for which the lenses are pairwise disjoint and
such pairs are called \emph{admissible}. Admissibility is equivalent to $R<s(\Lambda)/2$, where $s(\Lambda)$
is the length of the shortest non-zero vector in  $\Lambda$.

Fr\k{a}czek and Schmoll in \cite{Fr-Sch} first observed the
phenomenon that light rays in lattice systems of Eaton lenses are
often trapped inside bands of finite width.  Note that
 for every
light ray there is a direction $\theta \in [0,2\pi]$ such
that the light ray flows in direction $\theta $ or $-\theta $
outside the lenses. Then $\theta $ is called the direction of the
light ray. This direction is unique modulo $\pi$.  Let us
say that a direction $\theta  $ is \emph{trapped} if there exists
$C(\theta )>0$ and $v(\theta ) \in S^1$ such
that every light ray on $L(\Lambda,R)$ in direction $\theta $ is
trapped in an infinite band of width $C(\theta)>0$ parallel to the unit tangent vector
$v(\theta)$.

Fr\k{a}czek and Schmoll considered \emph{randomly chosen lattices}
and  proved that for every $0<R< \sqrt{2\sqrt{3}}$ and for almost
every $R$-admissible lattice $\Lambda$ (with respect to the Haar
measure on the space of lattices),  light rays in the vertical
direction are trapped. They also provided explicit examples of
specific lattices and directions which are trapped. Their result,
though, does not provide any information for the behaviour of
typical light rays in a \emph{fixed} admissible lattice
configuration.

In this paper we answer this natural question (asked for example
by Marklof and by the referee of~\cite{Fr-Sch}) by describing
the behavior of light  orbits on $L(\Lambda,R)$ in direction
$\theta$ when an admissible pair $(\Lambda,R)$ is fixed and the
parameter $\theta$ varies.
\begin{thm}\label{thm:main}
Let $(\Lambda,R)$ be an admissible pair. Then a.e.~$\theta\in
[0,2\pi]$  is trapped.
\end{thm}
This result is proved in \S~\ref{sec:Eaton_proof}. As in
\cite{Fr-Sch},  we first reduce the system of Eaton lenses to  a
simpler model, a \emph{system of flat lenses}, which can be
unfolded and reduced to an infinite translation surface. For the
definition and for more results for this related system, we refer
the reader to \S~\ref{sec:Eaton_proof}.

\subsection{Gap distribution of fractional parts of square roots}
\label{sec:gap_intro}
Let $\N = \{1, 2, \ldots \}$.  Consider a sequence $\{ t_n\}_{n\in \mathbb{N}} \subset [0,1]$ which is \emph{equidistributed modulo one}, i.e.\  for any $0\leq a\leq b\leq 1$, $\# \{ 1\leq n \leq N  : t_n \in[a,b]\}/N$ tends to $|b-a|$ as $N \to\infty$. Given an initial block $\{t_1, \dots, t_N\}$, the lengths of  its complementary intervals, i.e.\ the connected components of $[0,1]\setminus \{ t_1, \dots, t_N\}$, are known as \emph{gaps}.
A natural question is to study  the \emph{gap distribution}, i.e.\ the  limiting behavior
 of the sequence of gaps of $\{t_1, \dots, t_N\}$, renormalized by their average length $1/N$, as $N\to \infty$  (see below for precise statements).  In a celebrated paper,   Elkies and McMullen considered  the sequence $\{\sqrt{n} \mod 1\}_{n\in \N}$ of \emph{fractional parts of square roots} and, establishing a connection with
homogeneous dynamics, showed that   the gap distribution of  $\{\sqrt{n}\mod 1\}_{n\in \mathbb{N}}$ exists and is a  non-standard distribution on $\R_{\ge 0}$ (in contrast to the Poissonian gap distribution displayed by random sequences).

More precisely, fix a real number $r\ge 1$ and let
$0 =t_1\le t _2\le   \ldots, \le t_{\lfloor r\rfloor}<1$  be
increasingly ordered fractional parts of  $\{\sqrt 1,\sqrt 2, \ldots,
\sqrt {\lfloor r\rfloor}\}$, where ${\lfloor r\rfloor}$
is the smallest integer less than or equal to $r$.  We set $t_{\lfloor r \rfloor +1}=1$
for convenience.
   The gap distribution of square roots of natural numbers  describes the limit behaviors
as $r\to \infty$  of renormalized consecutive gaps
 $\{(t_{n+1}-t_n) \lfloor r \rfloor\}_{n\le \lfloor  r \rfloor}$.
Elkies and McMullen proved in  \cite{em04} that there is a continuous probability density
$F:\R_{\ge 0}\to\R_{\ge 0}$ of the form
\begin{align}\label{eq;density F}
F(s)=\left\{
\begin{array}{ll}
{6}/{\pi^2} & t\in [0, {1}/{2}], \\
F_2(t) & t\in [{1}/{2}, 2], \\
F_3 (t)   & t\in [2, \infty),
\end{array}
\right.
\end{align}
where $F_2$ and $F_3$ are explicit real analytic
functions (we refer  the reader to \cite[Theorem 3.14]{em04} for
explicit formulas) such that for every interval $[a,b]\subset\R_{\ge 0}$
\[\frac{1}{\lfloor r \rfloor}\#\{1\leq n\leq \lfloor r\rfloor : (t_{n+1}-t_n)\lfloor r \rfloor \in[a,b]\}
\to \int_a^bF(s)\,ds\quad \text{ as } r\to \infty.\]
An effective version of this result was recently obtained by Browning
and Vinogradov \cite{bv}.

We consider the distribution  of normalized  gaps containing a fixed $s\in [0,1]$,
which was  suggested by Marklof. Let $
k(r, s)$ be the largest positive integer $k $ satisfying
$t_k \leq s$. We define
\begin{align}\label{eq;L_r(s)}
L_r(s)=
\lfloor r\rfloor(t_{k(r, s)+1}-t_{k(r, s)}).
\end{align}
For every $s\in[0,1]$ which is not a fractional part of the square
root of a natural number,   $L_r(s)$ is the normalized length of the
 gap  of the fractional parts of  $\{\sqrt 1, \ldots, \sqrt
{\lfloor r\rfloor}\}$ which contains $s$.
We are interested in the limit distribution of the sequence
$\{L_{r_n}(s)\}_{n\in \N}$ of the normalized  gaps
containing $s$ along geometric progressions $\{r_n=cq^n\}_{n\in \N}$
where $ q>1$ and $c\ge 1$. The reason we let $r_n$ goes geometrically but
not linearly is because from $r$ to $r+1$ we only add one number so at  most one  gap
can change.

The main result of this part is the following.
\begin{thm}\label{thm;g;gap}
Let $\{cq^n\}_{n\in\N}$ be any geometric progression with
$c\ge 1$ and $q>1$. Then for  Lebesgue almost every $s\in [0,1]$
the sequence  $\{L_{c q^n}(s)\}_{n\in\N}$ converges in distribution
to the probability measure $tF(t)\,dt$ on $\R_{\ge 0}$, i.e.\ for every
interval $[a,b]\subset\R_{\ge 0}$
\begin{align}\label{eq;g;gap}
\frac{1}{N}\#\{1\leq n\leq N:L_{c q^n}(s)\in[a,b]\}\to \int_a^b t F(t)\,dt\quad \text{ as } N\to \infty.
\end{align}
\end{thm}

Theorem \ref{thm;g;gap} raises a natural question, namely whether, for almost every  $s\in [0, 1]$, (\ref{eq;g;gap}) still hold
 if we replace the sequence $\{cq^n\}_{n\in\N}$ by $\{r_n(s)\}_{n\in \N} $ where $r_n(s)$
 is the natural number such that the gap containing $s$ is changing for the $n$-th time.

\subsection{Outline and structure of the paper}\label{sec:outline}
We now give a brief explanation of  how all these three applications rely on the same results (formulated in the next section \S~\ref{sec;main results}). At the end of this section we then explain the structure of the rest of the paper. It turns out that all three problems are related to \emph{genericity along curves} in the  space of affine lattices as follows.

\smallskip
In the paper  \cite{Dra-Ra} where Dragovi\'c and Radnovi\'c introduce \emph{billiards in ellipses with barriers} as a new class of pseudo-integrable billiards, a key observation is that any given billiard trajectory in an
ellipse with a barrier can be  mapped, by a suitable change of coordinates, to a trajectory of a billiard in a rectangle with a barrier (see Proposition~\ref{prop:DraRa_reduction} in \S~\ref{sec:ellipses_proof}). Billiards in \emph{rational polygonal tables} (such as the rectangle with a barrier) 
  have been successfully studied in the past decades through their connection with \emph{Teichm\"uller dynamics}.  A classical construction  allows us to \emph{unfold} the billiard to a translation surface (a surface with an almost everywhere Euclidean metric, see \S~\ref{sec:translation_surf}) so that the billiard trajectory becomes a trajectory of the linear flow (i.e. a flat geodesic) on the surface and
in virtue of a	milestone result in Teichm\"uller dynamics, Masur's criterium (see Theorem~\ref{thm:Masur} in \S~\ref{sec:translation_surf}), unique ergodicity  follows if one can show that the corresponding translation surface is \emph{Birkhoff generic} for the Teichm\"uller flow (see \S~\ref{sec;concept} for the definition).
	
	Let us remark that a celebrated result  by Kerkhoff, Masur and Smillie  \cite{Ke-Ma-Sm}  from the '$80$s  guarantees that in any rational billiard, the billiard flow in almost every direction is uniquely ergodic. Unfortunately, this result does not yield any information about trajectories in pseudo-integrable billiards. Indeed, as one changes the direction of the trajectory considered in the elliptical billiard (and hence its \emph{caustic}), the \emph{parameters} of the corresponding rectangular billiard table (such as lengths of sides and barrier) change too (while the slope of the image trajectories by the change of coordinates is fixed and equal to $\pm 1$) and describe a one-parameter family, or \emph{curve}, in the moduli space of translation surfaces. Some sufficient condition on the tangent vectors to a curve under which one can show unique ergodicity for a.e. parameter were described by Minsky and Weiss in \cite{Mi-We14}, exploiting their previous work in \cite{Mi-We02} (see also the comment after Corollary \ref{cor:unique_ergodicity} in  \S~\ref{sec:Birkhoff_tori}).
Finding suitable conditions on a curve of translation surfaces so that \emph{almost every point on the curve} is Birkhoff generic 
is currently a widely open problem.   In our setup, fortunately, the translation surfaces which are obtained by unfolding are all double covers of flat tori and thus the problem reduces to a homogeneous dynamics setup. The analogous result on genericity for almost every point on a curve in the space of affine lattices which we need is the first of our main results presented in the next section, see Theorem \ref{thm;equi} and  Corollary~\ref{cor:gencurve} in \S~\ref{sec;concept}.

	\smallskip
One can reduce also the study of \emph{systems of Eaton lenses} to the setup of translation surfaces.  For a fixed direction of light rays, by replacing each Eaton lens by a flat lens (defined in \S~\ref{sec;Eaton reduction})  and then taking a double cover, one can indeed reduce the behaviour of a light ray in the array to the study of a  linear trajectory on an \emph{infinite} translation surface (see \S~\ref{sec;Eaton reduction}). Since the planar array of Eaton lenses is $\mathbb{Z}^2$ periodic, the infinite surface obtained is a periodic surface which is a $\mathbb{Z}^2$-cover of a genus two surface.  
 The global behaviour of trajectories in the  $\mathbb{Z}^2$-cover turns out to be intimately related to  Lyapunov exponents of the Teichm\"uller flow. In particular, as shown by Fr\k{a}czek and Schmoll in \cite{Fr-Sch} (see also \S~\ref{sec:trapped_vs_genericity}), directions of bands which trap light rays are directions which correspond to negative Lyapunov exponents (in the plane spanned by the two homology classes which determine the cover). Thus, to establish that a direction is trapped one needs to prove that the corresponding genus two translation surface is \emph{Oseledets generic} (see \S~\ref{sec:Oseledets}). The main result by Fr\k{a}czek and Schmoll in  \cite{Fr-Sch} (mentioned in \S~\ref{sec:Eaton}, see also  \S~\ref{sec;Eaton reduction}), which deals with \emph{random lattices},  follows from the standard Oseledets ergodic theorem (which is recalled in \S~\ref{sec:Oseledets}). To understand the behaviour for a \emph{fixed lattice} and in particular to prove Theorem~\ref{thm:main} one needs to know that almost every point on the curve of genus two translation surfaces obtained  by the above reduction is Birkhoff and Oseledets generic. \emph{Birkhoff genericity} reduces as before to Birkhoff genericity for curves in the space of affine lattices, since the genus two surface turns out to be also in this case a double cover of a flat torus with a marked point. \emph{Oseledets genericity} along the curve is based on our second main result, Theorem \ref{thm:Oseledets} in section \ref{sec:Oseledets}. Let us remark that our two main results generalize in a special setup a recent work by Chaika and Eskin in \cite{Es-Ch}, where they prove Birkhoff and Oseledets genericity for curves of translation surfaces which are \emph{circles} (see \S~\ref{sec:Oseledets} for the precise formulation).

\smallskip
The seminal paper \cite{em04} by Elkies and McMullen on the \emph{gap distribution of fractional parts of square roots} was the first to describe and to exploit  the connection of this problem with homogeneous dynamics. In their paper it is shown that existence of the gap distribution follows from a result on equidistribution of certain curve in the space of affine lattices under the geodesic flow. Exploiting the same arguments, one can see in \S~\ref{sec:gaps_proof}   that our Theorem \ref{thm;g;gap} can be derived from our result on Birkhoff genericity under the geodesic flow for the same curve (Theorem~\ref{thm;equi}).

\begin{out}
The rest of the paper is organized as follows. In the next section \S~\ref{sec;main results} we recall background material and formulate the main results on which the applications are based.
In  \S~\ref{sec:Birkhoff},  we state
Theorem \ref{thm;equi} and Corollary~\ref{cor:gencurve} on Birkhoff genericity in the space of affine lattices, from which we derive in \S~\ref{sec:Birkhoff_tori} the Birkhoff genericity result on translation surfaces (Theorem~\ref{thm:BirOs}) used in the applications.
We recall background material on translation surfaces and Teichm\"uller dynamics in
\S~\ref{sec:translation_surf} and \S~\ref{sec:Oseledets}.
Theorem \ref{thm:Oseledets} on Oseledets genericity is stated in \S~\ref{sec:Oseledets}.
Using the results in \S~\ref{sec;main results}
 we
 prove the three main applications in   \S~\ref{sec:ellipses_proof} (Theorem~\ref{thm:ellbill}),   \S~\ref{sec:Eaton_proof} (Theorem~\ref{thm:main})
and   \S~\ref{sec:gaps_proof} (Theorem \ref{thm;g;gap}). Section \S~\ref{sec:Birkhoff_proof} is devoted to the proof of Theorem \ref{thm;equi} on Birkhoff genericity, while  section \S~\ref{sec:Oseledets_proof} contains the proof of   Theorem \ref{thm:Oseledets} on Oseledets genericity.   We refer the reader to the beginnings of sections \S~\ref{sec;main results}, \S~\ref{sec:Birkhoff_proof} and
\S~\ref{sec:Oseledets_proof} for a more detailed outline of the content for  each  of these sections.
\end{out}

\section{Results on genericity along curves}\label{sec;main results}
In this section we formulate the results on which the applications described in the introduction are based.  Two of the main theorems in ergodic theory, the Birkhoff ergodic theorem (recalled at the beginning of \S~\ref{sec:Birkhoff}) and the Oseledets multiplicative ergodic theorem (see  \S~\ref{sec:Oseledets}), guarantee that given an ergodic  measure preserving dynamical system  on a probability  space $(Y, \mu)$,  almost every point in $Y$ with respect to the measure $\mu$ is generic, in the sense that either the conclusion of the Birkhoff theorem
holds for $f\in L^1(Y)$ or  the Oseledets theorem holds for a cocycle under suitable assumptions. If one considers a \emph{curve} in the space $Y$, which has zero measure, a priori the conclusion of both Birkhoff and Oseledets theorems could fail for points in the curve. The main results in this paper concern two situations in which one can prove that almost every point along a certain curve is generic.  The first result concerns curves in the space of affine lattices satisfying a non-degeneracy condition and is stated in \S~\ref{sec:Birkhoff} after recalling Birkhoff ergodic theorem and the definition of the space of affine lattices. To state the second result, which concerns Oseledets genericity, we first recall background material on translation surfaces in \S~\ref{sec:translation_surf}, and recall the definition of the Kontsevich-Zorich cocycle, a linear cocycle over the Teichm\"uller geodesic flow on the space of translation surfaces, which plays a fundamental role in Teichm\"uller dynamics. The main result stated in \S~\ref{sec:Oseledets}  states that almost every point on certain curves in the moduli space of translation surfaces is Oseledets generic for the Kontsevich-Zorich cocycle. Finally, in \S~\ref{sec:Birkhoff_tori} we first introduce the space of translation surfaces of a special form, i.e.\ double covers of flat tori, which plays a key role in applications and describe its strict connection with the space of affine lattices. We then deduce from the Birkhoff genericity result in  the space of affine lattices a result on unique ergodicity for curves of translation surfaces which are double covers of tori, which is then directly used in the applications.

\smallskip
\begin{nota}
We define here some notation which is used throughout the paper. The ring of $d\times d$ matrices $M_{d}(\R)$ acts on $\R^d$ via linear transformations and
this action will appear in the paper as $h\xi$ for any $h\in M_{d}(\R)$ and $\xi\in\R^d$.
Let $\|\cdot\|$ denote the Euclidean norm on $\mathbb R^d$ as well
as operator norm of matrices defined  as
\begin{align} \label{eq;changing back}
\|h\|=\sup_{v\in \R^d\setminus \{0\}} \frac{\|hv\|}{\|v\|}.
\end{align}
The identity matrix in $M_{d}(\R)$ will be denoted by $Id$.
We use $|\cdot|$ to denote the absolute value of real numbers
or the Lebesgue measure of subsets of $\mathbb R$ according to the context.
$\# S$ will denote the  cardinality of a set $S$.
\end{nota}


\subsection{Birkhoff genericity along curves in $ASL_2( \mathbb{R})/ASL_2( \mathbb{Z})$}\label{sec:Birkhoff}\label{sec;concept}
We begin this section by first recalling the statement of Birkhoff ergodic theorem  and defining the concept of Bikhoff genericity.  Let $Y$ be a locally compact, Hausdorff  and second countable topological space.
Let $\{\psi_t\}_{ t \in \R}$ be a one-parameter topological flow on $Y$, i.e.~$t\mapsto \psi_t$ is
a homomorphism from $\R$ to the group of homeomorphisms of $Y$ and
$\R\times Y\ni(t,y)\mapsto \psi_t(y)\in Y$ is a continuous map.

\begin{sloppypar}
Let $\mu$ be a $\{\psi_t\}_{ t\in\R}$ invariant and ergodic
probability measure on $Y$.  The Birkhoff ergodic theorem says that,
given a real valued measurable  function $f$ on $Y$ with $\int_Y |f|\dd\mu<\infty$, one has
\end{sloppypar}
\begin{align}\label{eq;reorganize}
\lim_{T\to \infty}\frac{1}{T} \int_0^T f(\psi_t(y)) \,d  t= \int_Y f\dd \mu
\end{align}
for $\mu$ almost every $y\in Y$.
We say $y\in Y$ is \emph{Birkhoff generic}
with respect to  $(Y,\mu,  \psi_t)$
 if (\ref{eq;reorganize}) holds for
every $f\in C_c(Y)$ where
$C_c(Y)$  is the set of
 continuous compactly supported    functions on $Y$.

\smallskip
\begin{sloppypar}
We will be interested in Birkhoff genericity along curves in the space of affine lattices, that we now define.
Let $SL_2(\mathbb R)$ ($SL_2(\mathbb Z)$) be the group of $2$ by
$2$ matrices with real (integer) entries and determinant one. The
quotient space $SL_2(\mathbb R)/SL_2(\mathbb Z)$ parameterizes the
moduli space of two dimensional unimodular lattices. An
(unimodular) \emph{affine lattice} $\Lambda + \xi$ is determined
by a lattice $\Lambda$ in $SL_2(\mathbb R)/SL_2(\mathbb Z)$ and a
point $ \xi \in \mathbb{R}^2/\Lambda$ which describes  the shift
of new origin.
\end{sloppypar}

Let $G=ASL_2(\mathbb R):=SL_2(\mathbb R)\ltimes \mathbb R^2$,
$\Gamma =ASL_2(\mathbb Z):=SL_2(\mathbb Z)\ltimes \mathbb Z^2$ and
$X=G/\Gamma$. We represent elements of $G$ by
\[
(h, \xi):= \left(
\begin{array}{cc}
h & \xi \\
0 & 1
\end{array}
\right) \quad \text{where} \quad h \in SL_2(\mathbb R), \xi\in
\mathbb R^2.
\]
The multiplication for elements of $G$ is given by
\[
(h_1, \xi_1)\cdot(h_2, \xi_2)=(h_1h_2, h_1\xi _2+\xi_1).
\]
\begin{sloppypar}
One can check that the space $X$ parameterizes affine lattices:
if the coset $h (SL_2(\mathbb Z))$ in $SL_2(\mathbb R)/ SL_2(\mathbb
Z)$ represents the unimodular lattice $\Lambda_h=
h\mathbb{Z}^2$, then the coset $(h, \xi)\Gamma$ in $G/\Gamma$
represents the unimodular affine lattice $\Lambda_h+ \xi$.
\end{sloppypar}

Let us consider the following elements
\begin{equation}\label{subgroups}
a_t
=\left (
\begin{array}{ccc}
e^t & 0 & 0 \\
0 & e^{-t} & 0 \\
0 & 0 & 1
\end{array}
\right)
\quad \mbox{and}\quad
u(s_1,s_2 , s_3)=\left (
\begin{array}{ccc}
1 & s_1 & s_2 \\
0 & 1 & s_3 \\
0 & 0 & 1
\end{array}
\right).
\end{equation}
The action of the diagonal group $\{a_t:{t\in\R}\}$ on $X$ by left
multiplication  is known as \emph{geodesic flow} on the space of
affine lattices. It is well known that this action is ergodic with
respect to the probability  Haar measure $\mu_X$ on $X$.

\smallskip
We will consider curves in $X$ of the  form $u_\varphi(s)\Gamma$ where
 $\varphi: [0,1]\to \mathbb R$ is a $C^1$-function and  $u_\varphi(s):=u(s,
\varphi(s), 0)$. Let us remark that these curves have $\mu_X$ measure zero, hence the Birkhoff theorem does not yield information about genericity for points along these curves.
\begin{thm}\label{thm;equi}
Let $\varphi:[0,1]\to \mathbb R$ be a $C^1$-function such that
for any rational line $\mathcal L$ in $\mathbb R^2$ the Lebesgue measure
of  $\{s\in [0,1]: (s, \varphi(s))\in \mathcal L\}$
is zero. Then for almost every $s\in [0,1]$
the coset $u_{\varphi}(s)\Gamma \in X$ is Birkhoff generic with respect to  $(X, \mu_X, a_t)$.
\end{thm}

We can derive from Theorem \ref{thm;equi} similar results for an arbitrary base point $x\in G/\Gamma'$,
where $\Gamma'$ is a lattice in $G$ commensurable with $\Gamma$, i.e.~$\Gamma\cap \Gamma'$
has finite index in both $\Gamma$ and $\Gamma'$. We refer the reader to
 Corollary \ref{cor;general base} for more details.
For the applications, we will need a more general class of curves,
for  which equidistribution still holds and can be deduced from
Theorem~\ref{thm;equi}.

\begin{cor}\label{cor:gencurve}
Suppose that $\psi:[0,1]\to G$ is a $C^2$-curve of the form
\begin{align}\label{eq;psi}
{\psi}(s)=\left(\begin{pmatrix}
                          h_{11}(s) & h_{12}(s) \\
                          h_{21}(s) & h_{22}(s)
                        \end{pmatrix},
\begin{pmatrix}
                          v_1(s) \\
                          v_2(s)
                        \end{pmatrix}
                        \right),
\end{align}
so that the determinant of  the Wronskian matrix
\begin{align}\label{eq;M_psi}
M_{\psi}(s)=M_{h_{11},h_{12},v_1}(s)=\begin{pmatrix}
                          h_{11}(s) & h_{12}(s) &  v_1(s) \\
                           h_{11}'(s) & h_{12}'(s) &  v_1'(s) \\
                           h_{11}''(s) & h_{12}''(s) &  v_1''(s)
                        \end{pmatrix}
\end{align}
is  non-zero for a.e.~$s\in [0,1]$. Then given $(h,\xi)\in G$ and a lattice
$\Gamma'$ of $G$ commensurable with $\Gamma$ one has
for a.e.\ $s\in [0,1]$ the coset $\psi(s)(h,\xi)\Gamma'$ is Birkhoff
generic with respect to  $(G/\Gamma', \mu_{G/\Gamma'},a_t)$
where $\mu_{G/\Gamma'}$ is
the unique $G$-invariant probability measure  on $G/\Gamma'$.
\end{cor}

Elkies and McMullen proved in \cite{em04} that  the curve $\{u_\varphi(s)\Gamma: s\in [0,1] \}$, pushed under diagonal flow $a_t$,   equidistributes
as $t$ tends to infinity (i.e. the uniform measure on the curves, renormalized to be a probability measure, tends to the Haar measure).  Let us stress that the problem we are considering, namely showing that a.e. point on the curve is Birkhoff-generic under $a_t$, is independent and more subtle (heuristically, the two results are in a similar relation to proving $L^1$ convergence versus pointwise convergence of a family of functions).  Striking generalizations of the result by Elkies and McMullen have been proved by Shah, which proved equidistribution of curves under diagonal flows in much greater generality: in particular, in \cite{Sha2009-a} and \cite{Sha2009-b}, Shah proved asymptotic equidistribution 
 of certain analytic
curves in different homogeneous spaces.  On the other hand, it is an open question whether Birkhoff genericity  is true or
not in these settings.

The  proof of   Theorem~\ref{thm;equi} will take almost all of \S~\ref{sec:Birkhoff_proof} and is based on
  quantitative estimates on  the measures supported on  the curve translates obtained by using a \emph{height function}. While height functions are a classical tool since the seminal work \cite{emm98} by Eskin, Margulis and Mozes, the height function for our problem has to be carefully and ingeniously constructed.
  The techniques used in this paper are similar in spirit to those used to prove Birkhoff genericity   by Chaika and Eskin in \cite{Es-Ch} and by the second author in \cite{s}.
Different methods  are used instead in \cite{ksw} and \cite{snew}  to prove Birkhoff genericity  for one parameter  diagonalizable subgroup actions on homogeneous spaces: there the key step consists of proving effective double equidistribution of translates of volume  measures on horospherical orbits.



The proof of
Corollary~\ref{cor:gencurve}
will be given
at the end of \S~\ref{sec:outline_Birkhoff}.  It will  be clear  from  \S~\ref{sec:outline_Birkhoff}  that the assumptions on  $\varphi$ in Theorem~\ref{thm;equi} are necessary.
 The Birkhoff genericity result in the space of affine lattices has an immediate application to Birkhoff genericity for certain curves in a space of flat surfaces, which are branched covers of flat tori. In order to state this result (see \S~\ref{sec:Birkhoff_tori}), let us first recall some basic background material on translation surfaces.

\subsection{Background material on translation surfaces}\label{sec:translation_surf}
A \emph{translation surface} is a pair $(M,\omega)$ where $M$ is
an orientable compact surface and  $\omega$ is a translation
structure on $M$, that is a non-zero  holomorphic $1$-form also
called an \emph{Abelian differential}. Let us underline that the
translation structure $\omega$ determines both a complex structure
and an Abelian differential on $M$. Let $\Sigma_\omega\subset M$
denote the set of zeros of $\omega$ which are also the
\emph{singular points} of the translation structure. Let us
consider the volume form
$\nu_{\omega}=\frac{i}{2}\omega\wedge\overline{\omega}=\Re(\omega)\wedge\Im(\omega)$
which also will be treated as a volume measure. Since $M$ is
compact, $\nu_\omega (M)$ is finite and called the area
of $(M,\omega)$.

\begin{sloppypar}
Let  $M$ be a compact connected orientable surface and let $\Sigma\subset M$ be finite.
Denote by $\operatorname{Diff}^+(M,\Sigma)$ the group of
orientation-preserving diffeomorphisms of $M$ that fix elements of
$\Sigma$. Denote by $\operatorname{Diff}_0^+(M,\Sigma)$ the
subgroup of elements $\operatorname{Diff}^+(M,\Sigma)$ which are
isotopic to the identity. Let us denote by
$\Gamma(M,\Sigma):=\operatorname{Diff}^+(M,\Sigma)/\operatorname{Diff}_0^+(M,\Sigma)$
the {\em mapping-class} group. We will denote by
$\mathcal{T}(M,\Sigma)$ (respectively $\mathcal{T}_1(M,\Sigma)$ )
the {\em Teichm\"uller space of Abelian differentials}
(respectively of unit area Abelian differentials), that is the
space of orbits of the natural action of
$\operatorname{Diff}_0^+(M,\Sigma)$ on the space of all Abelian
differentials $\omega$ on $M$ with $\Sigma_\omega=\Sigma$
(respectively, the ones with total area $1$). We will denote by
$\mathcal{M}(M,\Sigma)$ ($\mathcal{M}_1(M,\Sigma)$) the {\em
moduli space of (unit area) Abelian differentials}, that is the
space of orbits of the natural action of
$\operatorname{Diff}^+(M,\Sigma)$ on the same space of (unit area)
Abelian differentials. Thus
$\mathcal{M}(M,\Sigma)=\mathcal{T}(M,\Sigma)/\Gamma(M,\Sigma)$ and
$\mathcal{M}_1(M,\Sigma)=\mathcal{T}_1(M,\Sigma)/\Gamma(M,\Sigma)$.
\end{sloppypar}

The group $SL_2(\R)$ acts naturally on  the space of Abelian differentials
as follows.
  Given a translation
structure $\omega$, consider the charts given by  local primitives
of the holomorphic $1$-form. The new  charts defined by
postcomposition of this charts with an element of $SL_2(\R)$ yield
a new complex structure and a new differential which is Abelian
with respect to this new complex structure, thus a new translation
structure. We denote by $g \omega$ the translation structure on
$M$ obtained by acting $g \in SL_2(\R)$ on a translation structure
$\omega$ on $M$.  Since the $SL_2(\R)$ action commutes with that
of $\operatorname{Diff}^+(M,\Sigma)$, it descends to action on
$\mathcal T _1(M, \Sigma)$ and $\mathcal M _1(M, \Sigma)$.
The {\em Teichm\"uller flow} is
the restriction of this action to the diagonal subgroup
$\{a_t=\operatorname{diag}(e^t,e^{-t}): t\in\R\}$ of $SL_2(\R)$ on
$\mathcal{T}_1(M,\Sigma)$ and $\mathcal{M}_1(M,\Sigma)$.
Here we slightly abuse the notation of $a_t$ which
 has different meaning in (\ref{subgroups}).
 For $\theta\in \R$ we let
\[
r_{\theta}=\left (
\begin{array}{cc}
\cos \theta & \sin \theta \\
-\sin \theta & \cos \theta
\end{array}
\right)\in SL_2(\R).
\]
We will
deal also with the rotations $\{r_{\theta}: {\theta\in \R}\}$ that
acts on a translation structure $\omega$  by
$r_\theta\omega=e^{i\theta}\omega$.

Let  $x_0\in\mathcal{M}_1(M,\Sigma)$ and denote by
$\mathcal{M}=\overline{SL_2(\R)x_0}$ the closure of the $SL_2(\R)$-orbit of $x_0$ in $\mathcal{M}_1(M,\Sigma)$. The celebrated result of Eskin, Mirzakhani and Mohammadi
(see \cite{EM} and \cite{EMM}) says that $\mathcal{M}\subset
\mathcal{M}_1(M,\Sigma)$ is an affine $SL_2(\R)$-invariant submanifold. 
Denote by $\mu_{\mathcal{M}}$ the corresponding affine
$SL_2(\R)$-invariant probability measure supported on
$\mathcal{M}$. Recall that
 $\mu_{\mathcal{M}}$ is ergodic for the Teichm\"uller flow.

 Recently Chaika and Eskin in \cite{Es-Ch} proved a  finer result saying for a.e.\ $\theta\in [0, 2\pi]$ the surface $r_\theta x_0$
is Birkhoff generic with respect to  $(\mathcal{M},\mu_{\mathcal{M}},a_t)$.
These results have applications to the  dynamics on a translation
surface, in virtue of  Masur's ergodicity criterion, which
constitutes one of the first and central results in Teichm\"uller
dynamics.

On a given translation surface $(M, \omega)$, for every
$\theta\in [0, 2\pi]  $ denote by
$F_\theta=F^{\omega}_\theta$ the vector field in direction
$\theta$ on $M\setminus\Sigma$, i.e.\
$\omega(F^{\omega}_\theta)=e^{i\theta}$. The corresponding
directional flow $\{\psi^{\theta}_t\}_{t\in\R}$, also called a
\emph{translation flow}, on $M\setminus\Sigma$ preserves the
volume measure $\nu_{\omega}$. We will use the notation
$\{\psi^{v}_t\}_{t\in\R}$  for the \emph{vertical flow}
(corresponding to $\theta = \frac{\pi}{2}$).  The flow
$\{\psi^{\theta}_t\}_{t\in\R}$ is uniquely ergodic if the area is
the unique invariant probability  measure.  Masur's ergodicity criterion
relates unique ergodicity of the vertical flow on $(M, \omega)$
with the behaviour of the Teichm\"uller geodesic through  $(M,
\omega)$.

\begin{thm}[Masur's criterion, see \cite{Ma92}]\label{thm:Masur}
If $\{a_t \omega\}_{t \in \mathbb{R}}$ returns infinitely often to a
compact set of $\mathcal{M}_1(M,\Sigma)$,  then the vertical flow
on $(M, \omega)$ is uniquely ergodic.
\end{thm}
Clearly, if $(M, \omega)$ is Birkhoff generic   with respect to  $(\mathcal M,\mu_{\mathcal M} , a_t)$,
then the assumption of the theorem holds.

\subsection{Oseledets genericity for Kontsevich-Zorich cocycle}\label{sec:Oseledets}
Let us recall one of the definitions of the Kontsevich-Zorich cocycle, which is a fundamental tool in the study of translation surfaces and Teichm\"uller dynamics.
Consider the projection $\pi:\mathcal{T}_1(M,\Sigma)\to
\mathcal{M}_1(M,\Sigma)$ where $\mathcal{T}_1(M,\Sigma)$ and $\mathcal{M}_1(M,\Sigma)$ are
respectively the Teichm\"uller and moduli space of translation
surfaces introduced in the previous section. 
Let $\Delta\subset\mathcal{T}_1(M,\Sigma)$
be a Borel fundamental domain for the action of
$\Gamma(M,\Sigma)$. For every $h\in SL_2(\R)$ and
$\tilde{x}\in\mathcal{T}_1(M,\Sigma)$ denote by
$\psi_{h,\tilde{x}}$ the only element of $\Gamma(M,\Sigma)$ such
that $\psi_{h,\tilde{x}}(h\cdot \tilde{x})\in\Delta$.

Let us consider the cocycle $A:SL_2(\R)\times
\mathcal{M}_1(M,\Sigma)\to GL(H_1(M,\R))$ given by
\begin{equation}\label{def:KZsl2r}
A(h,x)\zeta=(\psi_{h,\tilde{x}})_*\zeta\text{ for
$\tilde{x}\in\Delta$ such that
 $\pi(\tilde{x})=x$.}
\end{equation}
Note that the cocycle $A$ preserves $H_1(M,\Z)$ and the
non-degenerated symplectic structure on $H_1(M,\R)$ given by the
algebraic intersection form $\langle\,\cdot\, ,\,\cdot\,\rangle$.
Therefore, $A$ can be considered as a cocycle taking values in
$Sp(2g,\Z)$. By the {\em Kontsevich-Zorich (KZ) cocycle} we mean
the restriction of $A$ to the diagonal subgroup $\{a_t:{t\in\R}\}$
of $SL_2(\R)$, i.e.\
\[A^{KZ}:\R\times \mathcal{M}_1(M,\Sigma)\to GL(H_1(M,\R)),\qquad A^{KZ}(t,x)=A(a_t,x).\]

Let $\mu:=\mu_{\mathcal M}$ be the affine $SL_2(\R)$-invariant and ergodic probability
measure on $\mathcal{M}=\overline{SL_2(\R)x_0}$. By Moore's ergodicity
theorem $\mu$ is ergodic for the subgroup $\{a_t\}$ action.

Suppose that $W\subset H_1(M,\R)$ is a symplectic subspace (the
symplectic form restricted to $W$ is non-degenerated) of dimension
$2d$. Moreover, assume that $W$ is invariant for $SL_2(\R)$ action
on $\mathcal{M}$, i.e.\
\[h\in SL_2(\R),\ x\in \mathcal{M} \Rightarrow A(h,x)W=W.\]
Then we can pass to the restricted cocycle
$A_W^{KZ}:\R\times \mathcal{M}\to GL(W)$.
Let
 \[
e^{2\lambda_1(t, x)}\ge \ldots \ge e^{2\lambda_d(t, x)}\ge e^{-2\lambda_d(t, x)}\ge \ldots \ge e^{-2\lambda_1(t, x)}
\]
 be  the eigenvalues
of $\big(A_W^{KZ}(t, x)\big)^{\mathrm{tr}}A_W^{KZ}(t, x)$.
 Oseledets multiplicative ergodic theorem says that  there exist
 \[\lambda_1\geq\lambda_2\geq\ldots\geq\lambda_d\geq-
\lambda_d\geq\ldots\geq-\lambda_2\geq-\lambda_1,\]
 which are called \emph{Lyapunov exponents},
 such that for $\mu$ almost every
$x\in \mathcal M$  and any $1\leq i\leq d$
\begin{align}\label{eq;weak oseledec}
\lim_{t\to \infty}\frac{1}{t} \lambda_i(t, x) = \lambda_i.
\end{align}
If (\ref{eq;weak oseledec}) holds for every  $1\leq i\leq n$ we say $x$ is \emph{Oseledets generic} with
respect to
$(\mathcal M,  \mu, a_t, A^{KZ}_W)$, or simply Oseledets generic if the dynamical system and cocycle
is understood.

\smallskip
 Recently Chaika and Eskin in \cite{Es-Ch} considered \emph{circles} in the space of translation surfaces, i.e.\
 curves of the form $\{ r_\theta x_0, \theta \in [0,2\pi]\}$ where $x_0$ is any given translation surface and proved that for a.e.\ $\theta\in [0, 2\pi]$ the surface $r_\theta x_0$ is Oseledets generic.  Their result turns out to be a fundamental tool for
proving dynamical properties of directional flows on non-compact
periodic translation surfaces, cf.~for example \cite{AH}, \cite{DHL},
\cite{Fr-Ulc:nonerg} and \cite{Fr-Hu} for the Ehrenfest wind tree
model.

As we will see in \S~\ref{sec:Eaton_proof}, the Chaika-Eskin
theorem  is however not sufficient to examine the behavior of light rays
in lattice Eaton models as the angle of the rays changes. In such
models we need to show that a.e.\ point is Birkhoff and Oseledets
generic for curves different from those of the form $\theta\mapsto
r_\theta x_0$.

\smallskip
We will prove Oseledets genericity for a class of curves which we
will call \emph{well approximated by horocycles}.
Let
 $\mathcal T=\pi^{-1}(\mathcal M)$, where $\mathcal M$ is, as before, an $SL_2(\R)$ orbit closure. Consider a metric $d:\mathcal{T}\times\mathcal{T}\to\R_{\ge 0}$ satisfying the following $\Gamma$-\emph{invariance} and \emph{growth} conditions:
\begin{align}
& d(\gamma(\tilde{x}),\gamma(\tilde{y}))=d(\tilde{x},\tilde{y})\text{ for all }\gamma\in\Gamma(M,\Sigma); \label{assum:b2}\\
&d(\tilde{y},u(t)\cdot\tilde{y})\leq
|t|\text{ for every
}t\in\R
\quad \mbox{where } u(t)=\left(
\begin{array}{cc}
1 & t \\
0 & 1
\end{array}
\right).\label{assum:b4}
\end{align}
Note that the distance on $\mathcal{T}(M,\Sigma)$ defined in
\cite{AGY}  and derived from a Finsler $\Gamma$-invariant metric on
$\mathcal{T}(M,\Sigma)$  satisfies the above conditions.

The technical assumption required on the curves for which we will prove Oseledets-type results is the following.
\begin{de}\label{def:well_approx}
We say that $\varphi:I\to\mathcal{M}$ is \emph{well approximated by horocycles} if it is  of the form
$\varphi=\pi\circ\widetilde{\varphi}$ with
$\widetilde{\varphi}:I\to\mathcal{T}$ of class $C^1$ and
there exist $C(\varphi)>0$, $\rho\in\N$ and a metric $d$ satisfying properties
\eqref{assum:b2}, \eqref{assum:b4} such that
\begin{equation}\label{assum:phi}
d\big(a_t\cdot\widetilde{\varphi}(s+re^{-2t}),u(-r)\cdot
a_t\cdot\widetilde{\varphi}(s)\big)\leq C(\varphi)
|r|(1+|r|^\rho)e^{-t}
\end{equation}
for all $s\in I,r\in [-1, 1],t\ge 0$ with $s+re^{-2t}\in I$.
\end{de}

We can now state the first of two Oseledet genericity results along curves well approximated by horocycles (the other is Theorem~\ref{thm:Oseledetsgen} below).
\begin{thm}\label{thm:Oseledets}
Let $\mu$ be the probability affine measure on an
$SL_2(\mathbb{R})$-orbit closure $\mathcal{M}$.
Let $\varphi:I\to\mathcal{M}$ be a curve which is well
approximated by horocycles in the sense of
Definition~\ref{def:well_approx} and  such that for a.e.\ $s\in I$ the
point $\varphi(s)$ is Birkhoff generic with respect to  $(\mathcal M,\mu, a_t)$. Suppose that the sum
of positive Lyapunov exponents of the restricted KZ-cocycle
$A^{KZ}_W:\R\times \mathcal{M}\to GL(W)$ is less than one. Then for
a.e.\ $s\in I$ one has
$\varphi(s)$ is Oseledets generic with respect to  $(\mathcal M, \mu,  a_t, A_W^{KZ})$.
\end{thm}

\begin{rem}\label{rem:appl}
We remark that in our applications (in Section~\ref{sec:Eaton_proof}) the assumption that the sum
of positive exponents is less than one in Theorem~\ref{thm:Oseledets}  is automatically satisfied.
Indeed,  we apply Theorem~\ref{thm:Oseledets} to the KZ-cocycle
$A^{KZ}_W$ restricted to a two dimensional symplectic subspace $W$, which is symplectic orthogonal to the tautological bundle
(the two dimensional $SL_2(\R)$-invariant subbundle corresponding to extremal Lyapunov exponents $1$ and $-1$). Thus,
$A^{KZ}_W$ has at most one positive Lyapunov exponent which is less than one.
\end{rem}
In virtue of the above remark, Theorem~\ref{thm:Oseledets}  is sufficient for the purpose of the applications in this paper. Nevertheless, we state now a second Oseledets genericity result, for which the  assumption on Lyapunov exponents is \emph{not} required, but is replaced by a natural and verifiable condition on the subspace $W$ on which the action of the KZ cocycle is restricted.
This alternative statement  is based on the work \cite{EFW} by Eskin-Filip-Wright, which was not available when this paper was first written.  Using their work,  one can now prove that small variations on the proof of Theorem~\ref{thm:Oseledets} (explained in \S~\ref{sec:no_assumption}) allow to prove the following.

\smallskip
\begin{sloppypar}
Let $p:H^1(M,\Sigma,\R)\to H^1(M,\R)$ denote the forgetful map and let $T_\R M\subset H^1(M,\Sigma,\R)$ be the real part of the tangent space of $\mathcal{M}$ (see \S~\ref{sec:no_assumption} for more details).
\end{sloppypar}
\begin{thm}\label{thm:Oseledetsgen}
Assume that $\varphi:I\to\mathcal{M}$ is a curve well
approximated by horocycles  such that for a.e.\ $s\in I$ the
point $\varphi(s)\in\mathcal{M}$ is Birkhoff generic with respect to $(\mathcal M,\mu, a_t)$.
Suppose that $W\subset H_1(M,\R)$ is an $SL_2(\R)$-invariant symplectic subspace which is symplectic orthogonal to $p(T_\R M)\subset H_1(M,\R)$.  Then for
a.e.\ $s\in I$ one has
$\varphi(s)$ is Oseledets generic with respect to  $(\mathcal M, \mu,  a_t, A_W^{KZ})$.
\end{thm}
We believe that both Theorems~\ref{thm:Oseledets} and \ref{thm:Oseledetsgen} have further applications to billiards and systems of lenses. In particular, the assumption on $W$ in Theorem~\ref{thm:Oseledetsgen} is very natural since it is automatically satisfied in loci of translation surfaces which are branched covers (see \S~\ref{sec:no_assumption}  and in particular Lemma~\ref{lem:orthogonality} within), and these type of loci appear naturally when considering  systems with  inner symmetries.

\subsection{Genericity along curves in branched covers of a flat torus}\label{sec:Birkhoff_tori}
Let $M$ be a compact connected orientable surface of genus $2$ and
let $\Sigma\subset M$ be a two-point subset. Denote by
$\mathcal{M}^{dc}\subset\mathcal{M}(M,\Sigma)$ the (\emph{moduli})
\emph{space of double covers of flat tori} ramified over two
distinguished points, i.e.\ elements of $\mathcal{M}^{dc}$ are
represented by
translation surfaces of the form $(M,   q^*\omega_0)$, where
$(M_0, \omega_0)$ is a flat torus, $q:M\to M_0$ is a double
cover ramified over $q(\Sigma)$ and $q^*\omega_0$ is the
pullback of the Abelian differential $\omega_0$. If the area of
$\omega_0$ is $1$ then the area of $q^*\omega_0$ is $2$. Denote
by $\mathcal{M}_2^{dc}\subset \mathcal{M}^{dc}$ the subspace of
area $2$ translation surfaces.

Consider a translation surface $(M, \omega) $   given in
Figure~\ref{branchedcover}: opposite sides of the two
parallelograms  are identified by parallel translations, while the
sides of the slits are identified by parallel translations as
indicated in Figure~\ref{branchedcover}, i.e.\ between different
copies. The surface  $(M, \omega) $  has genus two and two conical
singularities of cone angle $4\pi$, which correspond to the
endpoints of the slit. Moreover, $(M,\omega)
\in\mathcal{M}^{dc}$.
\begin{figure}[h]
\includegraphics[width=0.6\textwidth]{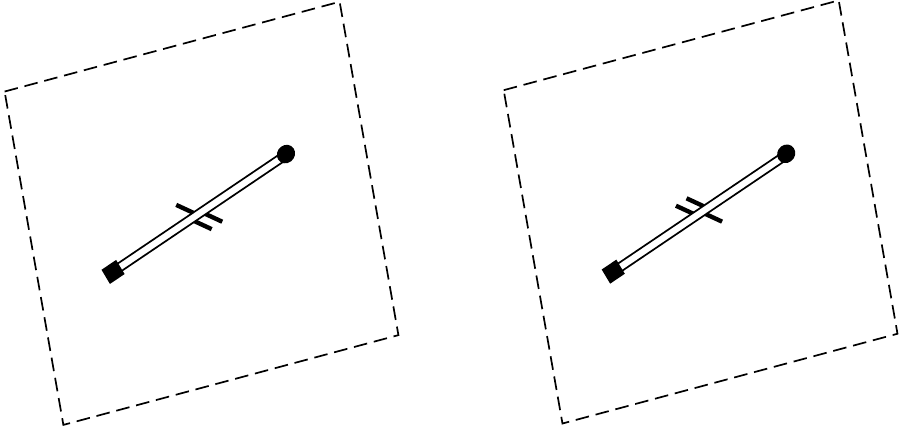}
\caption{A translation surface $(M, \omega)$ in the space $\mathcal{M}^{dc}$ of double covers of tori.}\label{branchedcover}
\end{figure}

The space $\mathcal{M}_2^{dc}\subset \mathcal{M}(M,\Sigma)$ is
closed and $SL_2(\R)$-invariant and under an assumption on
irrationality of $(M,\omega)\in \mathcal{M}_2^{dc}$ its   ${SL_2(\R)}$-orbit closure is  $\mathcal{M}_2^{dc}$.
In virtue of Lemma~\ref{lemma:correspondence}
below the  space $\mathcal{M}_2^{dc}$ and the dynamics  on
$\mathcal{M}_2^{dc}$ can be fully described in the language of
homogeneous spaces studied in \S~\ref{sec:Birkhoff}. Let $\mu_{dc}$ denote the affine $SL_2(\R)$-invariant
probability measure supported on $\mathcal{M}_2^{dc}$ (see
previous section~\ref{sec:translation_surf}).
Let
$G$, $\Gamma$ and $X$ be as in \S~\ref{sec:Birkhoff} and consider the lattice 
$\Gamma_2:=SL_2(\mathbb Z)\ltimes 2\mathbb Z^2$
which has  index $4$
 in $\Gamma$. Denote by $\mu_2$ the unique probability measure invariant under the left $G$-action on $G/\Gamma_2$. Let
\begin{align}\label{eq;mathcal x 2}
G':=G\setminus\{(h,h\Z^2):h\in SL_2(\R)\}\quad \text{and}\quad {X}_2:=G'\Gamma_2\subset G/\Gamma_2.
\end{align}
We remark here that $G'$ is an open subset of $G$ and  $\mu_2({X}_2)=1$.

\begin{lem}\label{lemma:correspondence}
There is a diffeomorphism $\Psi: \mathcal{M}_2^{dc}\to X_2$ such that:
\begin{enumerate}
\item[(i)] $\Psi$ commutes with $SL_2(\R)$-action;
\item[(ii)] $\Psi$ maps the natural affine measure $\mu_{dc}$ to $\mu_2$.
\end{enumerate}
\end{lem}

\begin{proof}
Every translation surface $(M,\omega)$ in $\mathcal{M}^{dc}_2$ is
uniquely determined by:
\begin{itemize}
\item the translation structure of the torus, i.e.\ $\T^2_\Lambda=\R^2/\Lambda$, where $\Lambda=h(\Z^2)$ for $h \in SL_2(\R)$;
\item two distinct marked points $\sigma_0,\sigma_1\in\T^2_\Lambda$ so that $\sigma_0=\mathbf {0}+\Lambda$;
\item and finally the double cover ramified over $\Sigma:=\{\sigma_0,\sigma_1\}$ is given
by a relative homology class $\gamma\in
H_1(\T^2_\Lambda,\Sigma,\Z/2\Z)\setminus
H_1(\T^2_\Lambda,\Z/2\Z)$.
\end{itemize}
Let $\xi _1=h(1,0)^{\mathrm{tr}}$, $\xi _2=h(0,1)^{\mathrm{tr}}$ be the
vectors in $\R^2$ and choose ${\xi}\in \sigma_1\subset\R^2$ (the
coset $\sigma_1$ is treated as a subset in $\R^2$) in the
parallelogram generated by $\xi _1,\xi _2$.
Denote by $\zeta_0,\zeta_{1},\zeta_2\in
H_1(\T^2_\Lambda,\Sigma,\Z/2\Z)$ the homology classes given by the
projections on $\R^2/\Lambda$ the vectors
$\overrightarrow{{\mathbf 0}{\xi}}$, $\xi _1$ and
$\xi _2$ respectively. Since $\zeta_0,\zeta_{1},\zeta_2$
establish a basis of $H_1(\T^2_\Lambda,\Sigma,\Z/2\Z)$ and
$\gamma$ is not absolute we can decompose
\[\gamma=\zeta_0+n_1\zeta_1+n_2\zeta_2\text{ with }n_1,n_2\in \Z/2\Z.\]
Finally define the map $\Psi:\mathcal{M}^{dc}_2\to
G/\Gamma_2$ by
\[ \Psi(M,\omega) = (h,{\xi}+h(n_1,n_2)^{\mathrm{tr}}) \mod \Gamma_2.\]
This map is one-to-one and onto ${X}_2$ in (\ref{eq;mathcal x 2}),
so we will identify $\mathcal{M}^{dc}_2$ with ${X}_2$. Let us consider the subgroup
$SL_2(\R)\ltimes\{{\mathbf 0}\}\subset G$
which we identify with $SL_2(\R)$. This group defines a natural
left action of $SL_2(\R)$ on $G/\Gamma_2$ and ${X}_2$.
Moreover, this action coincides with the standard
$SL_2(\R)$-action on $\mathcal{M}^{dc}_2$  via the correspondence
between $\mathcal{M}^{dc}_2$ and ${X}_2$.

Since $\mu_{dc}$ is a smooth $SL_2(\R)$-invariant and ergodic probability measure, the measure
$\Psi_* \mu_{dc}$ is also $SL_2(\R)$-invariant and ergodic.
Moreover, $\Psi: \mathcal{M}_2^{dc}\to X_2$ is  diffeomorphism and the subset ${X}_2\subset G/\Gamma_2$
is open and has full $\mu_2$ measure. According to Ratner's
measure classification theorem every $SL_2(\R)$-invariant and ergodic probability
measure on $G/\Gamma_2$ is either $\mu_2$ or supported on a
proper closed sub-manifold. Therefore we have $\Psi_*\mu_{dc}=\mu_2$.
\end{proof}

\begin{rem}\label{rem:corresp}
By  Lemma \ref{lemma:correspondence},  Birkhoff genericity of
$(M,\omega)$ in  $(\mathcal{M}_2^{dc},\mu_{dc},a_t)$ is equivalent to the Birkhoff
genericity of $\Psi(M,\omega)$ in $(G/\Gamma_2,\mu_2, a_t)$.
\end{rem}

\begin{rem}
Let $\mathcal T_2(M, \Sigma)$ be the Teichm\"uller space of Abelian differentials with area $2$. Denote by $\pi: \mathcal T_2(M, \Sigma)\to \mathcal M_2(M,
\Sigma)$ the quotient map and let $\mathcal T_2^{dc}:=\pi^{-1}(\mathcal M_2^{dc})$. Then we can lift the diffeomorphism $\Psi: \mathcal{M}_2^{dc}\to X_2$
to $\widetilde{\Psi}: \mathcal{T}_2^{dc}\to G'$ which also commutes with $SL_2(\R)$ action. This gives a natural identification of the Teichm\"uller space
$\mathcal{T}_2^{dc}$ of the double covers of flat tori with the subset $G'$ of the group $G$. Furthermore, $\widetilde{\Psi}$ conjugates the action of
$\Gamma(M,\Sigma)$ on $\mathcal{T}_2^{dc}$  with the right $\Gamma_2$-action on $G'$.
\end{rem}


Let
\[W:=\{\zeta\in H_1(M,\R):\tau_*\zeta=-\zeta\},\]
where $\tau:M\to M$ is the only nontrivial element of the deck transformation group of any ramified double cover.
The subspace $W$ does not depend on the choice of ramified cover. Moreover, $W$ is two dimensional symplectic and
$SL_2(\R)$-invariant. Therefore we can consider the restricted Kontsevich-Zorich cocycle $A^{KZ}_W:\R\times\mathcal{M}^{dc}_2\to GL(W)$.
The following result, proved at the end of this section, gives an effective criterion for
curves in $\mathcal{M}^{dc}_2$ to be almost everywhere Birkhoff and Oseledets (for $A^{KZ}_W$) generic.
\begin{thm}\label{thm:BirOs}
Let $\varphi:I\to\mathcal{M}^{dc}_2$ be a $C^2$-curve such that
$\Psi\circ\varphi(s)=\psi(s)g\Gamma_2$, where $g\in G$  and
$\psi:I\to G$ is a $C^2$-curve such that $\det M_{\psi}(s)\neq 0$  for a.e.~$s\in I$.
Then $\varphi(s)\in \mathcal{M}^{dc}_2$ is Birkhoff
and  Oseledets
generic
for
a.e.~$s\in I$.
\end{thm}

In virtue of Masur's criterion for unique ergodicity (stated in
\S~\ref{sec:translation_surf}), we immediately
get the following.

\begin{cor}\label{cor:unique_ergodicity}
Let $\{(M_s, \omega_s), s \in I\}$ be a curve in the space of translation structures so that its
image in   $ \mathcal{M}_2^{dc}$  satisfies the assumption of $\varphi$ in Theorem~\ref{thm:BirOs}.
Then for a.e.~$s \in I$ the vertical flow on $(M_s, \omega_s)$ is uniquely ergodic.
\end{cor}

This Corollary is used in \S~\ref{sec:ellipses_proof}  to prove the result on pseudo-integrable billiards (Theorem \ref{thm:ellbill}).  The already mentioned work   \cite{Mi-We14} by Minsky and Weiss gives a condition on curves in the space of translation surfaces (more precisely,  the result is stated for curves in the space of interval exchange transformations) which guarantees that for a.e. surface in the curve the vertical flow is uniquely ergodic. It might be possible to check the assumptions of their theorem in at least some special cases of curves coming from pseudo-integrable billiards. On the other hand, for the applications to Eaton lenses we need both  the Birkhoff genericity and Oseledets genericity result given by Theorem \ref{thm:main} in its full strength (in particular, Oseledets genericity requires Birkhoff genericity and not only recurrence as for Masur's unique ergodicity criterium).

\smallskip
The rest of the section is now devoted to proving
 Theorem \ref{thm:BirOs} using Theorem \ref{thm;equi} and Theorem \ref{thm:Oseledets}.

We fix a right invariant Riemannian metric on $G$ and use $d: G\times G\to \R_{\ge 0}$ to denote the corresponding metric. After rescaling we can assume
that $d(Id,u(1)\cdot Id)=1$. It is easy to see that the restriction of $d$ to $G'\times G'$ meets \eqref{assum:b2}-\eqref{assum:b4}. Since $G'$ is identified
with $\mathcal{T}_2^{dc}$ the transport of the metric $d$ to $\mathcal{T}_2^{dc}$, which will be also denoted by $d$, meets
\eqref{assum:b2}-\eqref{assum:b4} as well.

In the following lemma  we verify that the curve given in Theorem \ref{thm:BirOs} satisfies the
nondegeneracy condition.

\begin{lem}\label{lem;countable sub}
Let $\psi:I\to G$ ($I$ is a compact interval) be a $C^2$-curve of
the form $\psi(s)=( h(s),\xi(s))\cdot ( h_0,\xi_0)$. If
\begin{equation}\label{cond:det1}
 h_{11}(s) h'_{12}(s)- h_{12}(s) h'_{11}(s)\neq 0\quad\text{ for every }\quad s\in I
\end{equation} then there is a $C^2$-diffeomorphism  $\kappa: I_0\to I$ from some closed
interval $I_0$ such that  the curve $\psi\circ \kappa:I_0\to G$ satisfies the condition
\eqref{assum:phi} with respect to   metric  $d$ given above.
\end{lem}
\begin{proof}
Let us consider the ordinary differential equation
\[\kappa'(s)=1/( h_{11}(\kappa(s)) h'_{12}(\kappa(s))- h_{12}(\kappa(s)) h'_{11}(\kappa(s)))\]
and let $\kappa: I_0\to I$ be its solution.
After changing the parameter we can pass to the case where
$ h_{11}(s) h'_{12}(s)- h_{12}(s) h_{11}'(s)=1$ for
every $s$ and then we need to show that \eqref{assum:phi} holds.

Write  $l:=re^{-2t}$ and let us first estimate form above the norm
\begin{align*}
\|Id-&u(-r)\cdot a_t\cdot \psi(s)\cdot\psi(s+ l  )^{-1}\cdot a_{-t}\|\\
&=\|a_t\cdot(Id-u(- l  )\cdot
 \psi(s)\cdot\psi(s+ l  )^{-1})\cdot a_{-t}\|.
\end{align*}
Let
\[\upsilon( l  ,s):=Id-u(- l  )\cdot \psi(s)\cdot\psi(s+ l  )^{-1}\text{ and }\underline{\upsilon}( l  ,s):=Id-u(- l  )\cdot  h(s)\cdot h(s+ l  )^{-1}.\]
Since $\upsilon$ is of class $C^2$, $\upsilon(0,s)=0$ and
$\|\frac{\partial}{\partial  l  }\upsilon( l  ,s)\|\leq
C\|\psi\|^3_{C^1}$ for some $C>0$, by the mean value theorem,
\[\|\upsilon( l  ,s)\|\leq C\|\psi\|^3_{C^1}| l  |.\]
Moreover,
\[\frac{\partial}{\partial  l  }\underline{\upsilon}(0,s)=\begin{pmatrix}
  0 & 1 \\
  0 & 0
\end{pmatrix}- h(s)\cdot( h(s)^{-1})'\]
and $\|\frac{\partial^2}{\partial  l  ^2}\upsilon( l  ,s)\|\leq
C'\|\psi\|^2_{C^2}$ for some $C'>0$. It follows that
\[\frac{\partial}{\partial  l  }{\upsilon}_{12}(0,s)=\frac{\partial}{\partial  l  }\underline{\upsilon}_{12}(0,s)=
1- h'_{11}(s) h_{12}(s)+ h_{11}(s) h'_{12}(s)=0.\] Then, by
Taylor's formula,
\[|\upsilon_{12}( l  ,s)|\leq C'\|\psi\|^2_{C^2}| l  |^2.\]
Therefore,
\[a_t\cdot \upsilon( l  ,s)\cdot a_{-t}=\begin{bmatrix}
  \upsilon_{11}( l  ,s) & e^{2t}\upsilon_{12}( l  ,s) & e^t\upsilon_{13}( l  ,s) \\
  e^{-2t}\upsilon_{21}( l  ,s) & \upsilon_{22}( l  ,s) & e^{-t}\upsilon_{23}( l  ,s) \\
  0 & 0 & 0
\end{bmatrix}\]
with
\begin{align*}
|\upsilon_{11}( l  ,s)|\leq C\|\psi\|^3_{C^1}|r|e^{-2t},\ |e^{-2t}\upsilon_{21}( l  ,s)|\leq C\|\psi\|^3_{C^1}|r|e^{-4t}\\
 |e^{2t}\upsilon_{12}( l  ,s)|\leq C'\|\psi\|^2_{C^2}|r|^2e^{-2t},\ |\upsilon_{22}( l  ,s)|\leq C\|\psi\|^3_{C^1}|r|e^{-2t}\\
 |e^t\upsilon_{13}( l  ,s)|\leq
C\|\psi\|^3_{C^1}|r|e^{-t}, \ |e^{-t}\upsilon_{23}( l  ,s)|\leq
C\|\psi\|^3_{C^1}|r|e^{-3t}.
\end{align*}
It follows that there exists $C(\psi)>0$ such that
\[\|a_t\cdot \upsilon( l  ,s)\cdot a_{-t}\|\leq C(\psi)|r|(1+|r|)e^{-t}\leq 2C(\psi).\]
Moreover, there exists $C''>0$ such that for every $g\in G$ with $\|Id -g\|\leq 2C(\psi)$
we have $d(Id,g)\leq C''\|Id-g\|$. It follows that
\begin{align*}
d\big(&a_t\cdot\psi(s+re^{-2t}),u(-r)\cdot a_t\cdot\psi(s)\big)\\
&= d\big(Id,u(-r)\cdot a_t\cdot\psi(s)\cdot\psi(s+re^{-2t})^{-1}\cdot a_{-t}\big)\\
&\leq C''\|Id-u(-r)\cdot a_t\cdot\psi(s)\cdot\psi(s+re^{-2t})^{-1}\cdot a_{-t}\|\\
&=C''\|a_t\cdot \upsilon( l  ,s)\cdot a_{-t}\|\leq C'' C(\psi)|r|(1+|r|)e^{-t}.
\end{align*}
\end{proof}

\begin{proof}[Proof of Theorem~\ref{thm:BirOs}]
We choose  a countable collection $\{I_k\}_{k\in \N}$
of closed subinvervals of $I$ so that
$\bigcup_{k\in \N} I_k$ has full measure in $I$ and
(\ref{cond:det1}) holds for every $s\in I_k$.
We fix an interval $I_k$, it suffices to show that Theorem \ref{thm:BirOs} holds
for a.e.~$s\in I_k$. The Birkhoff genericity of $\varphi(s)$ for a.e.~$s\in I_k$ follows
from Corollary \ref{cor:gencurve} and Remark~\ref{rem:corresp}.
By Lemma \ref{lem;countable sub} and the correspondence between $X_2$ and $\mathcal{M}^{dc}_{2}$ the curve $\varphi|_{I_k}$
has a parameterization such that $ \varphi |_{I_k}\circ \kappa$ is well approximated by horocycles and $\kappa$ is a $C^2$-diffeomorphism.
Since $W$ is two dimensional and is the symplectic orthocomplement to the tautological bundle, using Theorem \ref{thm:Oseledets} together
with Remark~\ref{rem:appl} we conclude that for a.e.~$s\in I_k$, the element
$\varphi(s)\in \mathcal{M}^{dc}_{2}$ is Oseledets generic which completes the proof.
\end{proof}

\section{Ergodicity in elliptical billiards with barriers}\label{sec:ellipses_proof}
In this section we prove Theorem~\ref{thm:ellbill}  on unique
ergodicity in elliptical billiards with barriers. We exploit the reduction of these
family of billiards to polygonal billiards discovered by
Dragovi\'c and Radnovi\'c in \cite{Dra-Ra} and stated in
\S~\ref{sec:reduction_ellipses}. The billiard flow on each
$\mathcal{S}_\lambda$ is isomorphic to the vertical flow on a
surface from $\mathcal{M}^{dc}_2$, which gives a smooth curve in
$\mathcal{M}^{dc}_2$. In view of
Corollary~\ref{cor:unique_ergodicity}, it suffices to show that
the curve satisfies a non-degeneracy property.
In \S~\ref{sec:ellbillproof}, we prove the non-degeneracy
property for the aforementioned curve which yields
Theorem~\ref{thm:ellbill}.

\subsection{Reduction to a family of polygonal billiards}\label{sec:reduction_ellipses}
Let $\mathcal{D}_{\lambda_0}$ be the elliptic billiard table with
a barrier described in \S~\ref{sec:pseudo_integrable} and let
$\mathcal{S}_{\lambda}$ for $0<\lambda< a$ be one of its invariant
regions. The fundamental observation, made by Dragovi\'c and
Radnovi\'c in \cite{Dra-Ra}, is that the billiard flow on
$\mathcal{S}_{\lambda}$ can be reduced to a polygonal billiard
flow   by a suitable change of coordinates. The polygonal billiard
tables $\mathcal{P}_\lambda$ which are obtained after change of
coordinates (given by elliptic integrals) are either
a
non-planar billiard table $\mathcal{P}_\lambda$ which is a
cylinder with a vertical slit (for $0<\lambda< b$),
shown in Figure~\ref{EllipticBilliard},
or for  $b<\lambda<a$,
a rectangular
billiard table with a vertical slit, shown in
Figure~\ref{EllipticBilliardHyp}. Billiard trajectories in
$\mathcal{S}_{\lambda}$, which are tangent to the caustic
$\mathcal{C}_{\lambda}$, are mapped to billiard trajectories in
$\mathcal{P}_\lambda$ all in the \emph{same} family of directions
$\pm\pi/4$, $\pm3\pi/4$. More precisely, we have the following.

\begin{figure}[h]
\includegraphics[width=0.8\textwidth]{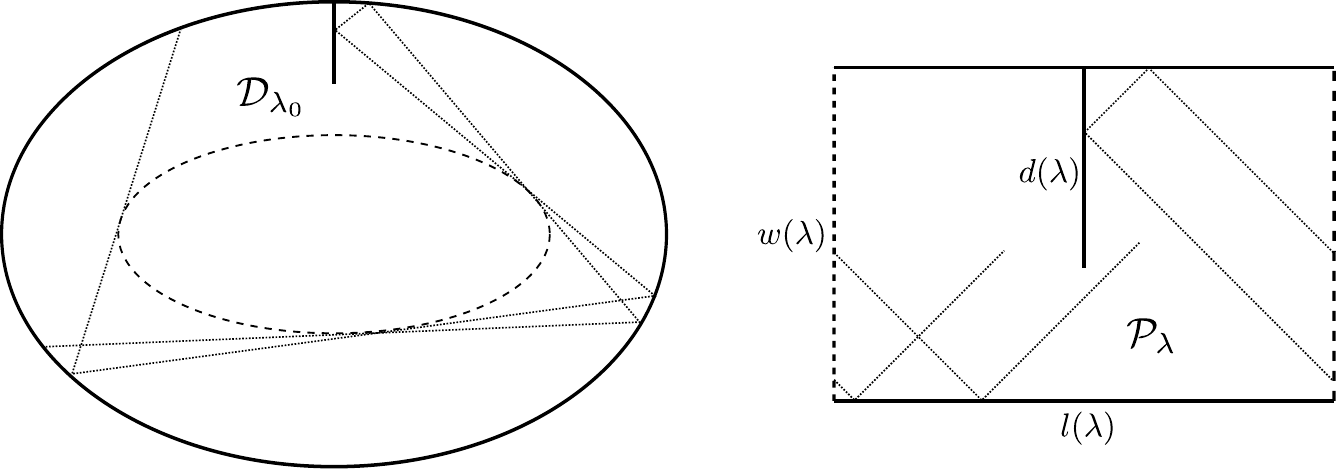}
\caption{}\label{EllipticBilliard}
\end{figure}
\begin{figure}[h]
\includegraphics[width=0.8\textwidth]{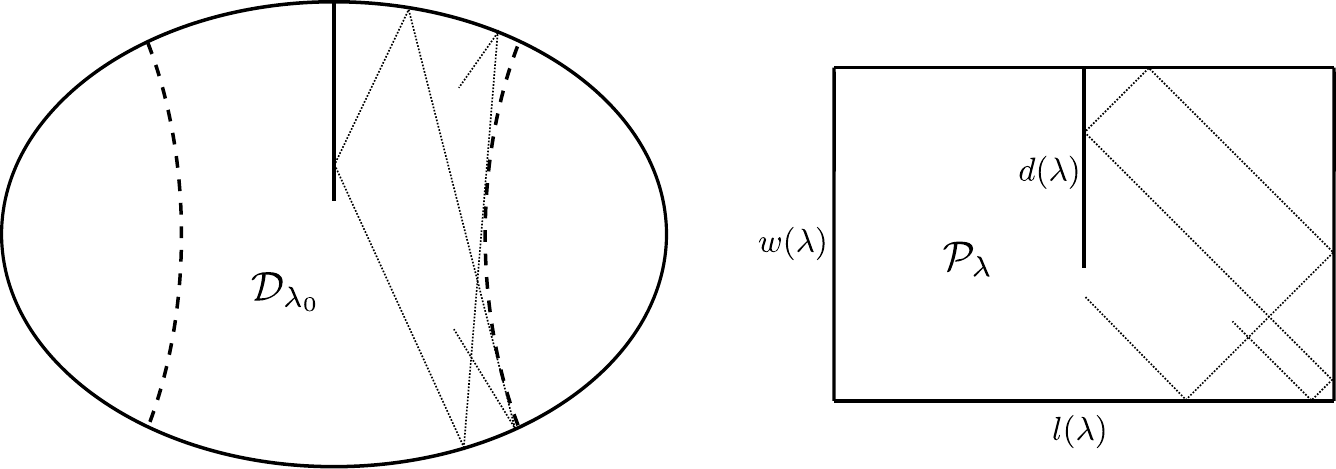}
\caption{}\label{EllipticBilliardHyp}
\end{figure}

\begin{prop}[c.f.\ \cite{Dra-Ra}]\label{prop:DraRa_reduction}
The billiard flow in $\mathcal{D}_{\lambda_0}$ restricted to the invariant region  $\mathcal{S}_\lambda$  is  isomorphic to the  a billiard flow in
directions $\pm\pi/4$, $\pm3\pi/4$ inside a polygonal billiard $\mathcal{P}_\lambda$ as follows:

\begin{itemize}
\item[{\bf (E)}] If $\lambda_0<\lambda< b$ then $\mathcal{P}_\lambda$ is the cylinder shown in Figure~\ref{EllipticBilliard} with length
\[l(\lambda)=4\int_b^a\frac{ds}{\sqrt{(a-s)(b-s)(\lambda-s)}}=4\int_{-\infty}^\lambda\frac{ds}{\sqrt{(a-s)(b-s)(\lambda-s)}}\]
and width
\begin{align*}
w(\lambda)&=\int_{0}^{\lambda}\frac{ds}{\sqrt{(a-s)(b-s)(\lambda-s)}}\\
&=\int_b^a\frac{ds}{\sqrt{(a-s)(b-s)(\lambda-s)}}-\int_{-\infty}^0\frac{ds}{\sqrt{(a-s)(b-s)(\lambda-s)}}
\end{align*}
and a linear vertical obstacle of length
\[d(\lambda)=\int_{0}^{\lambda_0}\frac{ds}{\sqrt{(a-s)(b-s)(\lambda-s)}}.\]

\item[{\bf (E')}] If $0<\lambda\leq\lambda_0$ then $\mathcal{P}_\lambda$ is a rectangle such that its length $l(\lambda)$ and width $w(\lambda)$ are the same as in the previous case.

\item[{\bf (H)}] If $b<\lambda<a$ then $\mathcal{P}_\lambda$ is the rectangle in Figure~\ref{EllipticBilliardHyp}  with length
\[l(\lambda)=2\int_\lambda^a\frac{ds}{\sqrt{(a-s)(b-s)(\lambda-s)}}=2\int_{-\infty}^b\frac{ds}{\sqrt{(a-s)(b-s)(\lambda-s)}}\]
and width
\[w(\lambda)=2\int_{0}^{b}\frac{ds}{\sqrt{(a-s)(b-s)(\lambda-s)}}\]
and a  linear vertical obstacle of length
\[d(\lambda)=\int_{0}^{\lambda_0}\frac{ds}{\sqrt{(a-s)(b-s)(\lambda-s)}}.\]
\end{itemize}
\end{prop}

The billiards on $\mathcal{P}_\lambda$, $\lambda_0<\lambda< b$,
are \emph{rational polygonal billiards}, i.e.\ the order of the
group generated by reflections at the polygon sides is finite (in
this case it is four).  A standard procedure, known as
\emph{unfolding} or \emph{Katok-Zemlyakov construction} (described
in~\cite{Fox-Ker} and \cite{Ka-Ze}), allows one to reduce a billiard in
a \emph{rational} polygon to a linear flow on a translation
surface.  In the case of the billiards $\mathcal{P}_\lambda$  in
Proposition~\ref{prop:DraRa_reduction}, the unfolding procedure
consists in taking four copies of the billiard, one for each
element of the group of reflections and gluing them. The resulting
surface are shown in Figure~\ref{SurfEandH}.

\begin{figure}[h]
\includegraphics[width=0.8\textwidth]{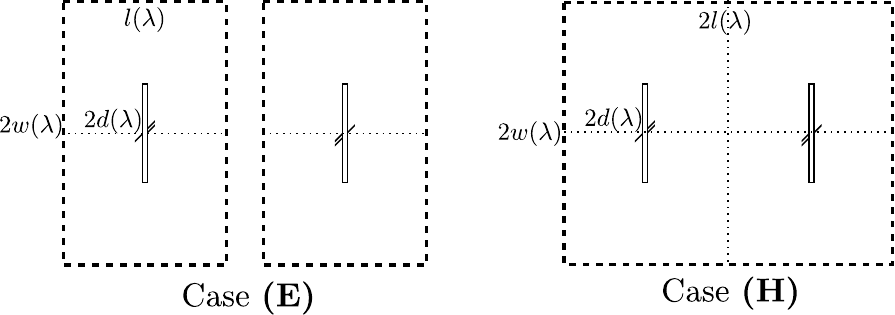}
\caption{The surface $({M}_\lambda,
\omega_\lambda)$}\label{SurfEandH}
\end{figure}

Thus, the unfolding procedure reduces the family of billiard flows in Proposition~\ref{prop:DraRa_reduction} to a family of
directional flows in a \emph{fixed} direction $\pi/4$ on   the
translation surfaces $({M}_\lambda, \omega_\lambda)\in
\mathcal{M}^{dc}$, in Figure~\ref{SurfEandH}, parameterized by
$\lambda_0<\lambda< b$. The case {\bf(E')}, where
$0<\lambda<\lambda_0$, should be treated separately, then
$(M_\lambda,\omega_\lambda)$ is a torus. For $\lambda_0<\lambda<
b$ we will assume that $(M_\lambda,\omega_\lambda)$ is rescaled so
that its area is two, and then rotated by $\pi/4$ so that the
linear flow in direction $\pi/4$ is mapped to the vertical linear
flow and still use the same notation $({M}_\lambda,
\omega_\lambda)$. Then these surfaces all belong to the moduli
space $\mathcal{M}^{dc}_2$ of double covers of flat tori,
described in~\ref{sec:Birkhoff_tori}.


Let us remark that given a \emph{fixed} translation surface  $(M,
\omega)$, for almost every direction $\theta \in [0, 2\pi]$ the linear
flow in direction $\theta$ is known to be uniquely ergodic by a
celebrated result by  Kerckhoff, Masur and Smillie
\cite{Ke-Ma-Sm} (or, in virtue of the recent result by Chaika and
Eskin, by Birkhoff genericity for a.e. point of the curve
$r_\theta (M, \omega)$ and Masur's ergodicity criterion, see
\S~\ref{sec:translation_surf}). Proving unique ergodicity of the
elliptical billiard flow on the invariant region
$\mathcal{S}_{\lambda}$ for almost every $\lambda_0<\lambda< b$,
on the other hand, is equivalent to proving unique ergodicity of the vertical flow on $({M}_\lambda, \omega_\lambda)$. By
 Masur's criterion it suffices to prove  Birkhoff genericity for  almost every point
on the curve  $\gamma:= \{({M}_\lambda, \omega_\lambda), \lambda \in
(\lambda_0, b) \}$ (see also Remark~\ref{rk:Birkhoff_vs_recurrence}). This curve $\gamma$ is explicitly described  in terms
of the coordinates given by the identification of
$\mathcal{M}^{cd}_2$ with the subset ${X}_2$ of the
homogeneous space $G/\Gamma_2$ in Lemma \ref{lemma:correspondence} (see \S \ref{sec:ellbillproof}). In \S~\ref{sec:ellbillproof} we verify that $\gamma$ satisfies the non-degeneracy assumptions of Corollary~\ref{cor:unique_ergodicity}, to conclude that almost every point is Birkhoff generic.

\begin{rem}\label{rk:Birkhoff_vs_recurrence}
In order to apply Masur's criterion, it is in principle enough to show that almost every point on the curve $\gamma$ is \emph{recurrent} under the Teichm\"uller flow (see Theorem~\ref{thm:Masur}). Let us remark though that recurrence under the diagonal flow in the homogeneous space $G/\Gamma_2$ (which is well-known) is \emph{not sufficient}. Indeed,  from Lemma \ref{lemma:correspondence} we know that the space $\mathcal{M}^{dc}_2$ is  isomorphic to a \emph{subset} $X_2\subset G/\Gamma_2$ obtained by \emph{removing} $2$  closed $SL_2(\R)$-{orbits}
which correspond to the locus of flat tori where the marked point coincide with the origin of the lattice. One can hence construct a recurrent diagonal orbit in $ G/\Gamma_2$  which approaches the removed set $H$ and lift it to a Teichm\"uller geodesics in $\mathcal{M}^{dc}_2$ which is \emph{divergent} (since the separating slit endpoints collide towards a pinched surface).

On the other hand, to get recurrence in $\mathcal{M}^{dc}_2$  it is sufficient to prove \emph{density}  in the homogeneous setup (i.e.\ under the diagonal flow $a_t$ for almost every point on the curve $\Psi \gamma$ in $ G/\Gamma_2$, the image of  $\gamma$ by the isomorphism $\Psi$ coming from Lemma~\ref{lemma:correspondence}). Density could in principle be deduced by equidistribution of the normalized volume measure on  $a_t\Psi \gamma$ as $t\to \infty$ using similar arguments as in  Elkies-McMullen \cite{em04},
thus yielding a different strategy of proof.
 We here prove and exploit the stronger conclusion of almost everywhere Birkhoff genericity on $\gamma$, since the criterium used to prove it  (Theorem~\ref{thm:BirOs}) is also needed  for the Eaton application in \S \ref{sec:Eaton_proof}, for which Birkhoff genericity is needed in its full strength.

\end{rem}


\subsection{Non-degeneracy of the curves in the space of double torus
covers.}\label{sec:ellbillproof} In this section we conclude the
proof of Theorem~\ref{thm:ellbill}. Let us consider the curve
\[
(\lambda_0,b)\cup (b,a)\ni
\lambda \mapsto (M_\lambda, \omega_\lambda)\in \mathcal{M}^{dc}_2
\]
\begin{sloppypar}
described in \S~\ref{sec:reduction_ellipses}. In view of
Corollary~\ref{cor:unique_ergodicity}, it suffices to check the
non-degeneracy assumptions for the curve
$\lambda\mapsto \Psi({M}_\lambda,
\omega_\lambda)$, where $\Psi:\mathcal{M}^{dc}_2\to X_2\subset G/\Gamma_2$
is the correspondence described in Lemma \ref{lemma:correspondence}.
Since
$(M_\lambda,\omega_\lambda)$ is a the surface after rotation by
$\pi/4$ one of the surfaces shown in Figure~\ref{SurfEandH}, its
image
is of the form $\psi(\lambda)\Gamma_2$, where
\end{sloppypar}
\begin{align*}
\psi(\lambda)=\left(\begin{pmatrix}
  \frac{l(\lambda)}{r(\lambda)} & -2\frac{w(\lambda)}{r(\lambda)} \\
  \frac{l(\lambda)}{r(\lambda)} & 2\frac{w(\lambda)}{r(\lambda)}
\end{pmatrix},
\begin{pmatrix}
  -2\frac{d(\lambda)}{r(\lambda)} \\
  2\frac{d(\lambda)}{r(\lambda)}
\end{pmatrix}
\right)
\end{align*}
 with
$r(\lambda)=2\sqrt{l(\lambda)w(\lambda)}$.

We will  show that for
 $\lambda\in(\lambda_0,b)\cup (b,a)$ the  determinant of  $M_{\psi}(\lambda)$ (see (\ref{eq;M_psi})) is  non-zero. The
following result will help us to verify this requirement.

\begin{lem}\label{lem:nonzerodet}
Let $f:(c_1,c_2)\cup(c_3,c_4)\to\R$ ($-\infty\le c_1<c_2<c_3<c_4\le \infty$) be a
positive  continuous function such that
$
\int_{c_1}^{c_2}f(s)\dd s$
and
 $ \int_{c_3}^{c_4}f(s)\dd s
$ are finite.
  Suppose that $\{A_i:1\leq i\leq
k\}$ is a family of pairwise disjoint subintervals of
$(c_1,c_2)\cup(c_3,c_4)$. Then for every $\lambda\in(c_2,c_3)$ we
have
\begin{equation}\label{ineq:int}
\det\left[\int_{A_{i}}\frac{f(s)\,ds}{(\lambda-s)^{j-1}}\right]_{i,j=1,\ldots,k}\neq 0.
\end{equation}
\end{lem}

\begin{proof}
By the Vandermonde determinant formula, we have
\begin{align*}
\det&\left[\int_{A_{i}}\frac{f(s)\dd s}{(\lambda-s)^{j-1}}\right]_{1\leq i,j\leq k}=
\int_{\prod_{i=1}^kA_{i}}\det\left[\frac{g(s_i)}{(\lambda-s_i)^{j-1}}\right]_{1\leq i,j\leq k}
\dd s_1\ldots \dd s_k\\
&=\int_{\prod_{i=1}^kA_{i}}\prod_{i=1}^k g(s_i)\prod_{1\leq j<i\leq k}\Big(\frac{1}{\lambda-s_i}-\frac{1}{\lambda-s_{j}}\Big)\dd s_1\ldots \dd s_k.
\end{align*}
Since  for $j<i$ the intervals $A_{i}$, $A_{i'}$ are disjoint and do not contain $\lambda$,  the map
\[
A_{i}\times A_{j}\to \R \qquad
(s_i,s_{j})\mapsto\frac{1}{(\lambda-s_i)}-\frac{1}{(\lambda-s_{j})}
\]
is of constant sign. Therefore, the integrated function is of constant sign as well (as the product of such functions).
This gives the non-vanishing of the
integral.
\end{proof}

\begin{rem}\label{rem:Ml}
\begin{sloppypar}
Suppose that $a,b,c,r:I\to\R$ are  $C^2$-functions such that $r$ takes only non-zero values. Then $M_{ra,rb,rc}(s)=r(s)U(s)M_{a,b,c}(s)$, where $U(s)$ is a
lower unitriangular matrix. It follows that non-vanishing of the determinant of $M_{ra,rb,rc}(s)$ is inherited by the matrix $M_{a,b,c}(s)$. This
observation  will be used in the following lemma.
\end{sloppypar}
\end{rem}

\begin{lem}\label{lem:concave}
The determinant of  $M_{\psi}(\lambda)$  is
non-zero for every  $\lambda\in(\lambda_0,b)\cup(b,a)$.
\end{lem}

\begin{proof}

{\bf Case (E):} $\lambda\in(\lambda_0,b)$. Let us consider the
$C^{\infty}$-function
\[e:((-\infty,\lambda_0)\cup(b,a))\times(\lambda_0,b)\to\R_+,\quad e(s,\lambda)=\frac{1}{\sqrt{(a-s)(b-s)(\lambda-s)}}.\]
Then
\begin{align*}
l(\lambda)&=4\int_b^ae(s,\lambda)\,ds,\quad \quad d(\lambda)=\int_0^{\lambda_0}e(s,\lambda)\,ds,\\
w(\lambda)&=\frac{l(\lambda)}{4}-\widetilde{w}(\lambda)\quad\text{with}\quad \widetilde{w}(\lambda):=\int_{-\infty}^0e(s,\lambda)\,ds
\end{align*}
and
\[\frac{\partial e}{\partial \lambda}(s,\lambda)=-\frac{1}{2}\frac{e(s,\lambda)}{\lambda-s}, \quad \frac{\partial^2 e}{\partial \lambda^2}(s,\lambda)=\frac{3}{4}\frac{e(s,\lambda)}{(\lambda-s)^2}.\]
Hence
\begin{align*}
l'(\lambda)=-2\int_b^a\frac{e(s,\lambda)\,ds}{\lambda-s}, & \qquad l''(\lambda)=3\int_b^a\frac{e(s,\lambda)\,ds}{(\lambda-s)^2},\\
\widetilde{w}'(\lambda)=-\frac{1}{2}\int_{-\infty}^0\frac{e(s,\lambda)\,ds}{\lambda-s}, & \qquad
\widetilde{w}''(\lambda)=\frac{3}{4}\int_{-\infty}^0\frac{e(s,\lambda)\,ds}{(\lambda-s)^2},\\
d'(\lambda)=-\frac{1}{2}\int_0^{\lambda_0}\frac{e(s,\lambda)\,ds}{\lambda-s}, & \qquad
d''(\lambda)=\frac{3}{4}\int_0^{\lambda_0}\frac{e(s,\lambda)\,ds}{(\lambda-s)^2}.
\end{align*}
In view of Remark~\ref{rem:Ml}, we can consider the matrix
$M_{l,\widetilde{w},d}(\lambda)$ instead of $M_{\psi}(\lambda)$.
Let $A_1=(a,b)$, $A_2=(-\infty,0)$, $A_3=(0,\lambda_0)$. Since
they are pairwise disjoint, by Lemma~\ref{lem:nonzerodet}, we have
\begin{align*}
\det M_{l,\widetilde{w},d}(\lambda)&=-\frac{3}{2}\det\left[\int_{A_i}\frac{e(s,\lambda)\,ds}{(\lambda-s)^{j-1}}\right]_{i,j=1,2,3}\neq 0.
\end{align*}
Since $r(\lambda)>0$, this completes the proof of the part {\bf(E)}.

{\bf Case (H):} $\lambda\in(b,a)$. Here we deal with
$e:(-\infty,b)\times(b,a)\to\R$. Then
\[l(\lambda)=2\int_{-\infty}^be(s,\lambda)\,ds,\quad w(\lambda)=2\int_0^be(s,\lambda)\,ds,\quad d(\lambda)=\int_0^{\lambda_0}e(s,\lambda)\,ds\]
and
\begin{align*}
l'(\lambda)=-\int_{-\infty}^b\frac{e(s,\lambda)\,ds}{\lambda-s}, & \qquad l''(\lambda)=\frac{3}{2}\int_{-\infty}^b\frac{e(s,\lambda)\,ds}{(\lambda-s)^2}\\
w'(\lambda)=-\int_0^b\frac{e(s,\lambda)\,ds}{\lambda-s}, & \qquad
w''(\lambda)=\frac{3}{2}\int_0^b\frac{e(s,\lambda)\,ds}{(\lambda-s)^2}\\
d'(\lambda)=-\frac{1}{2}\int_0^{\lambda_0}\frac{e(s,\lambda)\,ds}{\lambda-s}, & \qquad
d''(\lambda)=\frac{3}{4}\int_0^{\lambda_0}\frac{e(s,\lambda)\,ds}{(\lambda-s)^2}.
\end{align*}
Again, in view of Remark~\ref{rem:Ml}, we deal with the matrix
$M_{l,w,d}(\lambda)$  instead of $M_{\psi}(\lambda)$.
Let
\[A_1:=(-\infty,0),\ A_2:=(\lambda_0,b),\ A_3:=(0,\lambda_0).\]
Then
\[B_1:=(-\infty,b)=A_1\cup A_2\cup A_3,\ B_2:=(0,b)=A_2\cup A_3,\ B_3:=(0,\lambda_0)=A_3.\]
Since $A_1,A_2,A_3$ are pairwise disjoint, by Lemma~\ref{lem:nonzerodet}, we have
\begin{align*}
\det M_{l,w,d}(\lambda)&=-\frac{3}{2}\det\left[\int_{B_{i}}\frac{e(s,\lambda)\,ds}{(\lambda-s)^{j-1}}\right]_{i,j=1,2,3}\\
&=-\frac{3}{2}\det\left[\int_{A_{i}}\frac{e(s,\lambda)\,ds}{(\lambda-s)^{j-1}}\right]_{i,j=1,2,3}\neq 0.
\end{align*}
Since $r(\lambda)>0$, this completes the proof.
\end{proof}

\begin{proof}[Proof of Theorem~\ref{thm:ellbill}]
As $\Psi(M_\lambda,\omega_\lambda)=\psi(\lambda)\Gamma_2$ for
$\lambda\in(\lambda_0,b)\cup(b,a)$, in view of
Corollary~\ref{cor:unique_ergodicity}, for a.e.\
$\lambda\in(\lambda_0,a)$ the vertical flow on
$(M_\lambda,\omega_\lambda)\in\mathcal{M}^{dc}_2$ is uniquely
ergodic. Moreover, the billiard flow on $\mathcal{S}_\lambda$ is
isomorphic (up to a linear rescaling of time) to the vertical flow
on $(M_\lambda,\omega_\lambda)$ for every
$\lambda\in(\lambda_0,b)\cup(b,a)$. This completes the proof for
$\lambda\in (\lambda_0,a)$.

If $\lambda\in(0,\lambda_0)$ then, by
Proposition~\ref{prop:DraRa_reduction},  the billiard flow on
$\mathcal{S}_\lambda$ is isomorphic to the directional flow in
direction $\pi/4$ on the torus $\R^2/((2l(\lambda)\Z)\times
(2w(\lambda)\Z))$. This flow is uniquely ergodic if and only if
$w(\lambda)/l(\lambda)$ is irrational. The same argument as in the
proof of the case {\bf(E)} in Lemma~\ref{lem:concave} shows that
$w'(\lambda)l(\lambda)-w(\lambda)l'(\lambda)\neq 0$ also for all
$\lambda\in(0,\lambda_0)$. Therefore, the map
$\lambda\mapsto\frac{w(\lambda)}{l(\lambda)}$ is strictly
monotonic. It follows that $w(\lambda)/l(\lambda)$ is irrational
for all but countably many parameters $\lambda\in(0,\lambda_0)$,
which completes the proof.
\end{proof}

\section{The beaviour of light rays in Eaton lenses systems}\label{sec:Eaton_proof}
The application to the behaviour of light rays in periodic arrays
of Eaton lenses exploits both the result on Birkhoff and Oseledets
genericity (i.e.
Theorem~\ref{thm:BirOs}). The proof follows the arguments
developed in the work by Fr\k{a}czek and Schmoll \cite{Fr-Sch},
which reduces the behaviour of rays (and specifically being
trapped in a band) to a result on existence of Lyapunov exponents.
Since Fr\k{a}czek and Schmoll  in \cite{Fr-Sch} were relying on
the standard  formulation of Oseledets genericity, they could only
obtained a weaker result on random lattice configuration. The
genericity results along curves that we prove in this paper, on
the other hand, allows us to analyse the behaviour of any given
admissible lattice configuration of lenses in almost every
direction.

\subsection{Reduction to systems of flat lenses and periodic translation surfaces}
\label{sec;Eaton reduction}
As in \cite{Fr-Sch} we pass to a simpler model in which round
lenses are replaced by their flat counterparts. By a flat lens of
radius $R>0$ perpendicular to the direction $\theta$ we mean
simply any interval in $\R^2$ of length $2R$ perpendicular to
vectors in direction $\theta$. The light rays pass through the lens as follows: any
light ray in direction $\theta$ or $\theta+\pi$ runs until hitting
the flat lens and then is rotated by $\pi$ around the center of
the flat lens and runs  in the opposite direction, see
Figures~\ref{flat_lens}.
\begin{figure}[h]
\includegraphics[width=0.9\textwidth]{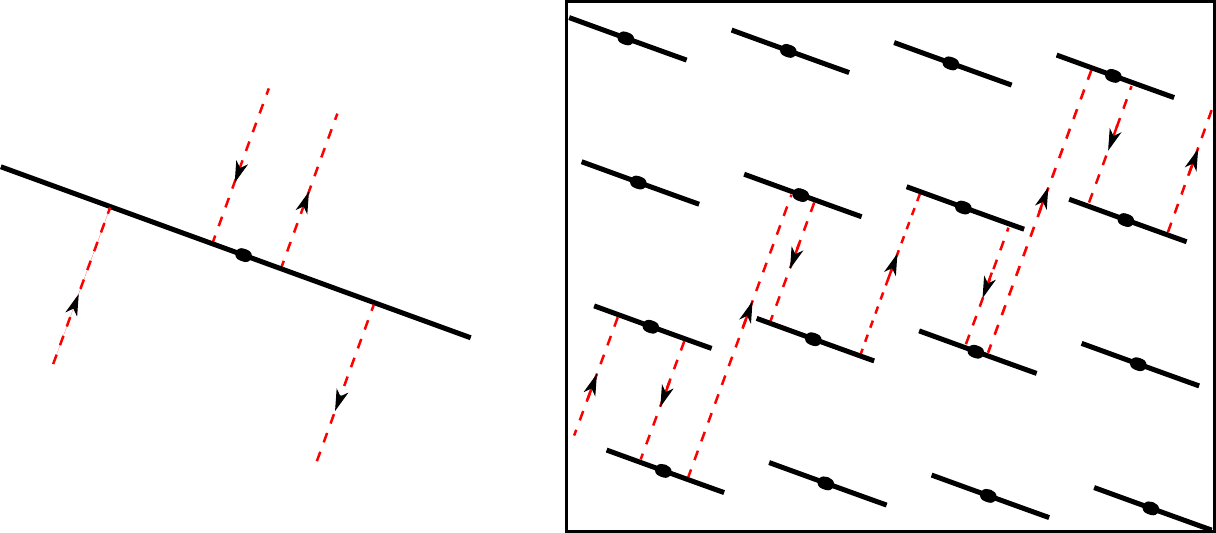}
\caption{The system of lenses
$F(\Lambda,R,\theta)$}
\label{flat_lens}
\end{figure}
The system of flat lenses of radius $R$ perpendicular to the direction  $\theta$
whose centers are arranged at the points of the lattice $\Lambda$
will be denoted by $F(\Lambda,R,\theta)$, see
Figures~\ref{flat_lens}. A simple observation (see \cite{Fr-Sch}
for details) shows that if a direction $\theta$ is trapped on
$F(\Lambda,R,\theta)$ then it is also trapped on $L(\Lambda,R)$.
Moreover, after rotation by $\pi/2-\theta$ the system of flat
lenses we can pass to vertically directed light rays.
We take $v\in S^1$ to be the  vertical unit
vector. Denote by
$r_\theta:\R^2\to\R^2$ the rotation (also the matrix of the
rotation $r_\theta\in SL_2(\R)$) by $\theta$ around the center of
$\R^2$. Then
$r_{\pi/2-\theta}F(\Lambda,R,\theta)=F(r_{\pi/2-\theta}\Lambda,R,v)$.
In summary, Theorem~\ref{thm:main} reduces to the following
result.
\begin{thm}\label{thm:maux}
Let $(\Lambda,R)$ be an admissible pair. Then for a.e.\ $\theta\in [0,2\pi]$ the vertical direction is trapped on $F(r_\theta\Lambda,R,v)$.
\end{thm}
In \cite{Fr-Sch} (see Theorem 1.2), a weaker version of this result was proved
saying that for every $R>0$ and a.e.\ unimodular lattice $\Lambda$ such that
$(\Lambda,R)$ is admissible the vertical direction is trapped on $F(\Lambda,R,v)$.
Using a simple Fubini argument
we also have the following seemingly stronger result which closely related to Theorem~\ref{thm:main}.
\begin{cor}\label{cor:eatonprel}
For every unimodular lattice $\Lambda$, a.e.\ $0<R<s(\Lambda)/2$
and a.e.\ $\theta\in [0, 2\pi]$ the vertical direction on $F(r_\theta\Lambda,R,v)$ is trapped.
In particular, for every unimodular lattice $\Lambda$ and a.e.\ $0<R<s(\Lambda)/2$
almost every direction $\theta\in [0, 2\pi]$ on $L(\Lambda,R)$ is trapped.
\end{cor}
\begin{proof}
Fix a unimodular lattice $\Lambda_0$ and
$0<R_0<1/\sqrt{2\sqrt{3}}$. Suppose contrary to our claim that
there exists a set $A\subset (0,s(\Lambda_0)/2)\times [0, 2\pi]$ not of
zero Lebesgue measure such that for every $(R,\theta)\in A$ the
vertical direction is not trapped on $F(r_\theta\Lambda_0,R,v)$.
For every $t\in R$  let $h(t):=\left(\begin{array}{cc}
         1 & 0 \\
         t & 1
       \end{array}\right)$.
Then for every admissible $(\Lambda, R)$
\[v\text{ is trapped on }F(\Lambda,R,v) \Leftrightarrow v\text{ is trapped on }a_sF(\Lambda,R,v)=F(a_s\Lambda,e^sR,v)\]
for all $s\in\R$ and
\[v\text{ is trapped on }F(\Lambda,R,v) \Leftrightarrow v\text{ is trapped on }F(h(t)\Lambda,R,v)\]
for all $t\in(-\vep,\vep)$ for some $\vep>0$. The two equivalences
essentially follows from the fact that both operations does not
change horizontal positions between flat lenses (up to rescaling
in the first case). Therefore, $(R,\theta)\in A$ implies that $v$
is not trapped on $F(h(t)a_{\log(R_0/R)}r_\theta\Lambda_0,R_0,v)$
for all $t\in(-\vep,\vep)$. Since the map
\[(t,s,\theta)\mapsto h(t)a_{\log(R_0/s)}r_\theta\Lambda_0\]
locally is a diffeomorphism between $\R\times \R\times [0, 2\pi]$ and
the space of unimodular lattices $SL_2(\R)/SL_2(\Z)$, the fact
that $A\subset \R^2$ is not of measure zero together with a Fubini
argument yield a set $\mathcal{A}\subset SL_2(\R)/SL_2(\Z)$ not of
measure zero such that $\Lambda\in\mathcal{A}$ implies that $v$ is
not trapped on $F(\Lambda,R_0,v)$. This contradicts Theorem 1.2 in
\cite{Fr-Sch}. The second part of the corollary follows
immediately from the first.
\end{proof}

\subsection{Light behaviour vs Oseledets genericity}\label{sec:trapped_vs_genericity}
In this section we relate the trapping phenomenon for light rays
in  almost all directions with Oseledets genericity along a curve of
translation surfaces. We rely for this section on  some basic
steps of the proof of Theorem~\ref{thm:maux} which were developed
in \cite{Fr-Sch} and a technical result recently proved in
\cite{Fr-Hu}.


Suppose that the pair $(\Lambda,R)$ is admissible ($\Lambda$ is
unimodular) and let us consider  flat lenses system
$F(\Lambda,R,v)$. By an unfolding procedure similar to the one
used in \S~\ref{sec:ellipses_proof}, the system of flat lenses
$F(\Lambda,R,v)$ can be reduced to a noncompact periodic
translation surface, which is a $\mathbb{Z}^2$-cover of a compact
translation surface $M(\Lambda,R)\in\mathcal{M}^{dc}_2$ obtained
as follows (we refer the reader to \cite{Fr-Sch} for further
details).
Take the
translation torus $\T^2_\Lambda:=\R^2/\Lambda$ with a horizontal
interval (slit) $I\subset\R^2/\Lambda$ of length $2R$. Since
$(\Lambda,R)$ is admissible, $I$ has no self-intersections. Next
take two copies of such slitted torus and glue them together so
that the bottom part of the slit on one torus is glued by
translation to its top counterpart on the other torus, see
Figure~\ref{surfcomlens}.
\begin{figure}[h]
\includegraphics[width=0.5\textwidth]{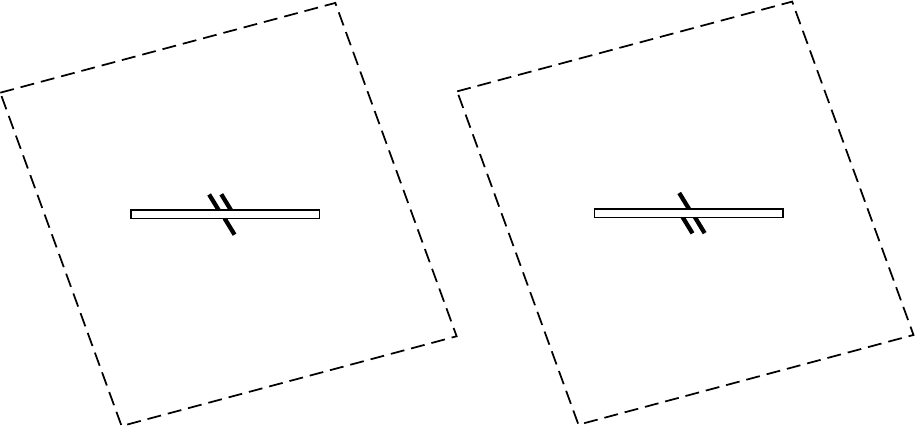}
\caption{The surface $M(\Lambda,R)$} \label{surfcomlens}
\end{figure}
Then
\begin{equation}\label{form:curv}
\Psi(M(\Lambda,R))=(h,(2R,0))\Gamma_2,\quad\text{where}\quad\Lambda=h(\Z^2).
\end{equation}

It turns out that the behaviour of the vertical trajectories in $F(\Lambda,R,v)$ (and hence in $L(\Lambda,R)$) can be described by studying the asymptotic
behaviour of  the following homology classes. For every translation surface $(M,\omega)$ and $x\in M$ such that its positive vertical semi-orbit
$\{\psi^{v}_tx:t\geq 0\}$ is well defined and for every $t>0$ denote by $\sigma_t(x)\in H_1(M,\Z)$ the homology class of the curve formed by the segment of
the vertical orbit starting from $x$ until time $t$ closed up by the shortest curve returning to $x$. Then we have the following.

\begin{prop}[Theorem 3.2 in \cite{Fr-Sch}]\label{prop:existband}
Let $\tau:M(\Lambda,R)\to M(\Lambda,R)$ be the only nontrivial
element of  the deck transformation group of the double cover.
Suppose that there is a non-zero homology class $\zeta\in
H_1(M(\Lambda,R),\R)$  and $C>0$ such that
\[\tau_*\zeta=-\zeta\text{ and }|\langle\sigma_t(x),\zeta\rangle|\leq C\text{ for all }x, t\text{ for which }\sigma_t(x)\text{ is defined.}\]
If the surface $M(\Lambda,R)$ has no vertical saddle connection,
i.e.\  there is no  vertical orbit segment that connect singular
points, then the vertical direction on $F(\Lambda,R,v)$ is
trapped.
\end{prop}

Oseledets genericity plays a central role in verifying this type of assumptions in Proposition~\ref{prop:existband}.
The technical tool which allows us to check the assumptions  in our specific context will be provided by the following general proposition recently proved by Fr\k{a}czek and Hubert in \cite{Fr-Hu}. We should mention that the idea behind this result is not new (can be found also in \cite{Fr-Ulc:nonerg} and \cite{DHL}) and exploits the phenomenon of bounded deviation discovered by Zorich in \cite{Zor1,Zor2}.

\begin{prop}[\cite{Fr-Hu}]\label{prop:genbound}
Let $\mathcal M\subset \mathcal{M}_1(M,\Sigma)$ be an $SL_2(\R)$-orbit closure.  Suppose that $W\subset H_1(M,\R)$ is a symplectic subspace which is
$SL_2(\R)$-invariant over $\mathcal M$. Suppose that the top Lyapunov exponent of the restricted cocycle $A_W^{KZ}:\R\times\mathcal{M}\to GL(W)$ is
positive. If $(M,\omega)$ is Birkhoff and Oseledets generic with respect to $(\mathcal M,\mu_{\mathcal M}, a_t)$ then for every $\zeta\in E_{\omega}^-$
where
\[
E_{\omega}^- =\Big\{v\in W:\lim_{t\to\infty}\frac{\log\|A^{KZ}_W( t,\omega)v\|}{t}<0\Big\},
\]
there exists $C>0$ such that
\[|\langle\sigma_t(x),\zeta\rangle|\leq C\text{ for all }x\in M, t>0\text{ for which }\sigma_t(x)\text{ is defined}.\]
\end{prop}

\begin{proof}[Proof of Theorems~\ref{thm:maux}~and~\ref{thm:main}]
Let $(\Lambda,R)$ be a admissible pair and let $\Lambda=h(\Z^2)$
for some $h\in SL_2(\R)$. As we have already mentioned in the
beginning of \S~\ref{sec;Eaton reduction}  that Theorem~\ref{thm:main}
reduces to Theorem~\ref{thm:maux}, so we need to prove that for
a.e.\ $\theta\in [0,2\pi]$ the vertical direction is trapped
in $F(r_\theta\Lambda,R,v)$.

Let us consider the restricted Kontsevich-Zorich cocycle $A^{KZ}_W:\R\times\mathcal{M}^{dc}_2\to GL(W)$, where
\[W:=\{\zeta\in H_1(M,\R):\tau_*\zeta=-\zeta\}\]
and $\tau:M\to M$ is  the only nontrivial element of the deck transformation group of the double cover. The subspace $W$ is two dimensional symplectic and
$SL_2(\R)$-invariant.
Let us consider the curve $[0,2\pi]\ni\theta\mapsto M(r_\theta\Lambda,R)\in\mathcal{M}^{dc}_2$. In view of \eqref{form:curv},
\[
\Psi(M(r_\theta\Lambda,R))=(r_\theta h,(2R,0))\Gamma_2.
\]
 Taking $\psi(\theta)=(r_\theta,(2R,0))$, we have that the determinant of
\[M_\psi(\theta)=\begin{bmatrix}
  \cos\theta & -\sin\theta & 2R \\
  -\sin\theta & \cos\theta & 0 \\
  \cos\theta & \sin\theta & 0
\end{bmatrix}\]
is non-zero. Therefore, by Theorem~\ref{thm:BirOs}, for a.e.\
$\theta\in [0,2\pi]$ the surface $M(r_\theta\Lambda,R)$ is
Birkhoff and Oseledets generic. Fix such $\theta\in
[0,2\pi]$. Recall that all Lyapunov exponents for all
invariant measures supported on the space of genus two surfaces
are positive, see e.g.\ \cite{Bain}. It follows that the top
Lyapunov exponent of the reduced cocycle $A^{KZ}_W$ is positive
(in fact, is equal to $1/2$). In view of
Proposition~\ref{prop:genbound},
it follows that there exists
non-zero $\zeta\in W$ and $C>0$ such that
$|\langle\sigma_t(x),\zeta\rangle|\leq C$ for all $x\in
M(r_\theta\Lambda,R)$, $t>0$  for which $\sigma_t(x)$ is defined.

\begin{sloppypar}
By Birkhoff genericity, the vertical flow on
$M(r_\theta\Lambda,R)$ has no saddle connection. Therefore, by
Proposition~\ref{prop:existband}, the vertical direction is
trapped on $F(r_\theta\Lambda,R,v)$. The result for Eaton model follows
from the correspondence between trapped direction in
$F(\Lambda,R,v)$ and $L(\Lambda,R)$.
\end{sloppypar}
\end{proof}

\section{Gap distribution of square root of integers}\label{sec:gaps_proof}
In this section, we give the proof of
Theorem \ref{thm;g;gap}, based on our Birkhoff genericity  result for curves
(Theorem~\ref{thm;equi}). Our proof follows the same idea of the proof by Elkies and McMullen in  \cite{em04}.

Let us recall that a sequence of real numbers $\{t_n\}_{n\geq 1}$
\emph{converges in distribution} to a  probability measure $\mu$
on $\mathbb R$ if for every $ f\in C_c(\R)$ we have
\[\lim_{N\to\infty}\frac{1}{N}\sum_{n=1}^N f(t_n)=\int_\R f(s)\,d\mu(s).\]
Convergence  in distribution to $\mu$ is
equivalent to
\begin{equation}\label{convseq}
\lim_{N\to\infty}\frac{1}{N}\#\{1\leq n\leq N: t_n\in(-\infty,b]\}\to \mu((-\infty,b])
\end{equation}
for every  $b\in\R$ with $\mu(\{b\})=0$ . Let us first prove two simple
Lemmas.
\begin{lem}\label{lem;g;bound}
Let $\{t_n\}_{n\geq 1}, \{l_n\}_{n\geq 1}$ be two  sequences of positive  real numbers. Suppose that
$
\lim_{n\to \infty}\frac{t_n}{l_n}=1
$
and
$\{t_n\}_{n\geq 1}$ converges in distribution to a probability measure $\mu$, then so
does $\{l_n\}_{n\geq 1}$.
\end{lem}
\begin{proof}
It suffices to show that for every  $ f\in C_c(\mathbb R)$ one has
\begin{align}\label{eq;g;sequence}
\lim_{n\to \infty} | f(t_n)- f(l_n)|=0.
\end{align}
Given $\varepsilon>0$, there exists $\delta >0$ such that if $0\le
s, t\le \delta $, then $| f(s)- f(t)|<\varepsilon$.  According to the
assumption, there exits an integer $N_1>0$ such that for $n\ge N_1$
either $t_n, l_n\le \delta$ or $t_n, l_n\ge \frac{\delta }{2}$.
Since $ f$ has compact support, there exists $N_2>0$ such that if
$n\ge N_2$ and  $t_n, l_n\ge \frac{\delta}{2}$  then $| f(t_n)-
f(l_n)|< \varepsilon$. If follows that $ | f(t_n)-
f(l_n)|<\varepsilon$ provided that  $n\ge \max\{N_1, N_2\}$.
Therefore  (\ref{eq;g;sequence}) holds and the proof is complete.
\end{proof}

\begin{lem}\label{lem;g;squeeze}
Let  $\{t_n\}, \{s_n\}, \{l_n\}$ be
sequences of real numbers. Suppose that $s_n\le t_n\le l_n$ for every $n\ge 1$ and
$\{s_n\},\{l_n\}$ converge in distribution to a
 probability measure $\mu$ on $\mathbb R$, then so does
$\{t_n\}$.
\end{lem}
\begin{proof}
The lemma follows from the usual squeeze lemma for sequences of real numbers and (\ref{convseq}).
\end{proof}

Let us now  turn to  the proof of Theorem \ref{thm;g;gap} using  the strategy in  \cite{em04}.
Recall that $L_r(s)$ defined in (\ref{eq;L_r(s)}) is the normalized gap containing $s$ of
fractional parts of integers up to $r$.
We approximate
$L_r(s)$ for $s>0$ by the maximal area $L'_r(s)$ of $\triangle
ABC$ (see Figure \ref{fig;g;one}) with the following properties: $B,C$ move
on the line $y=\sqrt r-s$; $B, C$ are above and below the line
$x=2sy+s^2$ respectively; the interior of  $\triangle ABC$
contains no  points of $\mathbb Z^2$.
 We remark here that if the line segment of $x=2sy+s^2$ contained
in $\triangle ABC$ has a lattice point, then $L'_r(s)=0$. For an
affine lattice $\Lambda$ we let $f(\Lambda)$  be the maximal  area
of triangles with the following properties: the triangle's interior  contains the line segment $\{0\}\times [0,1]$ but no lattice points of $\Lambda$; the triangle has
 one
vertex $(0,0)$ and the other vertices lie on the line $y=1$.

Recall that  $u(\cdot, \cdot, \cdot), a_t$ are defined in (\ref{subgroups}) and $\mu_X$ is the probability
Haar measure on $X=ASL_2(\R)/ASL_2(\Z)$.
It is  noticed  in \cite{em04} that
\begin{align}\label{eq;L_r'(s)}
L_r'(s)=f(a_{\log\sqrt r}u(-2s,-s^2, s)\mathbb Z^2).
\end{align}
Let $F: [0, \infty)\to [0, \infty)$ be the piecewise
analytic density  function given by (\ref{eq;density F}).
It is proved in \cite[\S~3.3]{em04} that
\[
\int_0^l tF(t)\dd t=\mu_X(\{\Lambda\in X:f(\Lambda)\leq l\}).
 \]

\setlength{\unitlength}{0.7cm}

\begin{figure}[h]

\center{
\begin{picture}(9.5,8.5)(-1.5,-1.5)
\put(-1.2,0){\vector(1,0){9.5}}
\put(0,-1.2){\vector(0,1){7}}
\put(-0.36,-0.6){\line(6,5){8}}
\put(4.2,5.5){$x=2sy+s^2$}
\put(-0.36,-0.6){\line(4,5){4.5}}
\put(-0.36,-0.6){\line(3,2){8.4}}
\put(8.4,-.05){$x$}
\put(-0.1,6){$y$}
\put(-1,0){\circle*{0.08}}
\put(0,-1){\circle*{0.08}}
\put(-0.36,-0.6){\circle*{0.1}}
\put (-2.9,-0.75){$A(-s^2, -s)$}
\put (-1.3, 0.13){$-1$}
\put (0.15, -1.1){$-1$}
\put (-0.4, 0.13){$O$}
\put(0,0){\circle*{0.08}}
\put (0.9, -.4){$1$}
\put(1,0){\circle*{0.08}}
\put(0,5){\circle*{0.08}}
\put (-1.7, 4.9){$\sqrt r-s$}
\multiput(0,5)(0.4,0){23}{\line(1,0){0.2}}
\put(8.05,5){\circle*{0.1}}
\put(4.13,5){\circle*{0.1}}
\put(4.13,5){\line(1,0){3.92}}
\put(3.7,5.1){$B$}
\put(8.1,5.1){$C$}
\end{picture}
}
\caption{}
\label{fig;g;one}
\end{figure}


\begin{lem}\label{lem;g;approximation}
Let $\{t_n\}_{n\in \N}$ be a sequence of natural numbers  such that $\sum_{n=1}^\infty t_n^{-1}<\infty $. Then
for Lebesgue  almost every $s\in [0, 1]$ we have
\[
\lim_{n\to \infty}\frac{L_{t_n^2}(s)}{L'_{t_n^2}(s)}=1.
\]
\end{lem}
\begin{proof}
Let $a, A> 1$ be integers. According to  \cite[Lemmas 3.3 and 3.7]{em04} the set
\[
\left\{\frac{1}{a-1}\le s\le 1-\frac{1}{a-1}:\frac{2A+1}{2A+2}L'_{a^2}(s)\le L_{a^2}(s)\le
\frac{2A+1}{2A} L'_{a^2}(s)\right\}
\]
has Lebesgue measure
\[
\ge1- \frac{ (A+2)(A-1)+2}{a-1}.
\]
The conclusion follows from
  the assumption  $\sum_{n=1}^\infty t_n^{-1}<\infty$ and the  Borel-Cantelli lemma.

\end{proof}
\begin{proof}[Proof of Theorem \ref{thm;g;gap}]
For every $r\ge 1$ and $s\in [0, 1]$ it can be checked directly
that
\begin{align}\label{eq;g;app}
\frac{\lfloor r\rfloor}{(\lfloor \sqrt{r}\rfloor +1)^2}L_{(\lfloor\sqrt{r}\rfloor+1)^2}(s)\le
 L_r(s)\le
 \frac{\lfloor r\rfloor}{\lfloor \sqrt{r}\rfloor ^2} L_{\lfloor\sqrt{r}\rfloor^2}(s).
\end{align}
Note that  coefficients of left and right hand sides of (\ref{eq;g;app}) converge to $1$ as
$r\to \infty$.

Let $r_n:=cq^n$ and set
\begin{align*}
t_n(s):=L'_{(\lfloor\sqrt{r_n}\rfloor+1)^2}(s),\ l_n(s):=L'_{\lfloor\sqrt{r_n}\rfloor^2}(s).
\end{align*}
By Lemmas \ref{lem;g;bound},  \ref{lem;g;squeeze} and \ref{lem;g;approximation},
it suffices to prove that for Lebesgue almost every $s\in [0, 1]$ the sequences
$\{t_n(s)\}_{n\in \N}$ and $\{l_n(s)\}_{n\in \N}$
converge in distribution to $\mu=tF(t)\,dt$ on $[0, \infty)$.
By (\ref{eq;L_r'(s)}),
\begin{align*}
t_n(s)&=f(a_{\log( \lfloor \sqrt {cq^n} \rfloor+1)}u(-2s, -s^2, s)\mathbb Z^2) =f(a_{nt+t_0+c^-_n}u(-2s, -s^2, s)\mathbb Z^2)\\
l_n(s)&=f(a_{\log \lfloor \sqrt {cq^n} \rfloor}u(-2s, -s^2, s)\mathbb Z^2) =f(a_{nt+t_0+c^+_n}u(-2s, -s^2, s)\mathbb Z^2)
\end{align*}
  where $t_0=\log\sqrt c, t=\log \sqrt q$ and
  \[
 c^-_n=\log\frac{\lfloor \sqrt {cq^n} \rfloor+1}{\sqrt {cq^n}}\to 0,\quad  c^+_n=\log\frac{\lfloor \sqrt {cq^n} \rfloor}{\sqrt {cq^n}}\to 0 \quad \text{as }n\to \infty.
  \]

We claim that
 for a.e.~$s\in[0,1]$
 and any $\varphi\in C_c(X)$ one has
 \[
 \lim_{N\to \infty}\frac{1}{N}\sum_{n=1}^N\varphi ({a_{(nt+t_0+c^\pm_n)}u(-2s, -s^2, s)\Gamma})=\int_X \varphi\dd \mu_X.
 \]
The claim follows from a discrete version of
  Corollary~\ref{cor:gencurve} which can be derived from the continuous version using
  entropy method.
  Note that for an $a_t$-invariant probability measure $\mu$ on $X$, the measure $\nu:=\int_0^t (a_r)_* \mu \dd r$ is invariant with respect to the geodesic flow. Moreover, the measures $\mu$ and  $\nu$ have the same entropy with respect to $a_t$.
   Recall that  $\mu_X$ is the unique
$a_t$-invariant probability measure on $X$ with maximal entropy (see e.g.~\cite{el}).  So if $\nu=\mu_X$, then so does $\mu$.

 For every $l\geq 0$ let $E_l=\{\Lambda\in X: f(\Lambda)\le l\}$
 and let $\mathbbm 1_{E_l}$ be the indicator function of $E_l\subset X$.
 It is proved in \cite[Proposition 3.9]{em04}
that $\mu_X(\partial E_l)=0$. Therefore, for a.e.\ $s\in[0,1]$
we have
\begin{align*}
\frac{1}{N}&\#\{1\leq n\leq N:f(a_{(nt+t_0+c^\pm_n)}u(-2s, -s^2, s)\Z^2)\leq l\}\\
&=\frac{1}{N}\sum_{n=1}^N\mathbbm{1}_{E_l}({a_{(nt+t_0+c^\pm_n)}u(-2s, -s^2, s)\Gamma})\to\mu_X(E_l)=\int_{0}^l tF(t)\,dt.
\end{align*}
Therefore, for a.e.\ $s\in [0,1]$  the sequences
$\{t_n(s)\}_{n\in \N}$ and $\{l_n(s)\}_{n\in\N}$ converge in distribution to $\mu$, which completes the proof.
\end{proof}

\section{Proof of the Birkhoff genericity along curves}\label{sec:Birkhoff_proof}

The aim of this section is to prove Theorem \ref{thm;equi}. We
first setup some notation that will be used throughout this
section. Let $G,\Gamma, X, a_t, u, u_\varphi$ be as in \S~\ref{sec:Birkhoff}
and $\varphi$ be as in Theorem \ref{thm;equi}.
For convenience we let $u(a)=u(a, 0, 0)$ and
$u(a,b)=u(a,b, 0)$.
The following notation will also be used during the proof:
\begin{align*}
 H&=\{(h,0): h\in SL_2(\mathbb R)\} \le G \\
 U& =\{u(t): t\in \mathbb R\} \\
 D& =\{a_t: t\in \mathbb R\}\\
  \rho&:G\to SL_2(\mathbb R) \quad\text{ which maps } (h,\xi)\to h.
\end{align*}
We also denote by $ \rho$ the induced  map of homogeneous spaces
\begin{equation*}
\rho:G/\Gamma\to SL_2(\mathbb R)/SL_2(\mathbb Z), \quad
\rho(g\Gamma)=\rho(g)SL_2(\mathbb Z).
\end{equation*}


 For every $s\in[0,1]$ and $T>0$
let $\mu_{s,T}$ be a probability measure on $X$ given by
\[\mu_{s,T}=\frac{1}{T}\int_{0}^T\delta_{a_tu_{\varphi}(s)\Gamma}\,dt\]
where $\delta_{a_tu_\varphi(s)\Gamma}$ is the point mass measure
on $a_t\, u_\varphi(s)\Gamma$.
All the constants depending on $\varphi$ will not be specified in
this section.

\smallskip

\subsection{Outline  of the proof of Theorem \ref{thm;equi}.} The proof consists of two parts. We first consider any  weak$^*$ limit of
$\mu_{s, T}$ as $T\to \infty$ and prove that it is a measure on $X$ invariant under the unipotent subgroup $U$ (see \S~\ref{sec:outline_Birkhoff}).
It is easy to see that any weak$^*$ limit $\mu$ is also  invariant under the group $D$. According to a result by Mozes in \cite[Theorem 1]{m95}, a finite measure on $X$ invariant under
the group $DU$ is automatically $H$-invariant. Hence, one can apply the celebrated Ratner's Theorem (see \cite{r91}), which gives  that
any $H$-invariant and ergodic probability
measure on $X$ is homogeneous and supported on some orbit closure
of $H$. The   $H$-orbits  are well known: each orbit is  either closed or dense. To prove Theorem \ref{thm;equi} we need a precise description of  closed orbits, which is  given  at the beginning of
 \S~\ref{sec:non_concentration}.  The main result to prove the second part is Proposition \ref{prop;main}. In \S~\ref{sec:non_concentration}, assuming Proposition \ref{prop;main}, we conclude the proof of  Theorem \ref{thm;equi}. The proof of Proposition \ref{prop;main}, which is rather long and technical, is postponed to the last two  sections, i.e. \S~\ref{sec;height}, where a suitable mixed height function is constructed,  and \S~\ref{sec;singular}, where the height function is used to show that there is no accumulation of mass on closed orbits. Before the proof of Proposition \ref{prop;main}, in \S~\ref{sec:proof_cor}, we deduce from Theorem \ref{thm;equi}  the  Birkhoff genericity result for more general curves, i.e. Corollary \ref{cor:gencurve}.

\subsection{Unipotent invariance}\label{sec:outline_Birkhoff}
In this section  we prove the unipotent invariance of every weak$^*$ limit of
$\mu_{s, T}$ as $T\to \infty$.  The methods of the proof are inspired by the work of Eskin-Chaika \cite{Es-Ch} and extends their technique to more general curves in our setup.

\begin{lem}[Lemma~3.4 in \cite{Es-Ch}]
\label{lem:ECh}
Let $\eta:[0,+\infty)\times[0,1]\to \R$ be a bounded measurable function for which there exist $C, L, \lambda>0$ such that
\begin{itemize}
\item[(i)] $|\eta(t,s)-\eta(l,s)|\leq L|t-l|$;
\item[(ii)] $\big|\int_0^1  \eta(t, s) \eta(l,s)\, ds\big| \le C e^{-\lambda|l-t|}$.
\end{itemize}
Then $\frac{1}{T}\int_0^T \eta(t,s)\,dt \to 0$ for a.e.\ $s\in[0,1]$.
\end{lem}
The above lemma  was proved in \cite{Es-Ch} (cf.\ Lemma 3.4) with a redundant assumption $\int_0^1\eta(t,s)\,ds=0$ for every $t\geq 1$. However, this assumption was not used in the proof.

\begin{prop}\label{prop;inva}
Let $\varphi:[0,1]\to \mathbb{R}$ be a $C^1$-function. Then
for a.e.\ $s\in [0,1]$  any weak$^*$ limit $\mu$ of $\mu_{s, T}$ as $T\to \infty$ is
invariant under the group $U$.
\end{prop}

\begin{proof}
For every function  $f\in C_c^{\infty}(X)$ and every $r\in \mathbb R$ we
let
\begin{align}\label{eq;f r t s}
 f _{ r}(t,s)=  f (u(r) a_t u_\varphi(s)\Gamma)- f (a_t u_\varphi(s)\Gamma).
\end{align}
Fix a countable set $\mathcal S\subset C_c^\infty(X)$ which is
dense in $C_c(X)$ with respect to  sup-norm. Then it suffices to show that for every
$f\in\mathcal S$ and $r\in \R$ we have
\begin{align}\label{eq;invariance}
 \frac{1}{T}\int_0^T f _r(t,s)\,dt \to 0\quad \text{ for a.e.}\ s\in[0,1] .
\end{align}
By  Lemma~\ref{lem:ECh}, to prove (\ref{eq;invariance}) it
suffices to show that
\begin{align}\label{eq;psi l t}
\Big|\int_0^1  f _{r}(t, s) f _{r}(l,s)\, ds\Big| \le M( f , r)
e^{- |l-t|},
\end{align}
where $M( f , r)$ is a constant depending on $ f $  and $r$. Since
in (\ref{eq;psi l t}) the roles of $l$ and $t$ are symmetric we
assume that $l>t$. Then for any
\[
s\in I(s_0)=[s_0-e^{{-l-t}}, s_0+e^{{-l-t}}]\subset [0, 1]
\]
we have
\begin{align}\label{eq;g t u phi}
\begin{aligned}
a_tu_\varphi(s)&=a_t u(s-s_0, \varphi(s)-\varphi(s_0))a_{-t}\cdot a_tu_\varphi(s_0) \\
&=u\big(e^{2t}(s-s_0)), e^t\varphi'(\tau(s))(s-s_0)\big)\cdot a_tu_\varphi(s_0)  \\
& =u(O(e^{-l+t}), O(e^{-l+t})) a_tu_\varphi(s_0),
\end{aligned}
\end{align}
where $\tau(s)$ is a real number determined  by the mean value
theorem. Since $f:X\to\R$ is compactly supported there exists
$C>0$ such that
\begin{equation}\label{neq:lip}
|f(g_1\Gamma)-f(g_2\Gamma)|\leq C\|(Id,0)-g_2g_1^{-1}\|\ \text{ for
all }\ g_1,g_2\in G,
\end{equation}
where if $g_2g_1^{-1}=(h, \xi)$ then $\|(Id,0)-g_2g_1^{-1}\|=\|Id-h\|+\|\xi\|$.
It follows from the definition
of $f_r(t, s)$ in (\ref{eq;f r t s}), (\ref{eq;g t u phi}) and the
smoothness of $f$ that
\begin{align*}
f_{r}(t, s)-f_r(t, s_0)=O_{f,r}(e^{-l+t}).
\end{align*}
Therefore
\begin{align}\label{eq;estimate}
\int_{I{(s_0)}}  f _{ r}(t,s) f _{ r}(l,s)\,  ds =  { f _{
r}(t,s_0)}\int_{I{(s_0)}}  f _{ r}(l,s)\,
ds+|I(s_0)|O_{f,r}(e^{-l+t}).
\end{align}
On the other hand
\begin{align*}
u(r)  a_l  u_\varphi(s)&=  a_l   u(0, \varphi(s)-\varphi(s+re^{-2l}))a_{-l}\cdot  a_l  u_\varphi(s+re^{-2l}) \\
&=u(0, -re^{-l}\varphi'(s+\tau(r) e^{-2l}))\cdot  a_l  u_\varphi(s+re^{-2l}) \\
&=  u(0, O_r(e^{-l+t})) a_l  u_\varphi(s+re^{-2l}),
\end{align*}
where $\tau(r)$ is determined by the mean value theorem. According
to \eqref{neq:lip}, it follows that
\begin{align*}
 f _r(l, s)=  f (  a_l  u_\varphi(s+re^{-2l})\Gamma)- f ( a_l  u_\varphi(s)\Gamma)+O_{f, r}(e^{-l+t}),
\end{align*}
and hence
\begin{align}\label{eq;psi l}
\int_{I(s_0)} f _r(l, s)\, ds=O_{f, r}(e^{-l+t})|I(s_0)|.
\end{align}
In view of (\ref{eq;estimate}) and (\ref{eq;psi l}), we have
 \begin{equation}\label{eq:lastest}
  \int_{I{(s_0)}}  f _{ r}(t,s) f _{ r}(l,s)\,  ds\\=O_{f, r}(e^{-l+t})|I(s_0)|.
 \end{equation}
Now (\ref{eq;psi l t}) follows by splitting $[0,1]=\bigcup_{1\le
k\le m} I_k$ into intervals $I_k=[s_{k-1}, s_k]$ with
$s_k-s_{k-1}=2e^{-l-t}$ for $1\le k< m$ and $s_m-s_{m-1}\le
2e^{-l-t}$, and then by applying \eqref{eq:lastest} to intervals
$I_k$ $(1\le k< m)$.
\end{proof}

\subsection{Proof of Theorem \ref{thm;equi} assuming Proposition \ref{prop;main}}\label{sec:non_concentration}
For every $h\in SL_2(\mathbb R)$ we have \[
 \rho^{-1}( h\, SL_2(\mathbb Z))\cong  \mathbb R^2/\mathbb Z^2,
\]
as $(h, \xi_1)\Gamma=(h, \xi_2)\Gamma$ if and only if
$\xi_1=\xi_2+h\xi_0$ for some $\xi_0\in \mathbb Z^2$. If the $H$-orbit
of  $x=(Id,  \xi)\Gamma\in X$ is closed, then
\[
 \rho^{-1}( SL_2(\mathbb Z))\cap Hx=\{(Id , \gamma \xi)\Gamma: \gamma\in SL_2(\mathbb Z)\}
\]
is closed. In view of \cite[Theorem 2]{gs04}, the $SL_2(\mathbb
Z)$-orbit of $\xi+ \mathbb Z^2\in \mathbb R^2/\mathbb Z^2$ is
finite if $\xi\in \mathbb Q^2$ and dense otherwise. Therefore the
closed $H$-orbits in $X$ are exactly orbits of $(Id,  \xi)$ with
$\xi\in \mathbb Q^2$.

For  every $n\in\N$ and $i, j\in \mathbb Z$ we let
\[
G[n]^{ij}=\left \{ (h, h\left({i}/{n}, {j}/{n}\right)): h\in
SL_2(\mathbb R) \right\}\subset G.
\]
Let $X[n]$ be  the image of $\bigcup_{i,j\in\Z} G[n]^{ij}\subset
G$ in $X$. Then $X[n]$ is a finite union of  closed $H$-orbits and
any closed $H$-orbit is contained in some $X[n]$. Therefore, it suffices to   show that for almost every $s\in [0,1 ]$ any  weak$^*$
limit $\mu$ of $\mu_{s,T}$ as $T\to \infty$ is a probability measure and $\mu$ puts no mass on
$X[n]$.


\begin{prop}\label{prop;work}
Under the assumptions of Theorem \ref{thm;equi}, for a.e.\ $s\in
[0, 1]$ any weak$^*$ limit $\mu$ of $\mu_{s, T}$ as $T\to \infty$
is a probability measure on $X$ and satisfies $\mu(X[n])=0$ for
every positive integer $n$.
\end{prop}

\begin{rem}\label{rem:redint}
We will  prove that the condition $\mu(X[n])=0$ holds for a.e.~$s\in I$, where $I\subset (0, 1)$ is a closed
interval such that $\varphi|_I$ satisfies an additional regularity property, i.e.~(\ref{eq;sigma}) holds. 	
The constant
\begin{align}
M_1=2\sup_{s\in [0,1]}|\varphi'(s) |+1  \label{eq;m 1}
\end{align}
plays an important role while defining this regularity property.
According to the assumption of $\varphi$ the set
\[
\{s\in [0, 1]: \mbox{there are } i,j\in \Z \mbox{ such that }|j/n|\le M_1 \mbox{ and }
\varphi(s)=js/n+i/n \}
\]
has Lebesgue measure zero. Since this set is closed, there exist
at most countably many open intervals $\{I_k\}$ such that:
\begin{itemize}
\item elements of $\{I_k\}$ are pairwise disjoint;
\item $\bigcup I_k$ has full measure in $[0,1]$;
\item for any $s\in I_k$ and any $i, j\in \Z$ we have $\varphi(s)\neq js/n+i/n$ if  $|j/n|\le M_1$.
\end{itemize}
Since  each interval $I_k$ is a countable union of closed intervals,
it suffices to show the condition $\mu(X[n])=0$ holds for
every closed interval $I$ contained in some $I_k$. For every such closed interval $I\subset I_k$
the condition (\ref{eq;sigma}) obviously holds.
\end{rem}

Let $K$ be a measurable subset of $X$, $T>0$ and $s\in [0,1]$. The
proportion of the trajectory $\{a_tu_\varphi(s)\Gamma:0\le t\le
T\}$ in $K$ is expressed by the function  $\mathcal A _K^{T}:[0,
1]\to [0, 1]$ defined by
\[
\mathcal A_K^T(s)=\frac{1}{T}\int_0^T {\mathbbm
1}_K(a_tu_\varphi(s)\Gamma)\, dt=\frac{1}{T}\int_0^T
\delta_{a_tu_\varphi(s)\Gamma}(K)\, dt=\mu_{s,T}(K),
\]
where $\mathbbm 1_K$ is the characteristic function of $K$.
Proposition \ref{prop;work} will   follow from the following
quantitative estimate of $\mathcal A_K^T(s)$.
\begin{prop}\label{prop;main}
Let $I$ be a closed interval of $(0,1)$ and let $n\in \N$. Suppose
that
\begin{equation}\label{eq;sigma}
\inf\{|\varphi(s)-js/n+i/n|: s\in I; i,j\in\mathbb Z \mbox{ and } |j/n|\le M_1  \}=\sigma>0.
\end{equation}
Then for any $\varepsilon >0$ there exists a compact subset $K$ of $X\setminus X[n]$
and  $\vartheta>0$ such that for any $T>0$
\begin{equation}\label{eq;proportion}
|\{s\in I :  \mathcal A_K^T(s)\le 1-\varepsilon  \}|\le e^{-\vartheta T}|I|.
\end{equation}
\end{prop}
The proof of Proposition \ref{prop;main} will be given in
\S~\ref{sec;height}~and~\S~\ref{sec;singular}. Here we use it to
derive Proposition \ref{prop;work} and Theorem \ref{thm;equi}.
\begin{proof}[Proof of Proposition \ref{prop;work}]
Fix $n\in \mathbb N$. By Remark~\ref{rem:redint}, it suffices to
show that the conclusion holds for almost every $s\in I$ whenever
$I$ is a closed interval  so that (\ref{eq;sigma}) holds. Given
 $0<   \varepsilon<1$ we choose a compact
subset $K_\vep\subset X\setminus X[n]$ and $\vartheta>0$ so that
(\ref{eq;proportion}) holds. By taking $T=m$ for $m\in \mathbb N$
in (\ref{eq;proportion}) and using the Borel-Cantelli lemma we can
find a  subset $I_ {\varepsilon}$ of $I$ with the following
property: $I_ {\varepsilon }$ has
 full measure in $I$ and
for any $s\in I_   {\varepsilon}$ any weak$^*$ limit $\mu$ of
$\mu_{s,T}$ as $T\to \infty$ satisfies $\mu(K_\vep)\geq
1-\vep$. Therefore, $\mu(X)\ge 1-\vep$ and $\mu(X[n])\le  \varepsilon$. It follows that for
any $s$ in the full measure subset  $\bigcap_{k\in \mathbb
N}I_{\frac{1}{k}}$ of $I$ we have $\mu(X)=1$ and $\mu(X[n])=0$.
\end{proof}
\begin{proof}[Proof of Theorem \ref{thm;equi}]
We claim that there exists a
full measure  subset $I'$
of $[0,1]$  such that for any $s\in I'$ any weak$^*$  limit $\mu $
of $\mu_{s,T}$ as $T\to \infty$ has the following properties:
\begin{enumerate}[label=(\roman*)]
\item $\mu$ is a probability measure;
\item $\mu$ is invariant under the group $U$;
\item $\mu(X[n])=0$ for any positive integer $n$.
\end{enumerate}
The claims (\rmnum{1}) and (\rmnum{3}) follow from Proposition
\ref{prop;work} and (\rmnum{2}) follows  from Proposition~\ref{prop;inva}.

Since  $\mu$ is  $DU$-invariant, it is also $H$-invariant, by
\cite[Theorem 1]{m95}. According to Ratner's measure
classification theorem any ergodic $H$-invariant probability
measure on $X$ is either $\mu_X$ or supported on some $X[n]$. Therefore
claim (\rmnum{3}) implies that $\mu=\mu_X$. Since $\mu$ is an
arbitrary weak$^*$ limit, it follows that $\mu_{s,T}\to\mu_{X}$ as
$T\to\infty$  for any $s\in I'$.
\end{proof}

\subsection{Birkhoff genericity for more general curves}\label{sec:proof_cor}
We conclude this section by considering some extensions of Theorem \ref{thm;equi}. In particular, we give the proof of Corollary \ref{cor:gencurve}.
\begin{cor}\label{cor;general base}
Let $\Gamma'$ be a lattice in $G$ commensurable with $\Gamma$.
Let $\varphi:[0,1]\to \mathbb R$ be a $C^1$-function, $h\in SL_2(\R)$ and $\xi=(v_1, v_2)^{\mathrm {tr}}\in \R^2$. Suppose for any $(b, l)^{\mathrm{tr}}\in h \Q^2$  the Lebesgue measure of $\{s\in [0, 1]:  \varphi(s)=(l-v_2)s+(b-v_1) \}$
is zero. Then for almost every $s\in [0, 1]$ the coset $u_\varphi(s)(h, \xi)\Gamma'$ is Birkhoff generic
with respect to  $(G/\Gamma', \mu_{G/\Gamma'},a_t)$.
\end{cor}
\begin{rem}\label{rem;homepage}
Let us consider $h=\left(
\begin{array}{cc}
h_{11} & h_{12}\\
h_{21}& h_{22}
\end{array} \right)\in SL_2(\R)$ with $h_{11}\neq 0$ and $\xi=(v_1, v_2)^{\mathrm{tr}}\in\R^2$.
Then for any $x\in G/\Gamma'$ the forward trajectories  of $(h, \xi)x$ and $u(\frac{h_{12}}{h_{11}},\frac{ v_1}{h_{11}})x$ with
respect to  the action of  $D$ are asymptotically parallel, i.e.~there
exists $g\in G$ such that the distance  between  $a_t(h, \xi)x$
and $ga_tu(\frac{h_{12}}{h_{11}},\frac{ v_1}{h_{11}})x$ tends to zero as $t\to \infty$.
Therefore if one of them equidistributes so does the other.
 This observation and the description of closed $H$-orbits before Proposition \ref{prop;work} explains the
assumption of $\varphi$ in Theorem \ref{thm;equi} and Corollary \ref{cor;general base}.
\end{rem}
\begin{lem}\label{lem;homepage}
Let $\Gamma'$ be a lattice in $G$ commensurable with $\Gamma$. An element $(g, \xi)\Gamma$ is Birkhoff generic with respect to  $(X,\mu_X,a_t)$ if and only if $(g, \xi)\Gamma'$ is Birkhoff generic with respect to  $(G/\Gamma',\mu_{G/\Gamma'}, a_t)$.
\end{lem}
\begin{proof}
In the proof we will not use  $\Gamma=ASL_2(\mathbb Z)$,  we only need to use commensurability of $\Gamma$ and $\Gamma'$.
So  without loss of generality we can assume that $\Gamma'\le \Gamma$. Let $\mu$ and $\mu'$
be weak$^*$ limits of
\begin{equation}\label{weaklim}
\frac{1}{T}\int_0^T\delta_{a_t(h, \xi)\Gamma}\dd t\ \text{ and }\
\frac{1}{T}\int_0^T\delta_{a_t(h, \xi)\Gamma'}\dd t
\end{equation}
respectively along the same sequence
$\{T_n\}$ with $T_n\to \infty$.  It suffices to show that $\mu=\mu_{G/\Gamma}$  if and only if $\mu'=\mu_{G/\Gamma'}$.

Let us consider the natural projection $\pi:G/\Gamma'\to G/\Gamma$ which is
a finite covering map. Since $\mu$ and $\mu'$ are weak$^*$ limits of \eqref{weaklim} along the same sequence,
it is easy to see that $\mu$ is  a probability measure if and only if so is $\mu'$ and then $\mu=\pi_*(\mu')$.

Suppose that $\mu'=\mu_{G/\Gamma'}$. Since the map $\pi:G/\Gamma'\to G/\Gamma$ is $G$-equivariant, the measure $\mu$ is also $G$-invariant
and probabilistic.  So $\mu$ is equal to $\mu_{G/\Gamma}$.

The other direction (assuming $\mu=\mu_{G/\Gamma}$) is proved using entropy theory and we refer the readers to \cite{el}
for backgrounds. Since $\mu$ and $\mu'$ are $D$-invariant, $\pi$ yields a factor map between $(G/\Gamma',a_1,\mu')$ and $(G/\Gamma, a_1,\mu)$.
Therefore $h_{\mu'}(a_1)\ge h_{\mu}(a_1)$.  The conclusion that $\mu'=\mu_{G/\Gamma'}$
follows from the following two facts, c.f. \cite[\S~7]{el}:
  (\rmnum{1}) $h_{\mu_{G/\Gamma}}(a_1)=h_{\mu_{G/\Gamma'}}(a_1)$;   (\rmnum{2})
 $h_{\mu'}(a_1)\le h_{\mu_{G/\Gamma'}}(a_1)$ and  the equality holds
  if and only if $\mu'=\mu_{G/\Gamma'}$.
\end{proof}

\begin{proof}[Proof of Corollary \ref{cor;general base}]
According to the  observation in Remark \ref{rem;homepage}
and Lemma \ref{lem;homepage} it suffices to show that  for a.e.~$s$ the coset
\[
u\left(\frac{h_{12} + h_{22}s}{h_{11}+h_{21}s}, \frac{v_2s+v_1+\varphi(s)}{h_{11}+h_{21}s}\right) \Gamma
\]
is Birkhoff  generic. By Theorem \ref{thm;equi} it suffices to check that for any $\tilde b, \tilde l\in \Q$ the set
\[
\{s\in [0, 1]: v_2s+v_1+\varphi(s)=\tilde l(h_{12} + h_{22}s)+\tilde b(h_{11}+h_{21}s)\}
\]
has Lebesgue measure zero. This follows from the assumption with
$l=h_{21}\tilde b+h_{22}\tilde l$ and $b=h_{11}\tilde b+h_{12}\tilde l$.
\end{proof}

\begin{proof}[Proof of Corollary~\ref{cor:gencurve}]
According to  Remark \ref{rem;homepage}
 and Lemma \ref{lem;homepage} it suffices to show that  for a.e.~$s\in [0,1]$ the coset
\begin{equation}\label{eleu}
u\left(\frac{h_{12}(s)}{h_{11}(s)}, \frac{v_1(s)}{h_{11}(s)}\right) (h,v)\Gamma
\end{equation}
is Birkhoff  generic. First we show that the closed sets
\begin{align*}
I_{l, b}&=\{s\in [0,1 ]: v_1(s)=l h_{12}(s) +bh_{11}(s) \}, \ \text{ where }\ l,b\in \R\\
I_1& =\{s\in [0,1]: h_{11}(s)=0  \} \ \text{ and }\\
I_2 & =\{s\in[0, 1]: h_{11}(s)h_{12}'(s)-h_{12}(s)h_{11}'(s)=0\}
\end{align*}
have Lebesgue measure zero. Indeed, since $h_{11}, h_{12}, v_1$
are $C^2$-functions, it is easy to see that for any Lebesgue
density point $s$ from $I_1$, $I_2$ and $I_{l,b}$ respectively we
have
\begin{equation}\label{I1}
h_{11}(s)=0,\ h'_{11}(s)=0, \ h''_{11}(s)=0;
\end{equation}
\begin{align}\label{I2}
\begin{aligned}
h_{11}(s)h_{12}'(s)-h_{12}(s)h_{11}'(s)=&0, \\ h_{11}(s)h_{12}''(s)-h_{12}(s)h_{11}''(s)=&(h_{11}h_{12}'-h_{12}h_{11}')'(s)=0;
\end{aligned}
\end{align}
\begin{align}\label{Ilb}
\begin{aligned}
v_1(s)=&l h_{12}(s) +bh_{11}(s), \ v'_1(s)=l h'_{12}(s)
+bh'_{11}(s), \\ v''_1(s)=&l h''_{12}(s) +bh''_{11}(s)
\end{aligned}
\end{align}
respectively. Moreover, each of conditions \eqref{I1}, \eqref{I2},
\eqref{Ilb} implies $\det M_\psi(s)=0$. Since we assume $\det
M_\psi(s)\neq 0$ almost every, the conclusion follows.

Since $[0,1]\setminus (I_1\cup I_2)$ is open and its Lebesgue
measure is one, we need to prove that for every closed interval
$I\subset[0,1]\setminus (I_1\cup I_2)$ the element \eqref{eleu} is
Birkhoff  generic in $X$ for a.e.\ $s\in I$. For every such
interval $I$ the map $s\mapsto h_{12}(s)/h_{11}(s)$ yields a
$C^2$-diffeomorphism between $I$ and a compact interval $J$.
Moreover, it gives a $C^2$-map $\varphi:J\to\R$ such that
\[\frac{v_1(s)}{h_{11}(s)}=\varphi\Big(\frac{h_{12}(s)}{h_{11}(s)}\Big)\ \text{ for all }\ s\in I.\]
Since
\[u\left(\frac{h_{12}(s)}{h_{11}(s)},
\frac{v_1(s)}{h_{11}(s)}\right) (h,v)\Gamma =
u\left(\frac{h_{12}(s)}{h_{11}(s)},
\varphi\Big(\frac{h_{12}(s)}{h_{11}(s)}\Big)\right) (h,v)\Gamma,\]
in view of Corollary~\ref{cor;general base}, it suffices to show
that for all real numbers $l,b$ the set $\{s\in J:\varphi(s)=ls+b\}$ has
zero measure which follows from the fact that the   Lebesgue measure of  $I_{l,b}$ is    zero . This completes the proof.
\end{proof}

\subsection{Height function}\label{sec;height}
 The aim of this section is to show that for $t$ sufficiently large there
 is  a mixed height function on $X$, with respect to $X[n]$, satisfying certain contraction property along the orbits of $a_t$.
Here mixed refers to the fact that we mix the height with respect to the cusp and
$X[n]$. This  height function will be applied in \S~\ref{sec;singular} to prove  the crucial Proposition \ref{prop;main}. Throughout this section let  $M_1$ be as in (\ref{eq;m 1}).
 \begin{lem}\label{lem;key}
Let $n\in \mathbb N$ and let
$I$ be the closed
interval as in Proposition \ref{prop;main}.
For $t$ sufficiently large (depending on $\sigma$ and $I$) there exists a
measurable function $\beta_n: X\to [1, \infty]$ with the following properties:
\begin{enumerate}[label=(\roman*)]
\item
there exists $b>0$ (depending on $\sigma, n, t$) such that for any  $m\in \mathbb Z_{\ge 0}$ and  any interval
$J\subset I$  with either  $|J|\ge e^{-2mt}$  or $J=I$
one has
\begin{align}\label{eq;linear ineq}
\int_{J}\beta_n( a_{(m+1)t}u_\varphi(s)\Gamma)\, ds< \frac{1}{2}
\int_{J}\beta_n( a_{mt}u_\varphi(s)\Gamma)\, ds+b|J|;
\end{align}

\item for any $c>0$ the set
$\{x\in X: \beta_n(x)\le c\}$
 is  compact;

\item for $x\in X$ one has $\beta_n(x)=\infty$ if and only if $x\in X[n]$;

\item
 for any $m\in \mathbb Z_{\ge 0}$, interval $J\subset I$ with $|J|\le 2e^{-2mt}$ and any $s, \tilde  s\in J $ one has
\begin{align}\label{eq;key 4}
\beta_n({a_{mt}u_\varphi(\tilde s)\Gamma})\le 3\sigma^{-1} \beta _n({a_{mt}u_\varphi( s)\Gamma});
\end{align}

\item for any $m\in \mathbb Z_{\ge 0}, s\in [0, 1]$ and  any $-t\le \tau\le t$ one has
\begin{align}\label{eq;key 5}
\beta_n(a_\tau a_{mt}  u_\varphi(s))\le e^{t} \beta_n (a_{mt}  u_\varphi(s)).
\end{align}

\end{enumerate}
\end{lem}

Before proving the lemma we do some preparation.
Let $t$ be a positive real number which will be specified
only in the proof of Lemma \ref{lem;key} (\rmnum{1}) (cf. \eqref{eq;t}).

Our mixed  height function $\beta_n$, inspired by \cite{EMM},
combines the height with respect  to the cusp and  $X[n]$.
The height of elements of $X$ with respect to the cusp
is measured by the continuous  function
$
\alpha_0: X\to [2^{-1/2}, \infty)
$
where
\begin{equation}\label{eq;alpha 0}
\alpha_0((h, \xi)\Gamma)=\sup_{\xi_0\in \mathbb Z^2\setminus \{0\}} \|h\xi_0\|^{-1/2}.
\end{equation}

\begin{lem}\label{lem;bounded}
Let  $\kappa:  [-1,1]\to [0, \infty]$ be a
measurable function. Suppose that there exists $c >0$ such that
\begin{align}\label{eq;c alpha}
|\{s\in  [-1,1]: \kappa(s)< \varepsilon \}|\le c\varepsilon
\end{align}
for every $\varepsilon >0$. Then
\[
\int_{ -1}^1\frac{ds}{\kappa(s)^{1/2}}\le 4c^{1/2}.
\]
\end{lem}
\begin{proof}
Let $\chi$ be the characteristic function of the set $\{(s, r)\in [-1,1]\times [0, \infty):  \kappa(s)^{-1/2}> r \}$. 	Then
\begin{align*}
\int_{ -1}^1\frac{ds}{\kappa(s)^{1/2}}&=\int_{-1}^1 \int_{0	}^{\infty} \chi(s, r) \dd r \dd s\\
&= \int_{0	}^{\infty}\int_{-1}^1  \chi(s, r) \dd s \dd r  && \mbox{by Fubini's theorem }\\
&= \int_{0	}^{\infty}|\{s\in  [-1,1]: \kappa(s)<  r^{-2} \}| \dd r \\
& \le\int_{\sqrt {c/2} }^\infty c r^{-2}\dd r+  2{\sqrt{c/2} } &&  \mbox{by (\ref{eq;c alpha}}) \\
&=2\sqrt 2 \sqrt c
\le 4 \sqrt c.
\end{align*}

\end{proof}

We will use the following lemma to check (\ref{eq;c alpha}).
\begin{lem}\label{lem;k m lemma}
\begin{sloppypar}
Let  $\kappa:  [-1,1]\to [0, \infty]$ be a
$C^1$-function. Suppose there exists $A_1, A_2>0$ such that  for every $s\in[-1, 1]$
\end{sloppypar}
\begin{align}\label{eq;k m lemma}
|\kappa(s)|, |\kappa'(s)|\le A_1\quad \kappa'(s)\ge A_2,
\end{align}
then (\ref{eq;c alpha}) holds for $c=\frac{24A_1}{A_2\sup_{s\in [-1, 1]} |\kappa(s)|}$.
\end{lem}
\begin{proof}
 It is a 	special case of \cite[Lemma 3.3]{km98}.
\end{proof}

An immediate consequence of above two lemmas is:
\begin{cor}\label{cor;c alpha good}
Let $a,  l \in\R$ be such that $a^2+  l ^2>0$.
Then \begin{align}\label{eq;c alpha good 1}
\int_{-1}^1\frac{ds}{|as+ l  |^{1/2}}< \frac{100}{(a^2+ l ^2)^{1/4}}.
\end{align}
\end{cor}
\begin{proof}
If $2|a|\le  |l|  $, then  $|as+ l |\ge | l |/2\ge {(a^2+ l ^2)}^{1/2}/4$ for all $s\in[-1,1]$, from which
(\ref{eq;c alpha good 1}) follows.
Now suppose $|2a|>  |l| $. Using Lemma \ref{lem;k m lemma} for $\kappa(s)=as+ l $ with
$A_1=|a|+| l |, A_2=(a^2+ l ^2)^{1/2}/\sqrt 5$ and $\sup_{s\in [-1,1]}|\kappa(s)|=|a|+| l |$ one has
(\ref{eq;c alpha}) holds for $c=24\sqrt 5/(a^2+ l ^2)^{1/2}$. Therefore, by Lemma \ref{lem;bounded},
\[\int_{-1}^1\frac{ds}{|as+ l |^{1/2}}\leq\frac{4\sqrt{24 \sqrt 5}}{(a^2+ l ^2)^{1/4}}
< \frac{100}{(a^2+ l ^2)^{1/4}}.\]
\end{proof}
This lemma allows us to get a linear inequality for the height function $\alpha_0$.
\begin{lem}\label{lem;alpha 0}
For every $t\ge 20$ and every $x\in X$ one has
\begin{align}\label{eq;contract 0}
\int_{-1}^{1}{\alpha_0(a_tu(s)x)}\, ds< \frac{1}{4}{\alpha_0(x)}+2e^t.
\end{align}
\end{lem}
\begin{proof}
Recall that $\rho $ is the natural projection map from $G=ASL_2(\R)$
to $SL_2(\R)$ and $\|\cdot\|$ is the Euclidean norm on $\R^2$ as well as
the operator norm on $SL_2(\R)$.
For every $s\in [-1, 1]$
\begin{align}\label{eq;f mt}
\| \rho(a_{t}u(s)) \|&\le 2e^{t}.
\end{align}
Write $x=g\Gamma$ and
recall that $\alpha_0$ is defined according to  shortest nonzero vectors of
the  lattice $\rho(g)\Z^2$.
If $\alpha_0(x)> (2e^{ t})^{1/2}$ and $\xi\in\mathbb R^2$ is
a shortest vector of $ \rho(g)(\Z^2\setminus\{0\})$, i.e.~$\alpha_0(x)=\|\xi\|^{-1/2}$, then
$\| \rho(a_tu  (s))\xi\|<1$ for every $s\in [-1, 1]$, by \eqref{eq;f mt}.
Hence $\rho(a_tu  (s))\xi$ is a shortest vector of $ \rho(a_tu  (s)g)(\Z^2\setminus\{0\})$ for
any $s\in [-1, 1]$.  Write $\xi=(v_1, v_2)^{\mathrm{tr}}$ then
\begin{align*}
\alpha_0(a_tu  (s)x)^2=\|\rho(a_tu  (s))\xi\|^{-1}=\|(e^t(v_2s+v_1),e^{-t}v_2)\|^{-1}\leq e^{-t}|v_2s+v_1|^{-1}.
\end{align*}
Using the  above inequality,  Corollary~\ref{cor;c alpha good} and the assumption $t\ge 20$ one has
\[\int_{-1}^{1}{\alpha_0(a_tu(s)x)}\, ds\leq 100e^{-t/2}\|\xi\|^{-1/2}\leq 100e^{-10}\|\xi\|^{1/2}<
 \frac{1}{4}{\alpha_0(x)^{1/2}}.\]
\begin{sloppypar}
Therefore, in this case (\ref{eq;contract 0}) holds.
If $\alpha_0(x)\le( 2e^{t})^{1/2}$, then (\ref{eq;f mt})
implies $\alpha_0(a_tu(s)x)\le  2 e^{t}$ for all $s\in [-1, 1]$,
from which (\ref{eq;contract 0}) follows.
\end{sloppypar}
\end{proof}

Now we turn to the construction of the mixed height function $\beta_n$.
 There is a natural height function given in \cite[\S~6]{bq13} using Riemannian distance
to $X[n]$ and
this function satisfies a contraction property for the first return map to compact subsets.
The height function used in \cite{EMM} is much more complicated but it satisfies the contraction
property without considering the first return map.
One of the key observations for the height function in \cite{EMM}  is that
the total number of  pieces of $X[n]$ whose distance to $x\in X$ is comparable to
$\alpha_0(x)^{-2}$ is finite.
 The following lemma can be interpreted as a simple version  of
this observation  in our situation.

\begin{lem}\label{lem;unique}

 For every  $(h,\xi)\in G$ there is  at most one element $\xi_0\in \frac{1}{n} \mathbb Z^2$ such
 that
 \begin{align}\label{eq;setJ}
  \|\xi-h\xi_0\|< \frac{1}{2n} \alpha_0((h,\xi)\Gamma)^{-2}.
 \end{align}
\end{lem}
The above lemma is clear from the definition of $\alpha_0$. We will denote the unique
element $\xi_0$ in this lemma by $\zeta_{h, \xi}$ if it exists.
Otherwise  we say $\zeta_{h, \xi}$ does not exist.
If $\zeta_{h,\xi}$  exists then
\begin{equation}\label{eq:defalpha}
\|\xi-h \zeta_{h,\xi}\| < \frac{1}{2n}2
 \le 1.
 \end{equation}
The  height function
$\alpha_n: X\to [1, \infty]$
  with respect to
 the singular subspace $X[n]$ is   defined by
\begin{equation}\label{eq;alpha-n}
\alpha_n((h, \xi)\Gamma)=\left\{
\begin{array}{ll}
\|\xi-h\zeta_{h, \xi}\|^{-1/2} & \text{ if } \zeta_{h, \xi} \mbox{ exists} \\
1 & \text{ otherwise,}
\end{array}
\right.
\end{equation}
where we adopt the convention that  $0^{-1/2}=\infty$.
It can be checked directly that the definition of $\alpha_n$ does not depend on the choice
 of $(h,\xi)$ in the coset $(h, \xi)\Gamma$.

\begin{rem}\label{rem:semicon}
By the definition of $\zeta_{h,\xi}$ and the continuity of $\alpha_0$, the element
$\zeta_{h, \xi}$ is locally constant if it exists.
  It follows that  height function
$\alpha_n: X\to [1, \infty]$ is lower semi-continuous.
\end{rem}

Suppose that $a_{mt}u_\varphi(s)=(h, \xi)$ and
$\xi_0= (i/n, j/n)^{\mathrm{tr}}\in\frac{1}{n}\Z^2$.
For $s\in [0, 1]$ we write
\begin{align}\label{eq;v 1 2}
\left(
\begin{array}{c}
v_1\\
v_2
\end{array}
\right)=
\left(
\begin{array}{c}
v_1(s)\\
v_2(s)
\end{array}
\right)=
\xi-h\xi_0=\left(
\begin{array}{c}
e^{mt}(\varphi(s)-\frac{i}{n}-\frac{js}{n}) \\
e^{-mt}\frac{j}{n}
\end{array}
\right).
\end{align}
Strictly speaking $v_i(s)$ above depends also on $m$ and $\xi_0$. But usually when we use
them $m$  and $\xi_0$ are fixed, so we omit this dependence for simplicity.
During the proof of Lemma \ref{lem;key}, we need to compare  $v_i(s)$ with another $v_i(\tilde s)$.
So we express the latter in terms of the former as follows:
\begin{align}
\begin{split}\label{eq;v i tilde}
\left(
\begin{array}{c}
v_1(\tilde s)\\
v_2(\tilde s)
\end{array}
\right)
=&\left(
\begin{array}{c}
e^{mt}\big(\varphi(\tilde  s)-\varphi(s)-\frac{j(\tilde s-s)}{n}\big)+v_1 \\
v_2
\end{array}
\right)
\\
=&
\left(
\begin{array}{c}
e^{mt}\big(\varphi'(\hat s)-\frac{j}{n}\big)(\tilde s-s)+v_1 \\
v_2
\end{array}
\right)
\end{split}
\end{align}
where   $\hat s$ lies between $s$ and $\tilde  s$ and is determined by the mean value theorem.

In view of  (\ref{eq;sigma}) there exists $\varepsilon=\varepsilon{( \sigma)}>0$ such that $I_{\varepsilon}\subset [0,1]$
where $I_\varepsilon$ is the closed $\varepsilon$-neighborhood of $I$ and
\begin{align}\label{eq;sigma tex}
\inf\{|\varphi({s})-js/n+i/n|:s\in I_\varepsilon ; i,j\in \Z; |j/n|\le M_1 \}\ge \sigma/2.
\end{align}

\begin{lem}\label{lem;readable}
Let $(h, \xi)=a_{mt}u_\varphi(s)$ for some $m\in \mathbb N, s\in I_\varepsilon$  and $t> \log( {2\sigma}^{-1})$. If
 $\zeta_{h, \xi}= (i/n, j/n)^{\mathrm{tr}}$ exists then  $|j/n|> M_1$.
\end{lem}

\begin{proof}
Assume the contrary, i.e.~$|j/n|\le  M_1$.
Using (\ref{eq;v 1 2}),  (\ref{eq;sigma tex}) and the assumption for   $t$,
we have
\[\|\xi-h\zeta_{h, \xi}\|\geq e^{mt}|\varphi(s)-\frac{i}{n}-\frac{js}{n}|
> e^t \sigma/2>1,\]
which is contrary to \eqref{eq:defalpha}.
\end{proof}

The mixed height function with respect to $X[n]$
is the function $\beta_n: X\to [1, \infty]$ defined  by
\begin{equation}\label{eq;height}
\beta_n(x)=\alpha_n(x)+8n e^t\alpha_0(x),
\end{equation}
where $t>0$ is a  real number which will be specified in the proof of
 Lemma~\ref{lem;key} (\rmnum{1}) (cf. \eqref{eq;t}).
We will prove each property of Lemma~\ref{lem;key} separately starting from the simplest.

\begin{proof}[\bf Proof of (\rmnum{3}) in Lemma~\ref{lem;key}]
Directly, by the definition of $\alpha_n$ and $\beta_n$, we have
\[\beta_n(x)=\infty\Longleftrightarrow\alpha_n(x)=\infty\Longleftrightarrow x\in X[n],\]
which gives (\rmnum{3}).
\end{proof}

\begin{proof}[\bf Proof of (\rmnum{2}) in Lemma~\ref{lem;key}]
\begin{sloppypar}
By the definition of $\alpha_0$, the set $\{x\in X: \alpha_0(x)\le c/8n e^t\}$ is a compact subset of $X$ for every $c>0$.
As its subset, the set
\end{sloppypar}
\begin{align}\label{eq;set beta M}
\{x\in X: \beta_n(x)\le c\}
\end{align}
is relatively compact. Since $\alpha_0$ is continuous and $\alpha_n$ is lower semi-continuous (Remark~\ref{rem:semicon}), the map
$\beta_n$ is also lower semi-continuous. Therefore, the set (\ref{eq;set beta M}) is closed and hence (\rmnum{2}) in Lemma
\ref{lem;key} holds.
\end{proof}

\begin{proof}[\bf Proof of (\rmnum{4}) in Lemma \ref{lem;key}]
We will show that  (\rmnum{4}) holds  for  $t> \log2\sigma^{-1}$.
In case where  $m=0$ we have $\alpha_0(u_\varphi(s)\Gamma)=1$ and
$1\le \alpha_n(u_\varphi(s)\Gamma)\le \sigma^{-1}$ for all $s\in I$.
Therefore (\ref{eq;key 4}) holds.

 Now assume  $m\ge 1$. Let
$a_{mt}u_\varphi( s)=(h,\xi),\  a_{mt}u_\varphi(\tilde s)=(\tilde h, \tilde \xi)$,
$x=a_{mt}u_\varphi(s)\Gamma$
and  $\tilde x=a_{mt}u_\varphi(\tilde s)\Gamma$.

 We claim that
\begin{align}\label{eq;raining 4}
\alpha_n(\tilde x)\le 6n\alpha_0(x)+3 \alpha_n(x).
\end{align}
If $\zeta_{\tilde h, \tilde\xi}$ does not exist then (\ref{eq;raining 4}) follows trivially.
Otherwise
suppose $ \zeta_{\tilde h, \tilde \xi}=\xi_0=(i/n, j/n)^{\mathrm{tr}}$.
By  Lemma~\ref{lem;readable} we have $|j/n|> M_1$.
To prove the claim we will use  the notation  of (\ref{eq;v 1 2}) and (\ref{eq;v i tilde}).
 Since  $|\varphi'(\hat s)|<M_1/2<|j|/2n$ (see (\ref{eq;m 1})), we have
$|\varphi'(\hat s)-j/n|\le3 |j|/2n$. By the assumption $|\tilde s-s|\le 2e^{-2mt}$, it follows that
\[|e^{mt}(\varphi'(\hat s)-j/n)(\tilde s-s)|\leq 3e^{-mt}|j|/n=3|v_2|.\]
 Hence by considering the cases where $|v_1|\ge 4|v_2|$ and $|v_1|< 4|v_2|$ separately one has
\begin{align*}
\|\tilde \xi-\tilde h\xi_0\|&\ge \max\{|v_2|, |v_1|-3|v_2|\}
\ge \frac{1}{4}\max\{|v_1|, |v_2|\} \\
& \ge  \frac{1}{8}\sqrt{v_1^2+v_2^2}=\frac{1}{8}\|\xi- h\xi_0\|.
\end{align*}
Therefore
\begin{align}
\alpha_n(\tilde x) = \|\tilde \xi-\tilde h\xi_0\|^{-1/2}
 \le 3\| \xi- h\xi_0\|^{-1/2} .
\label{eq;raining 3}
\end{align}
If $\zeta_{h, \xi}=\xi_0$ then (\ref{eq;raining 4}) follows from (\ref{eq;raining 3}).
If $\zeta_{h, \xi}\neq\xi_0$ (which means either $\zeta_{h, \xi}$ does not exist
or it exists but is  not equal to $\xi_0$), then it follows  from the definition of $\zeta_{h, \xi}$
 that
$ \|\xi- h\xi_0\|^{-1/2}\le  2n \alpha_0(x).
$
Combine this with
(\ref{eq;raining 3}) we get
 (\ref{eq;raining 4}).

We choose  $\xi_0\in \Z^2\setminus\{0\}$ with
 $\alpha_0(\tilde x)=\|\tilde h \xi_0\|^{-1/2}$.
 Note that by assumption $\|h\tilde h^{-1}\|=\|u(e^{2mt}(s-\tilde s))\|\le 3$. So
\begin{align}
\alpha_0(x)\ge \|h\xi_0\| ^{-1/2}=\|h \tilde h^{-1}\cdot\tilde h\xi_0\|^{-1/2}
\label{eq;change room}
\ge  \frac{1}{2}\| \tilde h\xi_0\|^{-1/2} = \frac{1}{2}\alpha(\tilde x).
\end{align}
Therefore \begin{align*}
\beta_n(\tilde x)&= \alpha_n(\tilde x)+8ne^t\alpha_0(\tilde x) \\
& \le 3\alpha_n(x)+6n\alpha_0(x)+16ne^t
 \alpha_0(x)&& \mbox{by (\ref{eq;raining 4}), (\ref{eq;change room})}\\
& \le 3\beta_n(x).
\end{align*}
This completes the proof.
\end{proof}

\begin{proof}[\bf Proof of (\rmnum{5}) in Lemma \ref{lem;key}]
It is easy to see that
\begin{align}\label{eq;txunion}
\alpha_0(a_\tau x)\le e^{t/2} \alpha_0(x).
\end{align}
Let $a_{mt}u_\varphi( s)=(h,\xi)$ and $x=a_{mt}u_\varphi(s)\Gamma$.
If $\zeta_{a_\tau h, a_\tau \xi}$ does not exist
then  (\ref{eq;key 5})   follows from (\ref{eq;txunion}).
Otherwise  take $ \zeta_{a_\tau h, a_\tau \xi}=\xi_0$.
We have
\begin{align}
\begin{split}\label{eq;txunion1}
\alpha_n(a_{\tau}x)&=\|a_\tau\xi-a_\tau h\xi_0 \|^{-1/2}  \le  e^{t/2}\|\xi-h\xi_0\|^{-1/2} \\
&\le \left \{
\begin{array}{ll}
2ne^{t/2}\alpha_0(x)&  \text{if }\xi_0\neq \zeta_{h, \xi} \\
e^{t/2}\alpha_n(x)  &\text{if }\xi_0 = \zeta_{h, \xi}
\end{array}
\right. \\
&\le  e^{ t/2} 2n\alpha_0(x)+e^{ t/2}\alpha_n(x).
\end{split}
\end{align}
By (\ref{eq;txunion}) and  (\ref{eq;txunion1})
we have
\begin{align*}
\beta_n(a_\tau x)
 \le   e^{t/2} 2n\alpha_0(x)+e^{t/2} \alpha_n(x)+e^{t/2}8ne^t\alpha_0( x)  \le e^{t} \beta_n(x).
\end{align*}
\end{proof}

\begin{proof}[\bf Proof of (\rmnum{1}) in Lemma \ref{lem;key}]

We will  show that for any $t, b>0$ with
\begin{align}
t& \ge   30 +\log{2\sigma^{-1}}+\log |I|^{-1}-{\log \varepsilon} \label{eq;t}\\
\quad b&\ge 32ne^{2t} +e^t\sigma^{-1/2}
\label{eq;b}
\end{align}
 property (\rmnum{1}) of
Lemma \ref{lem;key} holds.
We note that   (\rmnum{2}),  (\rmnum{3}), (\rmnum{4})  and (\rmnum{5})  of Lemma \ref{lem;key} hold according to the upper bound of $t$. So together with (\rmnum{1}) the proof of Lemma \ref{lem;key} is complete.

In the remaining proof $t\ge 30$ will be used several times without being explicitly mentioned.
The last term of the lower bound of $t$ guarantees that $s+e^{-2mt}\tilde s\in I_{\varepsilon}$
 for every $s\in I, \tilde s\in [-1, 1]$ and $m\in \N$.  The third term of the upper bound
 of $t$ implies that for $m\in \N$ we must have $|J|\ge e^{-2mt}$ according to the requirement
 of $J$.  The second term of the upper bound of $t$ shows that Lemma \ref{lem;readable} holds.

In case where $m=0$ we must have $J=I$ since  otherwise   $|J|\ge 1>|I|$ which contradicts
the assumption $J\subset I$.
For every $s\in I$ one has
\begin{align}
\begin{split}
\beta_n({a_tu_\varphi(s)\Gamma})&\le e^{t}\beta_n({u_\varphi(s)\Gamma})\qquad\qquad\qquad\qquad\qquad \mbox{by (\rmnum{5})}\label{eq;austin}\\
& =e^t \alpha_n(u_\varphi(s)\Gamma)+8ne^{2t}\alpha_0(u_\varphi(s)\Gamma).
\end{split}
\end{align}
It follows directly from the definition (see (\ref{eq;alpha 0}))  that $\alpha_0(u_\varphi(s)\Gamma)=1$.
Also using (\ref{eq;v 1 2}) and the definition (see  (\ref{eq;alpha-n})) one has
$\alpha_n(u_\varphi(s)\Gamma)\le \sigma^{-1/2}$.
Therefore continuing (\ref{eq;austin}) we have
\[
\beta_n({a_tu_\varphi(s)\Gamma})\le e^t\sigma^{-1/2}+8ne^{2t},
\]
which in view of (\ref{eq;b}) implies (\rmnum{1}) of Lemma \ref{lem;key}.

In the rest of the proof we assume
 $m\ge 1$ and
show that  (\ref{eq;linear ineq}) holds for the interval $J\subset I$
 with $|J|\ge e^{-2mt}$.
It follows from the lower bound of $|J|$ that
\begin{align}
\label{eq;double}
\int_J \beta_n(a_{(m+1)t}u_\varphi(s)\Gamma)\,
ds \le \int_J\int_{-1}^1 \beta_n(a_{(m+1)t}u_\varphi(s+e^{-2mt}
\tilde s )\Gamma)\, d \tilde s  ds.
\end{align}
Recall that $\beta_n$ consists two summands $\alpha_0$ and $\alpha_n$ and
the former is well understood by Lemma \ref{lem;alpha 0}. More precisely, since $t \ge 30$ and
\[
\alpha_0(a_{(m+1)t}u_\varphi(s+e^{-2mt} \tilde s)\Gamma)=\alpha_0(a_tu(\tilde s) a_{mt}u_\varphi(s)\Gamma),
\]
Lemma \ref{lem;alpha 0} provides
\begin{align}\label{eq;alpha 0 final}
\int_{-1}^1\alpha_0(a_{(m+1)t}u_\varphi(s+e^{-2mt} \tilde s)\Gamma)\, ds
\le \frac{1}{4}\alpha_0( a_{mt}u_\varphi(s)\Gamma)  +2e^t.
\end{align}

Our strategy for the term $\alpha_n$ is the following: for fixed $s$ we will give an upper bound of  the integral
\begin{align}\label{eq;omega n}
\omega(s):=\int_{-1}^1 \alpha_n(a_{(m+1)t}u_\varphi(s+e^{-2mt}\tilde s)\Gamma)\, d\tilde s
\end{align}
using the data in the definition of $\beta_n(a_{mt}u_{\varphi}(s)\Gamma)$.
This will be completed in (\ref{eq;omega final}).
To simplify the notation we will not specify the dependence on $s$ and
set
\begin{align}
&x  = a_{mt}u(s, \varphi(s))\Gamma, \quad
(h , \xi)= a_{mt}u(s, \varphi(s)), \notag\\
&(h( \tilde s ), \xi( \tilde s )) = a_{(m+1)t}u_\varphi(s+e^{-2mt} \tilde s ).
\notag
\end{align}
We  use $\zeta(\tilde{s})$ to denote
 $\zeta_{h( \tilde s ), \xi( \tilde s ) }$ for simplicity.
Let     $\xi_0=(i/n, j/n)^{\mathrm{tr}}\in \frac{1}{n}\mathbb Z^2$
and let the  notation be as in  (\ref{eq;v 1 2}) and (\ref{eq;v i tilde}).  We set
\begin{align}
\begin{split}
\left(
\begin{array}{c}
w_1( \tilde s ) \\
w_2( \tilde s )
\end{array}
\right):&= \xi( \tilde s )-h( \tilde s ) \xi_0 =
\left(
\begin{array}{c}
v_1( s+\tilde s e^{-2mt} ) \\
v_2(s+ \tilde s e^{-2mt} )
\end{array}
\right)
 \\
&=
\left(
\begin{array}{c}
  e^{(m+1)t}(\varphi(s+e^{-2mt}\tilde s)-\varphi(s))- e^t
  e^{-mt}\frac{j}{n}\tilde s+e^tv_1\\
e^{-t}v_2
\end{array}
\right),\label{eq;w 1 tilde}
\end{split}
\end{align}
where $v_1=e^{mt}(\varphi(s)-\frac{i}{n}-\frac{js}{n})$ and $v_2=e^{-mt}\frac{j}{n}$
as in (\ref{eq;v 1 2}).
By mean value theorem there exists $\hat s\in [0, 1]$ such that
\begin{align}\label{eq;reduce w 1}
 w_1(\tilde s)=e^t\left [e^{-mt}\left(\varphi'(\hat s)-\frac{j}{n}\right) \tilde s +v_1\right].
\end{align}
If $|j/n|>M_1$ which is always satisfied    below by Lemma \ref{lem;readable}
(recall that $m\ge 1$ and $t\ge \log 2\sigma^{-1}$),
then
we have $|\varphi'(\hat s)|<M_1/2<|j/2n|$, and hence $|j/2n|<|\varphi'(\hat s)-j/n|<2|j/n|$. It follows that
\begin{align}\label{eq;clarify}
|v_2|/2= {e^{-mt}|j|}/{(2n)}\le  |e^{-mt}(\varphi'(\hat s)-{j}/{n}) |
\end{align}
which will be used several times below for  estimating $w_1(\tilde s)$.
To estimate (\ref{eq;omega n}) we consider two cases of $\xi_0$ ({\bf Case A} and {\bf Case B} below).

{\bf Case A}:
Assume $\xi_0=\zeta(\tilde s)$ where $\tilde s\in [-1, 1]$ but $\xi_0\neq \zeta_{h, \xi}$ which include the case where $\zeta_{h, \xi}$ does not exist.
  We will show
that
\begin{align}\label{eq;setminus}
{\|\xi(\tilde s)-h(\tilde s)\xi_0\|^{-1/2}}\le ne^t\alpha_0(x).
\end{align}
If $|v_1|\ge 4|v_2|$,  by (\ref{eq;reduce w 1}) and  (\ref{eq;clarify}) we have
\begin{align*}
\|\xi(\tilde s)-h(\tilde s)\xi_0\| \ge |w_1(\tilde s)| \ge  e^t (|v_1|-2|v_2|)
\ge e^t\|\xi-h\xi_0\|/2\sqrt 2\ge   \frac{1}{n}\alpha_0^{-2}(x),
\end{align*}
from which (\ref{eq;setminus}) follows.
If $|v_1|<4 |v_2|$,
we have
\begin{align*}
\|\xi(\tilde s)-h(\tilde s)\xi_0\| \ge |w_2(\tilde s)| \ge  e^{-t} |v_2|
\ge e^{-t}\|\xi-h\xi_0\|/4\sqrt 2\ge  \frac{1}{ne^{2t}}\alpha_0^{-2}(x),
\end{align*}
from which (\ref{eq;setminus}) follows.

{\bf Case B}:
Assume  $\xi_0=\zeta_{h, \xi}$.
 We will show that
\begin{align}\label{eq;sum j 2}
\int_{-1}^1\|\xi(\tilde s)-h(\tilde s)\xi_0\|^{-1/2}\, d\tilde s< \frac{1}{4}\|\xi-h\xi_0\|^{-1/2}.
\end{align}
Since
$
\| \xi( \tilde s )-h( \tilde s ) \xi_0\|\ge |w_1( \tilde s )|,
$
it suffices to estimate the  lower bound of
$w_1(\tilde s)$.
Here we consider two subcases
({\bf Case B1} and {\bf Case B2} below).

{\bf Case B1:} $4|v_2|\le|v_1|$. Then, in view of \eqref{eq;reduce w 1} and  \eqref{eq;clarify}, for every $\tilde s\in[-1,1]$ we have
\[
 |w_1(\tilde s)|\ge e^t (|v_1|- 2|v_2|)\ge \frac{e^t}{2}|v_1|\ge  \frac{e^t}{2\sqrt 2}\sqrt{v_1^2+v_2^2}> {64}\|\xi-h\xi_0\|.
 \]
Therefore, \eqref{eq;sum j 2} holds.

{\bf Case B2:} $4|v_2|>|v_1|$. In this case we use
Lemma \ref{lem;k m lemma} for the $C^1$-function $\kappa=w_1$.
For this purpose we need upper and lower bound for $|w_1|$ and $|w_1'|$.
 By (\ref{eq;w 1 tilde}) one has
 \begin{align}
 \label{eq;w 1 prime}
 w_1'(\tilde s)=e^{(1-m)t}(\varphi'(s+e^{-2mt}\tilde s)-j/n)
 \end{align}
 whose sign is the same as $j/n$.
 By \eqref{eq;clarify}  we have
 \[
 |w'_1(\tilde s)|\le  2e^t |v_2|\le 2e^t\sqrt{v_1^2+v_2^2}=  2e^t \|\xi-h\xi_0\|
 \]
 and
 \[
 | w_1'(\tilde s)|\ge \frac{e^t}{2}|v_2|\ge \frac{e^t}{10}\sqrt{v_1^2+v_2^2}=\frac{e^t}{10} \|\xi-h\xi_0\|.
 \]
In view of  (\ref{eq;reduce w 1}) and  \eqref{eq;clarify}, we also have
\[
|w_1(\tilde s)|\le e^t(|v_1|+2|v_2|)\leq 3e^t\sqrt{v_1^2+v_2^2}= 3e^t\|\xi-h\xi_0\|.
\]
Moreover,  (\ref{eq;w 1 prime}) implies
\begin{align}\label{eq;mud}
\sup_{\tilde s\in [-1, 1]}|w_1(\tilde s)| \ge\inf_{\tilde s\in [-1, 1]}|w'_1(\tilde s)|\ge  \frac{e^t}{10}\|\xi -h\xi_0\|.
\end{align}
Therefore, applying
Lemma \ref{lem;k m lemma}  to $\kappa=w_1$ with
  $A_1=3e^t \|\xi-h\xi_0\|, A_2=\frac{e^t}{10}\|\xi-h\xi_0\|$ (here we use the notation of (\ref{eq;k m lemma})),
we have
\begin{align*}
\big|\{\tilde s\in [-1, 1]: |w_1(\tilde s)|\le\vep\}\big|\le
\frac{24 \cdot 30\varepsilon}{\sup_{\tilde s\in [-1, 1]}|w_1(\tilde s)|}
\leq \frac{7200\varepsilon}{e^t\|\xi-h\xi_0\|},
\end{align*}
where in the last inequality we use \eqref{eq;mud}.
Therefore, by Lemma \ref{lem;bounded} and $t\ge 30$, we get
\[
\int_{-1}^1\|\xi(\tilde s)-h(\tilde s)\xi_0\|^{-1/2}\le\int_{-1}^1\frac{ds}{|w_1(\tilde s)|^{1/2}}\le \frac{600}{e^{t/2}\|\xi-h\xi_0\|^{1/2}}< \frac{1}{4\|\xi-h\xi_0\|^{1/2}},
\]
which gives \eqref{eq;sum j 2}.

Now we are ready to estimate $\omega(s)$ in (\ref{eq;omega n}). Assume $\zeta_{h, \xi}$ exists for the moment.
Let
\begin{align*}
J_1&=\{\tilde{s}\in [-1, 1]: \zeta(\tilde s)\mbox{ exists and }\zeta(\tilde s)\neq\zeta_{h, \xi}\}, \\
J_2&=\{\tilde{s}\in [-1, 1]: \zeta(\tilde s)\mbox{ exists and }\zeta(\tilde s)=\zeta_{h, \xi}\}.
\end{align*}
By the definition of $\omega$
 (see (\ref{eq;omega n})) and $\alpha_n$ (see (\ref{eq;alpha-n})),   we have
 \begin{align}\label{eq;omega final}
 \begin{split}
\omega(s)
 &\le 2+\int_{J_1}\|\xi(\tilde s)-h(\tilde s)\zeta(\tilde s)\| ^{-1/2} \,d \tilde s+
\int_{J_2}\|\xi(\tilde s)-h(\tilde s)\zeta_{h, \xi}\| ^{-1/2} \,d \tilde s \\
 &< 2+2ne^t \alpha_0(x)+\int_{-1}^1\|\xi(\tilde s)-h(\tilde s)\zeta_{h, \xi}\| ^{-1/2} \,d \tilde s
\qquad\qquad\quad\mbox{ by (\ref{eq;setminus})} \\
&<   2+2ne^t \alpha_0(x)+\frac{1}{4}\|\xi-h\zeta_{h, \xi}\|^{-1/2} \qquad\qquad\qquad\qquad\qquad\mbox{by (\ref{eq;sum j 2})}
\\
&=   2+2ne^t \alpha_0(x)+\frac{1}{4}\alpha_n(x).
\end{split}
 \end{align}
 It is not hard to see using \eqref{eq;setminus} above that the final estimate of (\ref{eq;omega final}) still holds even though
 $\zeta_{h, \xi}$ does not exist.
Therefore,
\begin{align*}
\int_{J}&\beta_n(a_{(m+1)t}u_\varphi(s)\Gamma)\, ds\\
 &=\int_J\int_{-1}^1(\alpha_n+8ne^t\alpha_0)(a_{(m+1)t}u_\varphi(s+e^{-2mt}
\tilde s ))\,d  \tilde s\,d  s
 && \mbox{by (\ref{eq;double})}\\
 &< \int_J  \big(\omega( s)+2ne^t \alpha_0(a_{mt}u_\varphi(s)\Gamma)+16ne^{2t}\big)\,d  s
&& \mbox{by (\ref{eq;alpha 0 final}),  (\ref{eq;omega n})}\\
&< \int_{J}\big(\frac{1}{4}\alpha_n+4ne^t\alpha_0\big)
(a_{mt}u_\varphi(s)\Gamma)\,ds+(16ne^{2t}+2)|J| && \mbox{by (\ref{eq;omega final})}\\
&\le  \frac{1}{2}\int_{J}\beta_n(a_{mt}u_\varphi(s)\Gamma)\, ds+b|J|
&&\mbox{by (\ref{eq;b})} ,
\end{align*}
which completes the proof.
\end{proof}

\subsection{Proof of non-concentration on singular subspaces}\label{sec;singular}
In this section we provide the proof of the crucial Proposition~\ref{prop;main}. We first describe a general strategy to derive quantitative results from an inequality
similar to  (\ref{eq;linear ineq}). We believe that it is of independent interest and can be used
in other places.

Let $Y$ be a locally compact, Hausdorff and second countable topological space.
Let $\phi: I_0\to Y$ be a continuous map  from a compact interval
$I_0$ with positive length.
Let $f: Y\to Y$ be a Borel measurable map.
Let $\mathcal I_0=\{I_0\}$ and for every $m\in\N$
 let $\mathcal I_m$   be a partition of $I_0$ into at
most countably many subintervals with positive length.
   We assume that
  $\{\mathcal I_m\}_{m\in \Z_{\ge 0}}$ is
 a filtration, i.e.~$\mathcal I_{m+1}$ is a refinement of $\mathcal I_m$.
   We use
 $I_m(s)$  where  $s\in I_0$ to denote the unique interval in
 $\mathcal I_m$ that contains $s$.
Let $\beta: Y\to [1, \infty]$ be a measurable function  such that
 there exists $0<a<1$ and $b>0$ with the property that  for any $m\in\Z_{\geq 0}$
 and any
 atom $I_m$ of $\mathcal I_m$
 \begin{align}\label{eq;abstract}
 \int _{I_m}\beta(f^{m+1}\phi(s))\, ds<
a \int_{I_m}\beta (f^m\phi(s))\, ds+b|I_m|.
 \end{align}
 Note that (\ref{eq;abstract}) implies that the same inequality  holds for any $\mathcal I_m$ measurable subset  $I_m$.
 We assume that  $\beta$ satisfies the following Lipschitz properties
for some constant $M\ge 1$:
 for any $ \tilde s\in I_0, m\in \mathbb Z_{\ge 0}$ and  $ s\in I_m(\tilde s)$  one has
 \begin{align}
 \label{eq;lipschitz}
 \beta(f^m\phi( s))&\le     M \beta(f^m\phi(\tilde  s)), \\
\label{eq;lip f}
\beta(f^{m+1}\phi(\tilde s))&\le M  \beta (f^m\phi(\tilde s)).
\end{align}
We also assume that $\beta$ is bounded
on $\phi(I_0)$, i.e. for sufficiently large positive number $ l $
\begin{align}\label{eq;condition l}
\phi(s)\in Y_ l :=\{y\in Y: \beta(y)\le  l \}\quad\text{for all}\quad s\in I_0.
\end{align}
 Let $ l $ be  a positive number such that (\ref{eq;condition l}) holds and
 \begin{align}\label{eq;condition c}
 c:=\Big(a+\frac{b}{ l }\Big)<1.
 \end{align}
 Let $\omega_0:I_0\to \mathbb Z$ be the constant function
 $\omega _0(s)\equiv 0$.
 Next for every $r\in \mathbb N$ we  define the $r$-th return time to $Y_ l $
 inductively by
 \begin{align}\label{eq;omega r s}
 \omega_r(s):=\inf\{m>\omega_{r-1}(s): f^m\phi(\tilde s)\in Y_{ l }\text{ for some  }\tilde s\in I_{m}(s)\}.
 \end{align}
 In the degenerate case where  $\omega_{r-1}(s)=\infty$,  we define $\omega_r(s)=\infty$.
 The reason we use this unnatural-looking definition is that
 the $r$-th return  time  function $\omega_r: I_0\to \mathbb Z_{\ge 0}\cup \{\infty\}$
 is locally constant with respect to the filtration $\{\mathcal I_m\}_{m\in \Z_{\ge 0}}$.  That is, if $\omega_r(s)=m<\infty$, then $\omega_r(\tilde s)=m$ for
 all $\tilde s\in I_m(s)$. This kind of flattening (making a complicated function locally constant) allows us to calculate certain conditional expectations easily.

 \begin{lem}\label{lem;exponential}
  There exist $Q>0$ and $\vartheta>0$ such that for any integer
 $k\ge Q$, $r\in \mathbb Z_{\ge 0}$ and
 $\tilde s\in I_0$ with $m=\omega_r(\tilde s)<\infty$ the measure of the set
 \begin{align}\label{eq;J r k}
 J_{r,k}(\tilde s):=\big\{s\in I_{m}(\tilde s):\omega_{r+1}(s)-\omega_r(\tilde s)\ge k\big\}
 \end{align}
 is less than or equal to $e^{-\vartheta k}|I_{m}(\tilde s)|$.
 \end{lem}
 \begin{rem}\label{rem;exponential july}
 The above lemma says that  for $k$ sufficiently  large the probability of
 $s\in I_m(\tilde s)$ with the next return time to $Y_{ l }$ greater than or equal to $k$
 decays exponentially, i.e.~less than or equal to $e^{-\vartheta k}$.
Note that $\omega_0(\tilde s)\equiv 0$, so Lemma \ref{lem;exponential} implies
that for a.e.~$s\in I_0$ one has $\omega_r(s)<\infty$ for all $r\in \N$.
\end{rem}

 \begin{proof}
 We will show that the lemma holds  for \[
 \vartheta=-\frac{1}{2}\log c\quad\text{and} \quad Q= \frac{2\log c-2\log M}{\log c}.
 \]
Let $r,\tilde s$  and hence $m=\omega_r(\tilde s)$ be fixed.  For simplicity, set
$J_k={J_{r,k}}(\tilde s)$ and let
\[
 A_k:=\int_{J_{k+1}}\beta(f^{m+k}\phi(s))\, ds
\le  \int_{J_{k}}\beta(f^{m+k}\phi(s))\, ds.
\]
Note that the complement of  $J_k$ and hence $J_k$ itself is a disjoint union of intervals in $\mathcal I_{m+k-1}$,
so by  (\ref{eq;abstract}) we have
\begin{align}\label{eq;A k}
A_k\le aA_{k-1}+b|J_{k}|.
\end{align}
For all $s\in J_k$ we have $f^{m+k-1}\phi(s)\not \in Y_{ l }$, hence
 $\beta(f^{m+k-1}\phi(s))>  l $.
Therefore
\begin{align}\label{eq;J k}
|J_k|\le \frac{1}{ l }\int_{J_k}\beta(f^{m+k-1}\phi(s))\, ds=\frac{A_{k-1}}{ l }.
\end{align}
In view of (\ref{eq;A k}) and (\ref{eq;J k}), it follows that
\begin{align*}
A_k\le \Big(a+\frac{b}{ l }\Big)A_{k-1}\le cA_{k-1},
\end{align*}
and hence
\begin{align}\label{eq;A k 0}
A_k\le  c^{k}A_0.
\end{align}
According to (\ref{eq;omega r s}), there is $s_0\in I_m(\tilde s)$ with $f^m\phi(s_0)\in Y_{ l }$.
Since $J_1=I_m(\tilde s)$, by  Lipschitz property (\ref{eq;lipschitz}), we have
\[
A_0=\int_{I_m(\tilde s)}\beta(f^m\phi(s))\, ds\le
M \int_{I_m(\tilde s)}\beta(f^m\phi(s_0))\, ds\le   M  l  |I_m(\tilde s)|.
\]
Therefore,  by (\ref{eq;J k}) and (\ref{eq;A k 0}), we have
\[
|J_k|\le \frac{A_{k-1}}{ l }\le c^{k-1}    M |I_m(\tilde s)|.
\]
It is easy to check that the conclusion holds for  $k\ge Q$.
 \end{proof}

 For a measurable subset $K\subset Y$ and $N\in \mathbb N$ the proportion of
 the trajectory $ \{f^m\phi(s): 0\le m\le N-1\}$
 in $K$ is defined  by
 \[
 \mathcal D_K^N(s):=\frac{1}{N}\#\{0\le m\le N-1: f^m\phi(s)\in K\}.
 \]

 \begin{lem}\label{lem;discrete}
 Let $f:Y\to Y$, $\beta: Y\to [1, \infty]$ be measurable maps  such that
 (\ref{eq;abstract}), (\ref{eq;lipschitz}), (\ref{eq;lip f})
 hold
 for a filtration $\{\mathcal I_m\}_{m\ge 0}$ and $\beta$ is bounded on $\phi(I_0)$.
 Then
for every $0<\varepsilon_0<1$
there exist $0< l _0<\infty$ and $0<c_0<1$ such that
for   $K_0=Y_{ l _0}$  and  $N\in \mathbb N$
\begin{equation}\label{eq;gotoUK}
\left| \left\{
s\in I_0:\mathcal D_{K_0}^N(s)\le 1-\varepsilon_0\}
\right\}
\right|\le c_0^N|I_0|.
\end{equation}
\end{lem}
\begin{proof}
We choose $ l >0$ so that  properties (\ref{eq;condition c}) and (\ref{eq;condition l})  hold.
 Let $\mathcal A_0=\{\emptyset, I_0\}$ be the trivial $\sigma$-algebra on $I_0$.
Inductively, let
$\mathcal A_k$ for $k\ge 1$ be the $\sigma$-algebra on $I_0$ generated by  $\mathcal A_i$  for $i<k$ and the sets
\[
\{I_{\omega_k(s)}(s):s\in I_0,  \omega_k(s)<\infty\}.
\]
It follows that $\{\mathcal A_k\}_{k\geq 0}$ is a filtration
of $\sigma$-algebras on $I_0$
and $\omega_k$ is $\mathcal A_k$
measurable.
Let $I_0'$ be the subset of $I_0$ consisting of elements $s$ such that
$\omega_r(s)\neq \infty$ for any $r\in \N$.
As noted in Remark \ref{rem;exponential july} the set $I_0'$ has full measure in
$I_0$.
By Lemma \ref{lem;exponential}, there exist $Q>0$ and $\vartheta>0$
such that for every $k\ge Q, r\in \Z_{\ge 0} $ and   $s\in I_0$ we get the exponential decay for the measure of the set $J_{r, k}(s)$
in (\ref{eq;J r k}).
So all the conditions of
 \cite[Lemma 3.1]{s}
 are satisfied. It follows that there   exist  $Q_0\in \N, 0< c_0<1$ such that for every $N\in \N$ the measure of the  set
\[
J_N=\left\{s\in I_0':\frac{1}{N} \sum_{r=1}^N  \mathbbm 1 _{Q_0} (\omega_{r}(s)-\omega_{r-1}(s))\ge \varepsilon_0
\right \}
\]
is less than or equal to $ c_0^N|I_0|$.
Here  $ \mathbbm 1 _{Q_0}:\mathbb N\to \mathbb Z_{\ge 0}$ is the truncation of the  identity function    defined
by
\begin{equation*}\label{eq;truncation}
 \mathbbm 1 _{Q_0}(k)=\left\{
\begin{array}{cl}
k & \text{ if } k\ge {Q_0} \\
0 & \text{otherwise.}
\end{array}
\right.
\end{equation*}

We claim that the lemma holds for this $c_0$ and  $ l _0= l  M^{Q_0}$.
We now show that
\begin{align}\label{eq;readable 3}
\left\{
s\in I_0':\mathcal D_{K_0}^N(s)\le 1-\varepsilon_0\right\}\subset J_N,
\end{align}
which will complete the proof.

We fix $s\in I_0'$ with $\mathcal D_{K_0}^N(s)\le 1-\varepsilon_0$.
By \eqref{eq;condition l}, $\phi(s)\in Y_ l \subset Y_{ l _0}=K_0$.
Denote by $0<n_1< \cdots<  n_k<N$ the sequence of consecutive times for which $f^{n_j}\phi(s)\not\in K_0$, i.e.
\begin{align}\label{eq;contradict}
\beta(f^{n_j}\phi(s))> M^{Q_0}l.
\end{align}
Since $\mathcal D_{K_0}^N(s)\le 1-\varepsilon_0$, we have $k/N\ge \vep_0$.
To prove (\ref{eq;readable 3}) it suffices to find
a subset $R$ of $\{1,2,  \ldots,N \}$ such that
\begin{align}
&\omega_r(s)-\omega_{r-1}(s)\ge  Q_0\mbox{ for every }r\in R
\label{eq;readable 1}; \\
&\mbox{for every }1\le j\le k \mbox{ there exists }
 r\in R \mbox{ such that } \omega_{r-1}(s)<n_j< \omega_r(s).
 \label{eq;readable 2}
\end{align}
 Let
\begin{align*}
0\le m_1& =\max \{ m<n_1: m=\omega_r(s) \mbox{ for  some } r\ge 0\}, \\
\infty > m'_1& =\min \{ m>n_1: m =\omega_r(s)\mbox{ for  some } r\ge 0\}.
\end{align*}
Then there exists a positive integer $r$ with   $r\le n_1< N$ such that
 $m_1'=\omega_{r}(s)$ and $m_1=\omega_{r-1}(s)$.
We claim that $n_1-m_1\ge Q_0$.
Indeed, otherwise, by \eqref{eq;lip f} and (\ref{eq;lipschitz}) we have (note that $Q_0\in \N$)
\[
\beta(f^{n_1}\phi(s))=\beta(f^{n_1-m_1}f^{m_1}\phi(s))\le M^{n_1-m_1}\beta(f^{m_1}\phi(s))\le M^{Q_0}l,
\]
which contradicts (\ref{eq;contradict}).
Therefore, (\ref{eq;readable 1}) holds for $r_1=r$.
If  $n_k< m_1'$ then $R=\{r_1\}$ satisfies (\ref{eq;readable 2}) and
we are done. Otherwise we choose the smallest $j$ such that
$ n_{j}>m_1'$.
Then we can repeat the construction to find $r_2=r$ with $r_1<r\le n_{j}$ so that
$\omega_{r-1}(s)<n_{j}<\omega_{r}(s)$ and (\ref{eq;readable 1}) holds.
We continue this procedure until
for $r_i=r$ we have
 $n_k<\omega_{r}(s)$.
It follows directly from the construction that $R=\{r_1, \dots, r_i\}$ satisfies (\ref{eq;readable 1})
and (\ref{eq;readable 2}).
\end{proof}

Let $\{\psi_t\}_{t\in\R}$ be a one-parameter flow on $Y$, see \S~\ref{sec;concept}.
Suppose that $f=\psi_{\tau}$ for some $\tau>0$ and  $\beta$ satisfies the following Lipschitz  property stronger than (\ref{eq;lip f}):
\begin{align}\label{eq;continuous lip}
\beta(\psi_t f^m \phi(s) )\le M\beta(f^m\phi(s))\ \text{ for all }t\in [0, \tau], m\in \mathbb Z_{\ge 0} \text{ and } s\in I_0.
\end{align}
For any $K\subset Y$ and $T>0$ denote by  $\mathcal A_K^T: I_0\to [0, 1]$ the function
\[
\mathcal A_K^T(s):=\frac{1}{T}\int_0^T {\mathbbm 1}_K(\psi_t\phi(s))\, dt.
\]

\begin{lem}\label{lem;continuous}
Let $\{\psi_t\}_{t\in\R}$ be a continuous flow on $Y$ and let $f=\psi_\tau$ for some $\tau>0$.
Let $\beta: Y\to [1, \infty]$ be a measurable function such that
(\ref{eq;abstract}), (\ref{eq;lipschitz})  (\ref{eq;continuous lip})  hold for
 a filtration $\{\mathcal I_m\}_{m\geq 0}$ and $\beta$ is bounded on $\phi(I_0)$.
Then for any $\varepsilon>0$ there exist
$0< l _1<\infty$ and
$0< c_1<1$  such that for $ K=Y_{  l _1}$ and every $T>0$ one has
\begin{equation}\label{eq;booktrain}
\left |\{s\in I_0: \mathcal A_{ K}^T(s)\le 1- \varepsilon    \}\right |\le   c_1^T|I_0|.
\end{equation}
\end{lem}
\begin{proof}
By Lemma \ref{lem;discrete} applied to $f:Y\to Y$ and $\varepsilon_0=\varepsilon/2$, we can find
$0<c_0<1$  and  $ l _0>0$ such that
(\ref{eq;gotoUK}) holds for every positive integer $N$ and ${K_0}=Y_{ l _0}$.
By possibly enlarging $ l _0$ we assume that (\ref{eq;condition l}) holds for $l=l_0$.
We fix  $n\in \N$ so that for $T\ge n\tau$ we have
\begin{equation}\label{eq;easyread}
\left(1-\frac{\tau}{T}\right)\left(1-{\varepsilon_0}\right)\ge 1-\varepsilon.
\end{equation}
 We claim that the lemma holds for     $  l _1= l _0M^{n}$ and $ c_1=c_0^{1/2\tau}$.

For $T\le n\tau$ it follows from (\ref{eq;condition l}) and (\ref{eq;continuous lip})
that  (\ref{eq;booktrain}) holds, since the left hand side of it is always zero.
Now assume  $T> n\tau$.
In view of \eqref{eq;continuous lip},  for every  $m\in \mathbb N$,
\[f^m\phi(s)\in  K_0\ \Longrightarrow\ \psi_t\phi(s)\in K \ \text{ for every }
\ t\in [m\tau, (m+1)\tau].\]
Therefore,  in view of  (\ref{eq;easyread}), we have
\[\mathcal A_{ K}^T(s)\le 1-\varepsilon\ \Longrightarrow \mathcal D_{K_0}^{\lfloor T/\tau\rfloor}(w)\le 1-{\varepsilon_0}.\]
It follow that
\[\left |\{s\in I_0: \mathcal A_{ K}^T(s)\le 1- \varepsilon    \}\right |\leq \left |\{s\in I_0: \mathcal D_{K_0}^{\lfloor T/\tau\rfloor}(w)\le 1-{\varepsilon_0}\}\right |\leq c_0^{\lfloor T/\tau\rfloor}|I_0|\leq c_1^T|I_0|.\]
\end{proof}

\begin{proof}[Proof of Proposition \ref{prop;main}]
We will  apply Lemma \ref{lem;continuous} in the case where $Y=X$, $I_0=I$, $\psi_t=a_t$ and $\phi(s)=u_\varphi(s)\Gamma$ as in the statement of Theorem
\ref{thm;equi}.
We fix $t=\tau$ and $\beta=\beta_n$ such that Lemma \ref{lem;key} holds.
We need to check that for this $\beta$ and $f=\psi_\tau$ the assumptions of  Lemma \ref{lem;continuous} hold.  First we construct a filtration $\{\mathcal I_m\}_{m\in \N}$.
Suppose we have already constructed $\mathcal I_{m-1}$. If $e^{-2mt}\le  |I|$, then
we set $\mathcal I_m=\mathcal I_{m-1}$. Otherwise we  divide each interval  $J$ of
$\mathcal I_{m-1}$ consecutively into intervals of length $e^{-2mt}$, except for the
last one which we allow to have length between $e^{-2mt}$ and $2e^{-2mt}$.
It is easy to see that
(\ref{eq;abstract}), (\ref{eq;lipschitz})  and  (\ref{eq;continuous lip}) follow
 from (\rmnum{1}),   (\rmnum{4}) and  (\rmnum{5})   of Lemma \ref{lem;key} respectively.
 The assumption (\ref{eq;sigma})
 in Proposition \ref{prop;main} and (\rmnum{3}) of
 Lemma \ref{lem;key}   imply that $\beta$ is finite on $\phi(I_0)$.

Therefore we can use Lemma \ref{lem;continuous}  to find  $0< l <\infty$ and $\vartheta>0$ so that
(\ref{eq;proportion}) holds for $K=X_{ l }:=\{x\in X: \beta_n(x)\le  l \}$.
Finally we note that
  (\rmnum{2}) and  (\rmnum{3}) of  Lemma~\ref{lem;key}  imply
 $K$ is a compact subset of $X\setminus X[n]$.
\end{proof}

\section{Proof of the Oseledets genericity along curves}\label{sec:Oseledets_proof}
In this section, we prove Theorems~\ref{thm:Oseledets} and \ref{thm:Oseledetsgen}. Throughout
the section, we denote by  $\mu$ the probability affine measure on
an $SL_2( \mathbb{R})$-orbit closure $\mathcal{M}$.
Suppose that $W\subset H_1(M,\R)$ is a symplectic subspace which is invariant for the cocycle $A^{KZ}:SL_2(\R)\times\mathcal{M}\to GL(H_{1}(M,\R))$.
This gives a symplectic $SL_2(\R)$-invariant bundle $\mathcal{W}\to \mathcal{M}$ which is locally constant.
We deal with the restricted cocycle $A^{KZ}_W:SL_2(\R)\times\mathcal{M}\to GL(W)$.
\begin{rem}\label{rem:reduct}
For every $1\leq p\leq 2d$ ($2d$ is the dimension of $W$) we consider the $p$-th exterior power cocycle $A^{p}:SL_2(\R)\times\mathcal{M}\to GL(\bigwedge^p
W)$ of $A^{KZ}_W$.
It is well known (see e.g.\ \cite{Go-Ma} and \cite{r79}) that in order to prove the Oseledets genericity of some
$x\in \mathcal{M}$ with respect to  $(\mathcal{M},\mu,a_t, A_W^{KZ})$ it suffices to show
an easier condition relating to the top Lyapunov exponent of every
exterior power of the Kontsevich-Zorich cocycle. More precisely, it is enough to show that
for every $1\leq p\leq 2d$ we have
\[\lim_{t\to\infty}\frac{1}{t}\|A^p(a_t,x)\|=\lambda_{top}(A^p).\]
Moreover, the top Lyapunov exponent of the $p$-th exterior power cocycle $A^{p}:SL_2(\R)\times\mathcal{M}\to GL(\bigwedge^p
W)$ is equal to the sum of $p$ greatest Lyapunov exponents of $A_W^{KZ}:SL_2(\R)\times\mathcal{M}\to GL(W)$ (see also \cite{Go-Ma} and \cite{r79}).
\end{rem}
Therefore, the proof of Theorem~\ref{thm:Oseledets} reduces to the following result:

\begin{thm}\label{thm:Oseledets_red}
Assume that the top Lyapunov exponent $\lambda_{top}$ of $A^{p}:SL_2(\R)\times\mathcal{M}\to GL(\bigwedge^p
W)$ is less than $1$.
Suppose that $\varphi:I\to\mathcal{M}$
is a $C^1$-curve which is well approximated by horocycles (in the
sense of Definition~\ref{def:well_approx}) and such that for a.e.\
$s\in I$ the element $\varphi(s)\in\mathcal{M}$ is Birkhoff
generic with respect to  $(\mathcal{M},\mu,a_t)$. Then for a.e.\ $s\in I$ we
have
\[\lim_{t\to\infty}\frac{1}{t}\|A^p(a_t,\varphi(s))\|=\lambda_{top}.\]
\end{thm}


\subsubsection*{Outline of the proof}
The proof of Theorem~\ref{thm:Oseledets_red} is divided into several parts, which we now outline and are presented as separate subsections in the proof. Preliminarily,   one should notice
 that, as observed by Chaika-Eskin in \cite[see the proof of Theorem 1.4]{Es-Ch},
the vector bundle $\bigwedge^p\mathcal{W}\to \mathcal{M}$ has an $SL_2(\R)$-invariant (for the cocycle $A^p$) splitting
$\bigoplus_{i=1}^k\mathcal{V}_i$ which is of class $C^\infty$ and such that the restriction
of the cocycle $A^p$ to the subbundle $\mathcal{V}_i$ (denoted by $A_{\mathcal{V}_i}$) is strongly irreducible.
This result is based on a recent result of Filip \cite{Filip} on semisimplicity and rigidity
of the Kontsevich-Zorich cocycle.
For the definition of strong irreducibility we refer to \cite{Es-Ch}.

The first step of the proof, which is describled in \S~\ref{subsec:goodset}, consists then in  studying the growth of the restricted cocycle $A_{\mathcal{V}}$, where $\mathcal{V}$ is a smooth invariant subbundle. Let us denote by $\lambda_{\mathcal{V}}$ the top Lyapunov exponent of $A_{\mathcal{V}}$ with respect of the measure $\mu$.
By assumption, the image by
$a_t$ of the curve $\varphi$ over small intervals is well
approximated by horocycle arcs. Therefore, it is useful to control
the growth of $A_{\mathcal{V}}$ restricted to short horocycle arcs. In
\S~\ref{subsec:goodset},building on the work of 
\cite{Es-Ch}, we provide so called "good" subsets of $\mathcal{M}$
whose $\mu$-measures are close to $1$ and such that the growth of
$A_{\mathcal{V}}$ is exponential with the exponent
$\lambda_{\mathcal{V}}$ over short horocycle arcs starting from
the "good" set.

In \S~\ref{rem:filtr}, the second step is to describe and prove a tool for showing that for a.e.\ $s\in I$ the Teichm\"uller trajectory of $\varphi(s)$ visits the "good" set with the frequency close to $1$.
This is based on a simple application of the strong law of large numbers for weakly correlated random variables and the Birkhoff genericity of the curve.

An auxiliary result (needed in \S~\ref{sec:Os_pr}) is
presented in \S~\ref{subsec:compKZ}. Here we show that  conveniently chosen  fundamental domains for the cocycle action and  
 the fact that $A^p$ is piecewise constant provide tools for controlling the norm of $A^p$.

Finally, in  \S~\ref{sec:Os_pr}, we use the results proved in the previous steps to prove Theorem~\ref{thm:Oseledets_red}. Since for a.e.\
$s\in I$ the Teichm\"uller trajectory of $\varphi(s)$ visits the
"good" set with high frequency, we can study the growth of
$A_{\mathcal{V}_i}$ for $1\leq i\leq k$ by dividing the trajectory into steps and showing that for most
of these steps the trajectory is close to small horocycle arcs
whose initial points belong to the "good" set. Since for such points the growth of
$A_{\mathcal{V}_i}$ is exponential with the exponent
$\lambda_{\mathcal{V}_i}$, one can then deduce that the growth of
$A^p$ is exponential with the exponent
$\max_{1\leq i \leq k}\lambda_{\mathcal{V}_i}=\lambda_{top}$.

In summary, the exponential growth for most steps finally yields the same growth
along the whole trajectory starting from $\varphi(s)$.

\subsection{Existence of good sets}\label{subsec:goodset}

Let $\mathcal V$ be a smooth $SL_2(\R)$-invariant subbundle of $
\bigwedge ^p\mathcal{W}\to \mathcal M$ such that $A_{\mathcal V}$ is strongly irreducible.
In this subsection  we prove for sufficiently large
real number $t$ the existence of a "good" open subset of
$\mathcal{M}$ whose measure increases to $1$ as $t\to\infty$ and
 for  its element $x$ the growth of the cocycle
$A_{\mathcal{V}}(a_tu(r),x)$ is exponential with the exponent
$\lambda_{\mathcal{V}}$ when $r$ runs through some intervals of
length proportional to $e^{-2t}$. The precise statement is formulated in Lemma~\ref{lem:goodset} below.
 This result helps to control the
growth of $A_{\mathcal{V}}(a_t,\,\cdot\,)$ over short (of length
$\approx e^{-2t}$) horocycle arcs starting from elements of the
"good" set.

To proof of Lemma~\ref{lem:goodset} exploits
two results established by Chaika and Eskin in \cite{Es-Ch} (stated as Proposition~\ref{prop:EC1} and Proposition~\ref{prop:EC2} below). Proposition~\ref{prop:EC2} gives the existence of a "good"
set for the Konstevich-Zorich cocycle generated by a random walk
on $SL_2(\R)$. Proposition~\ref{prop:EC1} is a result on \emph{sublinear tracking} which allows us
to show that for every horocycle arc almost every orbit of the
Teichm\"uller flow starting from such arc can be tracked
with sublinear error by a generic trajectory on $\mathcal{M}$ generated by
the random walk. This result guarantees that
the growth of $A_{\mathcal{V}}$ along the Teichm\"uller and random
walk trajectories are the same and  enables us to pass from the
random walk action on $\mathcal{M}$ to Teichm\"uller trajectories
starting from a horocycle arc.


\smallskip
We state now the main  result of the subsection.

\begin{lem}[Good sets for horocycle arcs]\label{lem:goodset}
There exists $\lambda>0$ such that for every $\vep>0$, $\delta>0$
and $\sigma>0$ there exists $L_0\in\N$ such that for every natural
$L\geq L_0$ there exists an open subset
$E(\vep,\sigma,L)=E_{\mathcal{V}}(\vep,\sigma,L)\subset\mathcal{M}$ with
$\mu(E(\vep,\sigma,L))>1-\delta$ such that for every $x\in
E(\vep,\sigma,L)$, for every interval $[a,b]\ni 0$ of length
$\sigma$ and every $v\in V(x)$ with $\|v\|=1$ there exists a Borel
set $R(v)=R(\vep,L,x,v)\subset[a,b]$ with
$Leb(R(v))>(1-\vep)\sigma$ such that for any $r_0\in R(v)$ and
$r\in[r_0-\sigma e^{-2\lambda L},r_0+\sigma e^{-2\lambda L}]$ we
have
\[\lambda_{\mathcal{V}}-\vep\leq\frac{1}{L\lambda}\log\|A_{\mathcal{V}}(a_{L\lambda}u(r),x)v\|\leq\frac{1}{L\lambda}\log\|A_{\mathcal{V}}(a_{L\lambda}u(r),x)\|\leq\lambda_{\mathcal{V}}+\vep.\]
\end{lem}

Let us now state the two results
established in \cite{Es-Ch} and which the proof of the lemma is based.
Suppose that $\vartheta$ is an
$SO(2,\R)$-bi-invariant probability measure on $SL_2(\R)$ which is
compactly supported and absolutely continuous with respect to the
Haar measure on $SL_2(\R)$.

\begin{prop}[Sublinear tracking, see Lemma~4.1 in \cite{Es-Ch}]\label{prop:EC1}
\begin{sloppypar}
There exists $\lambda=\lambda(\vartheta)>0$ and a measurable map
$\theta:SL_2(\R)^\N\to [0, 2\pi]$ such that
$\theta_*(\vartheta^\N)=\frac{1}{2\pi}Leb_{[0, 2\pi]}$ and for $\vartheta^\N$-a.e.\
$\bar{g}=\{g_n\}_{n\in\N}\in SL_2(\R)^\N$ we have
\end{sloppypar}
\begin{equation}\label{eq:sublintruck_rw}
\lim_{n\to\infty}\frac{1}{n}\log\|a_{n\lambda}r_{\theta(\bar{g})}(g_n\cdots g_1)^{-1}\|=0.
\end{equation}
\end{prop}
\begin{prop}[Good sets for random walks, see Lemma~2.11 in\cite{Es-Ch}] \label{prop:EC2}
For every $\vep>0$ and $\delta>0$ there exists $L_0\in\N$ such
that for every $L\geq L_0$ there exists an open subset
$E_{rw}(\vep,L)\subset\mathcal{M}$ with
$\mu(E_{rw}(\vep,L))>1-\delta$ such that for every $x\in
E_{rw}(\vep,L)$ and for every $v\in \mathcal V(x)$ with $\|v\|=1$ there
exists a measurable set $H(v)=H(\vep,L,x,v)\subset SL_2(\R)^L$ with
$\vartheta^{L}(H(v))>1-\vep$ such that for any
$(g_L,\ldots,g_1)\in H(v)$ we have
\[\lambda_{\mathcal{V}}-\vep\leq\frac{1}{L\lambda}\log\|A_{\mathcal{V}}(g_L\cdots g_1,x)v\|\leq\frac{1}{L\lambda}\log\|A_{\mathcal{V}}(g_1\cdots g_L, x)\|\leq\lambda_{\mathcal{V}}+\vep.\]
\end{prop}
\begin{rem}
Recall that there exist $N\in\N$ and $C\ge 1$ such that
\begin{equation}\label{eq:normcoc}
\|A_{\mathcal{V}}(g,x)\|\leq \|A^p(g,x)\|\leq C\|g\|^N\quad\text{ for all }\quad
x\in \mathcal{M},\ g\in SL_2(\R).
\end{equation}
Since
\[A_{\mathcal{V}}(g_1g_2^{-1},g_2x)A_{\mathcal{V}}(g_2,x)=A_{\mathcal{V}}(g_1,x)\]
for $g_1,g_2\in SL_2(\R)$ and $x\in \mathcal{M}$, it follows that
\begin{equation}\label{eq:diflogcoc1}
\big|\log\|A_{\mathcal{V}}(g_1,x)\|-\log\|A_{\mathcal{V}}(g_2,x)\|\big|\leq
\log C+N\log\|g_1g_2^{-1}\|.
\end{equation}
Moreover, for every $v\in \mathcal  V(x)$ with $\|v\|=1$ we have
\begin{equation}\label{eq:diflogcoc2}
\big|\log\|A_{\mathcal{V}}(g_1,x)v\|-\log\|A_{\mathcal{V}}(g_2,x)v\|\big|\leq
\log C+N\log\|g_1g_2^{-1}\|.
\end{equation}
\end{rem}

\smallskip
 We can now prove the main result  of the section.
 \begin{proof}[Proof of Lemma~\ref{lem:goodset}]
Let us consider measurable map $\tan:[0, 2\pi]\to \R$ and let
\[\kappa:=\tan_*\Big(\frac{1}{2\pi}Leb_{[0, 2\pi]}\Big)=\frac{dr}{\pi(1+r^2)}.\]
Let $\theta:SL_2(\R)^\N\to [0, 2\pi]$ be the map guaranteed by  Proposition~\ref{prop:EC1}.
Taking
$\zeta:SL_2(R)^\N\to\R$ given by $\zeta=\tan\circ\theta$, we
have $\zeta_*(\vartheta^\N)=\kappa$. In view of
Proposition~\ref{prop:EC1}, there exists $\lambda>0$ such that
\begin{equation}\label{eq:sublintruck_horo}
\lim_{n\to\infty}\frac{1}{n\lambda}\log\|a_{n\lambda}u(\zeta(\bar{g}))(g_n\cdots g_1)^{-1}\|=0\
\end{equation}
for $\vartheta^\N$-a.e.\ $\bar{g}\in SL_2(\R)^\N$. Indeed, for
every $\alpha\in [0,2\pi]\setminus\{\pi/2,3\pi/2\}$, $t>0$ and $g\in
SL_2(\R)$ we have
\begin{align*}
\big|\log\|&a_{t}u(\tan(\alpha))g\|-\log\|a_{t}r_\alpha g\|\big|\leq
\log\|a_{t}u(\tan(\alpha))r_{-\alpha}a_{-t}\|\\
&\leq\log\Big\|\begin{pmatrix}
  \cos^{-1}\alpha & 0 \\
  e^{-2t}\sin\alpha & \cos\alpha
\end{pmatrix}\Big\|\leq  \log\sqrt{\cos^{-2}\alpha+1}.
\end{align*}
Together with \eqref{eq:sublintruck_rw}  this yields \eqref{eq:sublintruck_horo}.

Now fix $\vep>0$, $\delta>0$ and $\sigma>0$. Let $\sigma_0:=\frac{\sigma}{\pi(\sigma^2+1)}$. Then $0<\sigma_0\leq 1/(2\pi)$. In view of
Proposition~\ref{prop:EC2}, one can choose a natural number
\[L_1> \frac{2N(\sigma+\log C)}{\vep\lambda}\]
so that for all $L\geq L_1$ we have $\mu(E_{rw}(\sigma_0\vep/2,L))>1-\delta$. Set
\[E(\vep,\sigma,L):=E_{rw}(\sigma_0\vep/2,L)\subset \mathcal{M}.\]
Next  take
\[\vep_1:=\frac{\vep}{2N}-\frac{\sigma+\log
C}{L_1\lambda}>0\] and for $n\geq L_1$ set
\[E'(n,\vep_1):=\Big\{\bar{g}\in SL_2(\R)^{\N}:\frac{1}{n\lambda}\log\|a_{n\lambda}u(\zeta(\bar{g}))(g_n\cdot\ldots\cdot
g_1)^{-1}\|\leq \vep_1\Big\}.\] In view of \eqref{eq:sublintruck_horo}, there exists $L_0\geq L_1$ such that for $n\geq L_0$ we have
$\vartheta^{\N}(E'(n,\vep_1))>1-\sigma_0\vep/2$.

Take $L\geq L_0$ and fix $x\in E(\vep,\sigma,L)$ and $v\in \mathcal V(x)$
with $\|v\|=1$. Take the subset $H(\sigma_0\vep/2,L,x,v)\subset SL_2(\R)^L$
coming from Proposition~\ref{prop:EC2} so that
\[\vartheta^L(H(\sigma_0\vep/2,L,x,v))>1-\sigma_0\vep/2.\]
Setting
\[\widetilde{H}(v):=\{\bar{g}\in E'(L,\vep_1):(g_1,\ldots,g_L)\in H(\sigma_0\vep/2,L,x,v)\}\]
we have $\vartheta^\N(\widetilde{H}(v))>1-\sigma_0\vep/2$. Then there exists a Borel subset $R\subset \R$ with $\kappa(R)>1-\sigma_0\vep$ such that for every $r\in
R$ we have $\zeta^{-1}(\{r\})\cap \widetilde{H}(v)\neq\emptyset$.

Suppose that $r_0\in R$ and take $\bar{g}\in\widetilde{H}(v)$ so that $\zeta(\bar{g})=r_0$. Since $\bar{g}\in E'(L,\vep_1)$, we have
\[\frac{1}{L\lambda}\log\|a_{L\lambda}u(r_0)(\bar{g}_{[1,L]})^{-1}\|\leq \vep_1\quad\text{with}\quad \bar{g}_{[1,L]}=g_L\cdots
g_1.\]
If $|r-r_0|\leq\sigma e^{-2L\lambda}$ then
\begin{align*}
\Big|\log&\|a_{L\lambda}u(r)(\bar{g}_{[1,L]})^{-1}\|-\log\|a_{L\lambda}u(r_0)(\bar{g}_{[1,L]})^{-1}\|\Big|\\&\leq
\log\|a_{L\lambda}u(r-r_0)a_{-L\lambda}\|=\log\|u(e^{2L\lambda}(r-r_0))\|\\&\leq \log(1+e^{2L\lambda}|r-r_0|)\leq \log(1+\sigma)\leq\sigma,
\end{align*}
so \[ \frac{1}{L\lambda}\log\|a_{L\lambda}u(r)(\bar{g}_{[1,L]})^{-1}\|\leq \vep_1+\frac{\sigma}{L\lambda}.
\]
In view of \eqref{eq:diflogcoc1} and \eqref{eq:diflogcoc2}, both
\begin{align*}
\frac{1}{L\lambda}\big|\log\|A_{\mathcal{V}}(a_{L\lambda}u(r),x)\|-\log\|A_{\mathcal{V}}(\bar{g}_{[1,L]},x)\|\big|\\
\frac{1}{L\lambda}\big|\log\|A_{\mathcal{V}}(a_{L\lambda}u(r),x)v\|-\log\|A_{\mathcal{V}}(\bar{g}_{[1,L]},x)v\|\big|
\end{align*}
are bounded by
\[\frac{\log C}{L\lambda}+N\Big(\vep_1+\frac{\sigma}{L\lambda}\Big)
\leq N\Big(\vep_1+\frac{\sigma+\log C }{L_1\lambda}\Big)= \vep/2.\]
As $(g_1,\ldots,g_L)\in H(\sigma_0\vep/2,L,x,v)$ we have
\[\lambda_{\mathcal{V}}-\vep/2\leq\frac{1}{L\lambda}\log\|A_{\mathcal{V}}(\bar{g}_{[1,L]},x)v\|\leq\frac{1}{L\lambda}\log\|A_{\mathcal{V}}(\bar{g}_{[1,L]}, x)\|\leq\lambda_{\mathcal{V}}+\vep/2.\]
It follows that, if $r_0\in R$ and $|r-r_0|\leq\sigma e^{-2L\lambda}$ then
\[\lambda_{\mathcal{V}}-\vep\leq\frac{1}{L\lambda}\log\|A_{\mathcal{V}}(a_{L\lambda}u(r),x)v\|\leq\frac{1}{L\lambda}\log\|A_{\mathcal{V}}(a_{L\lambda}u(r), x)\|\leq\lambda_{\mathcal{V}}+\vep.\]

Finally we need to show that for any interval $[a,b]\ni 0$ of length $\sigma$ taking $R(\vep,L,x,v):=R\cap [a,b]$ we have $Leb([a,b]\setminus
R(\vep,L,x,v))<\sigma\vep$. It follows directly from the fact that \[\kappa([a,b]\setminus R(\vep,L,x,v))\leq \kappa(\R\setminus R)<\sigma_0\vep=\frac{\sigma}{\pi(\sigma^2+1)}\vep\] and the
density of $\kappa$ is $\frac{1}{\pi(r^2+1)}\geq \frac{1}{\pi(\sigma^2+1)}$ on $[a,b]$. This completes the proof.
\end{proof}

\subsection{Curve partitions and weak law of large numbers} \label{rem:filtr}
Let $E\subset\mathcal{M}$ be a "good" set for the time $t=L\lambda$. Then for every natural $n$ the interval $I$ can be partitioned into intervals of length $\approx e^{-2tn}$ such that
for most such intervals their images by $a_{nt}\circ \varphi:I\to\mathcal{M}$ are fully contained in $E$. This gives a sequence of partitions of $I$. In
this subsection, applying a strong law of large numbers for weakly correlated random variables, we present an abstract setting for proving that for a.e.\
$s\in I$ the frequency of $\{a_{nt}\varphi(s)\}_{n\geq 1}$ in $E$ is approximately $\mu(E)$.

Take $I=[0,1]$, $\vep>0$, $0<\alpha,\sigma<1$.
For every $n\geq 0$ let $d_n=\lfloor 1/(\sigma\alpha^n)\rfloor$ and denote by $\mathcal{I}_n$ the partition
of the interval $[0,d_n\sigma\alpha^n]\subset I$ into $d_n$ intervals of length $\sigma\alpha^n$.
 Denote by $\mathcal{F}_n$ the ring of sets generated by
$\mathcal{I}_n$, i.e.\ each element of $\mathcal{F}_n$ is the union of some intervals from $\mathcal{I}_{n}$.

Assume that $1-\vep<k\alpha<1$ for some natural number $k$. Let us consider a sequence $\{A_n\}_{n\geq 0}$ of subsets of $I$ such that for every $n\geq 0$ we have
$A_n\in\mathcal{F}_{n+1}$ and for every interval $I_n\in\mathcal{I}_n$ the set $A_n\cap I_n$ is the union of exactly $k$ intervals from $\mathcal{I}_{n+1}$. Then
\[k\alpha(1-\sigma\alpha^{n})\leq Leb(A_n)\leq k\alpha.\]
Therefore, for $n\geq 0$ and $l\geq 1$ we have
\[|Leb(A_n\cap A_{n+l})-Leb(A_n)Leb(A_{n+l})|\leq (2+\sigma)\alpha^{l-1}.\]
Then, by the strong law of large numbers for weakly correlated random variables (see e.g.\ Corollary~4 in \cite{Ly}),
\[\frac{1}{n}\sum_{j=0}^{n-1}({\mathbbm 1}_{A_j}(s)-Leb(A_j))\to 0\text{ for a.e.\ }s\in I,\]
so
\[\lim_{n\to\infty}\frac{1}{n}\#\{0\leq j\leq n-1:s\in A_j\}=k\alpha>1-\vep\text{ for a.e.\ }s\in I.\]

\subsection{Fundamental domains and norm of $A^p$.}\label{subsec:compKZ}
The Konstevich-Zorich cocycle is piecewise constant  in
time and space. This fact helps to control the norm of
$A^p$ for a fixed time $t$ and for nearby points
in $\mathcal{M}$. The properties of the choice of fundamental domain and the Lemma proved in this subsection will be used to
control the norm of $A^p$. This will be useful in the proof of
Theorem~\ref{thm:Oseledets_red}.

Let $\mathcal{T}$ be the Teichm\"uller space for the orbit closure $\mathcal{M}$ and
let $d$ be a metric on $\mathcal{T}$ satisfying \eqref{assum:b2}
and \eqref{assum:b4}. Denote by  $\pi:\mathcal{T}\to\mathcal{M}$
the natural projection. Fix a fundamental domain $\Delta\subset
\mathcal{T}$ for the action of the discrete group
$\Gamma:=\Gamma(M,\Sigma)$ on $\mathcal{T}$ so that
\begin{equation}\label{zal:delta}
\mu(\pi(\partial \Delta))=0.
\end{equation}
For every $\sigma>0$ let
\[\Delta_\sigma=\{\tilde{x}\in \Int\Delta:dist(\tilde{x},\partial \Delta)>\sigma\}.\]
Then the set $\Delta_\sigma\subset \mathcal{T}$ is open and
$\bigcup_{\sigma>0}\Delta_\sigma=\Int\Delta$. Therefore
\[\pi(\Delta_\sigma)\subset \mathcal{M}\text{ is open and }\lim_{\sigma\to 0}\mu(\pi(\Delta_\sigma))=\mu(\pi(\Int\Delta))=1.\]

\begin{lem}\label{lem:rownosckocykli}
Let $\tilde{x},\tilde{y}\in \mathcal{T}$, $x=\pi(\tilde{x}),y=\pi(\tilde{y})\in \mathcal{M}$ be such that
\begin{itemize}
\item[(i)] $\tilde x, \tilde y\in \Delta$; or,
\item[(ii)] $x\in\pi(\Delta_\sigma)$ and $d (\tilde{x},\tilde{y})\leq\sigma$.
\end{itemize}
If $g\in SL_2(\R)$ is such that $g\cdot x\in\pi(\Delta_\sigma)$
and $d (g\cdot \tilde{x},g\cdot \tilde{y})\leq\sigma$  then
$A^p(g,x)=A^p(g,y)$. Moreover, if
$x\in\pi(\Delta_\sigma)$ and $d (\tilde{x},g\cdot\tilde{x})\leq
\sigma$ for some $g\in SL_2(\R)$ then $A^p(g,x)=Id$.
\end{lem}

\begin{proof}
Choose $\gamma\in \Gamma$ so that $\gamma(\tilde{x})\in\Delta$.
By assumption (i) or (ii) ($d(\gamma(\tilde{x}),\gamma(\tilde{y}))=d
(\tilde{x},\tilde{y})\leq\sigma$), we also have
$\gamma(\tilde{y})\in\Delta$. Let $\widehat{\gamma}\in \Gamma$ be
such that $\widehat{\gamma}(g\cdot \tilde{x})\in\Delta$. Since
$g\cdot x\in\pi(\Delta_\sigma)$, it follows that
$\widehat{\gamma}(g\cdot \tilde{x})\in\Delta_{\sigma}$. Together
with
\[d (\widehat{\gamma}(g\cdot \tilde{x}),\widehat{\gamma}(g\cdot \tilde{y}))=d (g\cdot \tilde{x},g\cdot \tilde{y})\leq\sigma\] this gives
$\widehat{\gamma}(g\cdot \tilde{y})\in\Delta$. In summary
\[\gamma(\tilde{x}),\gamma(\tilde{y})\in\Delta\text{ and }
(\widehat{\gamma}\circ\gamma^{-1})\big({\gamma}(g\cdot
\tilde{x})\big),(\widehat{\gamma}\circ\gamma^{-1})\big({\gamma}(g\cdot
\tilde{y})\big)\in\Delta.\] It follows that
$A^p(g,x)=\bigwedge^p((\widehat{\gamma}\circ\gamma^{-1})_*)=A^p(g,y)$.

Similarly, if $x\in\pi(\Delta_\sigma)$ and $d
(\tilde{x},g\cdot\tilde{x})\leq \sigma$ then
$\gamma(\tilde{x}),\gamma(g\cdot\tilde{x})\in\Delta$. Therefore,
$A^p(g,x)=\bigwedge^p((Id)_*)=Id$.
\end{proof}

\subsection{Oseledets generiticy for almost every point on the curve}\label{sec:Os_pr}
In this last subsection, we conclude the proof: we exploit the Birkhoff generic behaviour along the curve and the existence of good sets to show that for most arcs which are quantitatively well approximated by a horocycle, thanks to the weak law of the large numbers, the time spent in each good set is large enough to guarantee typical Oseledets behaviour for almost every point in the curve.

\smallskip
Let us emphasize that most of the arguments of the proof does not need
the assumption $\lambda_{top}<1$ and this fact  will be crucially used for the proof of Theorem~\ref{thm:Oseledetsgen} in \S~\ref{sec:no_assumption}. The only
part of the proof for  which the assumption $\lambda_{top}<1$ is necessary is Step 8 in the proof presented below.

\begin{proof}[Proof of Theorem \ref{thm:Oseledets_red}]
Since  the proof is long, we divided it into several steps and provided for each step a title, in order to suggest to the reader what is the content of the step.
\smallskip

\noindent \textbf{Step 0: Choice of fundamental domain.}
For any $\tilde{x}\in \mathcal{T}$ one can build a fundamental
domain $\Delta(\tilde{x})\subset \mathcal{T}$ satisfying
\eqref{zal:delta} so that $\tilde{x}\in\Int\Delta(\tilde{x})$. It
follows that the interval $I$ can be covered by a finite family of
closed intervals such that for every such interval $J\subset I$
there is a fundamental domain $\Delta(J)\subset \mathcal{T}$
satisfying \eqref{zal:delta} so that
$\widetilde{\varphi}(J)\subset\Int\Delta(J)$. Therefore, without loss
of generality we can assume that $I=[0,1]$ and
$\widetilde{\varphi}(I)\subset\Int\Delta$  for some fundamental domain
$\Delta\subset \mathcal{T}$ satisfying \eqref{zal:delta}.
\medskip

\noindent \textbf{Step 1: Splitting into strongly irreducible subbundles.}
Let us consider a $C^\infty$ splitting $\bigoplus_{i=1}^k\mathcal{V}_i$ of the bundle
$\bigwedge^p \mathcal{W}\to \mathcal{M}$ such that each cocycle $A_{\mathcal{V}_i}$ is strongly irreducible.
Then
\[\lambda_{top}=\max_{1\leq i\leq k}\lambda_{\mathcal{V}_i}.\]
For every $x\in \mathcal{M}$ let us consider the isomorphism
\[W\ni w\mapsto (w_1(x),\ldots,w_k(x))\in \mathcal V_1(x)\times\ldots\times \mathcal V_k(x)\ \text{ with }\ w=w_1(x)+\ldots+w_k(x).\]
Since $\varphi$
is of class $C^1$ and each $\mathcal{V}_i$ is a $C^\infty$-subbundle,
there exists $C'\geq 1$ such that
\begin{equation}\label{neq:norm}
\sum_{1\leq i\leq k}\|w_i(\varphi(s))\|\leq C'\|w\| \ \text{ for every }\ s\in I\ \text{ and }\ w\in W.
\end{equation}

For every $1\leq i\leq k$ we choose a $C^1$-curve $I\ni s\mapsto v_i(s)\in
\mathcal{V}_i(\varphi(s))$ over $\varphi$
such that $\|v_i(s)\|=1$ for any $s\in I$.
Then there exists $l>0$ such that
\begin{equation}\label{lipforv}
\|v_i(s)-v_i(s')\|\leq l|s-s'|\ \text{ for all }\ s,s'\in I,\; 1\leq i\leq k.
\end{equation}
If all subbundles $\mathcal{V}_i$ are locally constant then we can deal with constant maps $v_i\in
\mathcal{V}_i(\varphi(s))$. This observation will be used in \S~\ref{sec:no_assumption}.
\medskip

\noindent \textbf{Step 2: Construction of a "good" set.}
Fix $0<\vep<1$.
Next choose $0<\sigma<1$ so that
$\mu(\pi(\Delta_{2\sigma}))>1-\vep/6$. By Lemma~\ref{lem:goodset},
there exist $\lambda>0$ and  $L\in\N$ such that
\begin{equation}\label{condL}
L>\frac{8\max\{l\sigma,C'\}/\vep+\log(C(\varphi)(1+\sigma^\rho))}{\lambda}
\end{equation}
and
\[\mu(E_{\mathcal{V}_i}(\vep/4k,\sigma,L))>1-\vep/6k\ \text{ for }\ 1\leq i\leq k,\]
where $C(\varphi)>0$, $\rho\in\N$ are the factors derived from the well approximation of $\varphi$ by horocycles (see \eqref{assum:phi}).
Therefore the set
\[U_\vep:=\bigcap_{i=1}^kE_{\mathcal{V}_i}(\vep/4k,\sigma,L)\cap\pi(\Delta_{2\sigma})\cap a_{-L\lambda}\pi(\Delta_{2\sigma})\subset \mathcal{M}\]
is  open  and $\mu(U_\vep)>1-\vep/2$.
\medskip

\noindent \textbf{Step 3: Construction of "good" subintervals of $I$.}
Let us consider the sequences  $\{\mathcal{I}_n\}_{n\geq 0}$, $\{\mathcal{F}_n\}_{n\geq 0}$ of families of subsets in $I $ described in \S~\ref{rem:filtr}
 for $\alpha:=e^{-2L\lambda}$. Then the length of each interval from $\mathcal{I}_n$ is $\sigma e^{-2L\lambda n}$.

For $n\geq 0$ let us consider any interval  $J=[m\sigma
e^{-2L\lambda n},(m+1)\sigma e^{-2L\lambda n}]\in \mathcal{I}_n$
such that $a_{L\lambda n}\varphi(J)\cap U_\vep\neq \emptyset$.
Next fix $s_0=s_{0,n}=s_0(J)\in J$  with $a_{L\lambda n}\varphi(s_0)\in
U_\vep$. Then $a_{L\lambda n}\varphi(s_0)\in
E_{\mathcal{V}_i}(\vep/4k,\sigma,L)\cap\pi(\Delta_{2\sigma})$ for $1\leq i\leq k$. Let
\begin{equation}\label{eq:eqcoc0}
v_{n,i}:=\frac{A_{\mathcal{V}_i}(a_{L\lambda n},\varphi(s_0))v_i(s_0)}{\|A_{\mathcal{V}_i}(a_{L\lambda n},\varphi(s_0))v_i(s_0)\|}\in \mathcal V_i(a_{L\lambda
n}\varphi(s_0))\ \text{ for }\ 1\leq i\leq k.
\end{equation}
Since $a_{L\lambda n}\varphi(s_0)\in E_{\mathcal{V}_i}(\vep/4k,\sigma,L)$, by Lemma~\ref{lem:goodset}, there exists a Borel set
\begin{align*}
R(v_{n,i})\subset\big[s_0e^{2L\lambda n}-(m+1)\sigma,s_0e^{2L\lambda n}-m\sigma\big]\text{ with }Leb(R(v_{n,i}))>\Big(1-\frac{\vep}{4k}\Big)\sigma
\end{align*}
such that for any $r_0\in R(v_{n,i})$ and $r\in[r_0-\sigma e^{-2\lambda L},r_0+\sigma e^{-2\lambda L}]$ we have
\begin{align*}
\lambda_{\mathcal{V}_i}-\frac{\vep}{4}&\leq\frac{1}{L\lambda}\log\|A_{\mathcal{V}_i}(a_{L\lambda}u(r),a_{L\lambda n}\varphi(s_0))v_{n,i}\|\\
&\leq\frac{1}{L\lambda}\log\|A_{\mathcal{V}_i}(a_{L\lambda}u(r),a_{L\lambda n}\varphi(s_0))\|\leq\lambda_{\mathcal{V}_i}+\frac{\vep}{4}
\end{align*}
for $1\leq i\leq k$.
Setting $\widetilde{R}_n:=s_0-e^{-2L\lambda n}\bigcap_{i=1}^kR(v_{n,i})\subset J$ we have $Leb(\widetilde{R}_n)>(1-\vep/4)|J|$.
Suppose that $s_1\in\widetilde{R}_n$ and $|s-s_1|\leq\sigma e^{-2L\lambda (n+1)}$. Then
$r_0:=(s_0-s_1)e^{2L\lambda n}\in R(v_{n,i})$
for every $1\leq i\leq k$  and $|e^{2L\lambda n}(s_0-s)-r_0|<\sigma e^{-L\lambda}$. Therefore,
\begin{align}\label{neq:firlog}
\begin{aligned}
\lambda_{\mathcal{V}_i}-\frac{\vep}{4}&\leq\frac{1}{L\lambda}\log\|A_{\mathcal{V}_i}(a_{L\lambda}u(e^{2L\lambda n}(s_0\!-\!s)),a_{L\lambda n}\varphi(s_0))v_{n,i}\|\\
&\leq\frac{1}{L\lambda}\log\|A_{\mathcal{V}_i}(a_{L\lambda}u(e^{2L\lambda n}(s_0\!-\!s)),a_{L\lambda
n}\varphi(s_0))\|\leq\lambda_{\mathcal{V}_i}+\frac{\vep}{4}
\end{aligned}
\end{align}
for every $1\leq i\leq k$.
It follows that  for any interval $J'\in \mathcal{I}_{n+1}$ with $J'\cap \widetilde{R}_n\neq\emptyset$
the inequalities \eqref{neq:firlog} hold for every $s\in J'$. Since $Leb(\widetilde{R}_n)>(1-\vep/4)|J|$,
the number $\tilde k$ of such intervals included in $J$ satisfies
\[(\tilde k+2)\sigma e^{-2L\lambda (n+1)}\geq Leb(\widetilde{R}_n)>(1-\vep/4)\sigma e^{-2L\lambda n}.\]
Therefore, $\tilde k\geq k_0:=\lceil(1-\vep/4)e^{2L\lambda}-2\rceil$ with
\[k_0e^{-2L\lambda}\geq 1-\frac{\vep}{4}-2e^{-2L\lambda}>1-\frac{\vep}{2}.\]
The last inequality follows directly from \eqref{condL}.
Let $A_n(J)$ be the union of exactly $k_0$ intervals $J'\in \mathcal{I}_{n+1}$ with $J'\subset J$ and $J'\cap \widetilde{R}_n\neq\emptyset$.
In summary, since \eqref{neq:firlog} holds for every $s\in A_n(J)$, by \eqref{eq:eqcoc0}, we have the following:
\begin{align}\label{anj}
\begin{aligned}
&\text{for any $J\in \mathcal{I}_n$ with $a_{L\lambda n}\varphi(J)\cap U_\vep\neq\emptyset$}\\
&\text{there exists $s_{n}=s_n(J)\in J$ so that $a_{L\lambda n}\varphi(s_n)\in U_\vep$ and }\\
&\lambda_{\mathcal{V}_i}-\frac{\vep}{4}\leq\frac{1}{L\lambda}\log\frac{\|A_{\mathcal{V}_i}\big(a_{L\lambda}u(e^{2L\lambda n}(s_{n}-s))a_{L\lambda n},\varphi(s_n)\big)v_i(s_n)\|}{\|A_{\mathcal{V}_i}\big(a_{L\lambda n},\varphi(s_{n})\big)v_i(s_n)\|}\\
&\quad\quad\quad\leq\frac{1}{L\lambda}\log\|A_{\mathcal{V}_i}(a_{L\lambda}u(e^{2L\lambda n}(s_n-s)),a_{L\lambda
n}\varphi(s_{n}))\|\leq\lambda_{\mathcal{V}_i}+\frac{\vep}{4}\\
&\text{hold for every $s\in A_n(J)$ and $1\leq i\leq k$.}
\end{aligned}
\end{align}
\medskip

\noindent \textbf{Step 4: Frequency of a typical point is "good" intervals.}
Note that we have defined the set $A_n(J)$ for every $J\in \mathcal{I}_n$ with $a_{L\lambda n}\varphi(J)\cap U_\vep\neq\emptyset$. We can also do it if  $a_{L\lambda n}\varphi(J)\cap U_\vep=\emptyset$. Then $A_n(J)$ is the union of some $k_0$ intervals
in $\mathcal{I}_{n+1}$ included in $J$.
Finally set
\[A_n:=\bigcup_{J\in \mathcal{I}_{n}}A_n(J)\in \mathcal{F}_{n+1}.\]
Since $k_0e^{-2L\lambda}>1-\vep/2$, by \S~\ref{rem:filtr},
\[\lim_{n\to\infty}\frac{1}{n}\#\{0\leq j\leq n-1:s\in A_j\}=k_0\alpha=k_0e^{-2L\lambda}>1-\vep/2\text{ for a.e.\ }s\in I.\]
\medskip

\noindent \textbf{Step 5: Frequency of the orbit of $\varphi(s)$ in the "good" set.}
On the other hand, $\varphi(s)\in\mathcal{M}$ is Birkhoff generic
for a.e.\ $s\in I$. Therefore for a.e.\ $s\in I$ we have
\[\liminf_{n\to\infty}\frac{1}{n}\#\{0\leq j\leq n-1:a_{L\lambda j}\varphi(s)\in U_\vep\}\geq \mu(U_\vep)>1-\vep/2.\]
It follows that
\begin{equation}\label{eq:densprop}
\liminf_{n\to\infty}\frac{1}{n}\#\{0\leq j\leq n-1:a_{L\lambda j}\varphi(s)\in U_\vep, s\in A_j\}>1-\vep
\end{equation}
for a.e.\ $s\in I$. Therefore for every $s\in I$ satisfying \eqref{eq:densprop} there exists $n_0(s)$ so that  setting
\[D(s):=\{j\in\N:a_{L\lambda j}\varphi(s)\in U_\vep, s\in A_j, j\ge 2\rho+2\}\] we have
\begin{equation}\label{eq:ds}
\# (D(s)\cap[0,n-1])>(1-\vep)n\quad\text{ for }\quad n\geq n_0(s).
\end{equation}
\medskip

\noindent \textbf{Step 6: Growth control for $A^p$ in one step.}
Take any $s\in I$ satisfying \eqref{eq:densprop} and  any $n\in D(s)$. Then $n\ge 2(\rho+1)$. Choose
$J\in\mathcal{I}_n$ containing $s$. Then $a_{L\lambda
n}\varphi(J)\cap U_\vep\neq\emptyset$
and $s\in A_n\cap J=A_n(J)$.
In view of \eqref{anj}, it follows that there exists $s_{n}:=
s_{n}(J)\in J$ such that
\begin{equation}\label{eq:s0in}
a_{L\lambda n}\varphi(s_n)\in U_\vep\subset \pi(\Delta_{2\sigma}).
\end{equation}
and
\begin{align}\label{anj1}
\begin{aligned}
\lambda_{\mathcal{V}_i}-\frac{\vep}{4}&\leq\frac{1}{L\lambda}\log\frac{\|A^p\big(a_{L\lambda}u(e^{2L\lambda n}(s_{n}\!-\!s))a_{L\lambda n},\varphi(s_n)\big)v_i(s_n)\|}{\|A^p
\big(a_{L\lambda n},\varphi(s_{n})\big)v_i(s_n)\|}\\
&\leq \frac{1}{L\lambda}\log\|A_{\mathcal{V}_i}\big(a_{L\lambda}u(e^{2L\lambda n}(s_n\!-\!s)),a_{L\lambda
n}\varphi(s_{n})\big)\|\leq\lambda_{\mathcal{V}_i}+\frac{\vep}{4}
\end{aligned}
\end{align}
for every $1\leq i\leq k$.
Moreover,
$a_{L\lambda n}\varphi(s)\in U_\vep$ implies
\begin{equation}\label{indeltasigma}
a_{L\lambda n}\varphi(s),a_{L\lambda (n+1)}\varphi(s)\in\pi(\Delta_{2\sigma}).
\end{equation}
As the points $s$, $s_{n}$ belong to $J\in\mathcal{I}_n$, we have
$|e^{2L\lambda n}(s_{n}\!-\!s)|\leq \sigma$. In view of
\eqref{assum:phi}, it follows that for $m=n$ or $m=n+1$, we have
\begin{align}\label{eq:eqcoc1}
\begin{aligned}
d &\big(a_{L\lambda m}\cdot\widetilde{\varphi}(s), u(e^{2L\lambda
m}(s_{n}-s))\cdot a_{L\lambda
m}\cdot\widetilde{\varphi}(s_{n})\big)\\&\leq C(\varphi)|e^{2L\lambda m}(s_{n}\!-\!s)|(1+e^{2\rho L\lambda m}|s_{n}\!-\!s|^\rho)e^{-L\lambda m}\\
&\leq C(\varphi)\sigma e^{2L\lambda (m-n)}(1+\sigma^\rho e^{2\rho L\lambda (m-n)})e^{-L\lambda m}\\
&\leq C(\varphi)\sigma (1+\sigma^\rho)e^{L\lambda(2\rho-n)}\leq C(\varphi)\sigma (1+\sigma^\rho)e^{-L\lambda
}\leq\sigma.
\end{aligned}
\end{align}
The last two inequality follows directly from $2(\rho +1)\leq n$  and \eqref{condL}.
By \eqref{eq:eqcoc1} applied to $m=n$ and $m=n+1$, we have
\begin{equation}\label{eq:dlan}
d \big(a_{L\lambda n}\cdot\widetilde{\varphi}(s),u(e^{2L\lambda n}(s_{n}\!-\!s))\cdot a_{L\lambda n}\cdot\widetilde{\varphi}(s_{n})\big)\leq \sigma
\end{equation}
and\
\[d \big(a_{L\lambda}\cdot a_{L\lambda n}\cdot\widetilde{\varphi}(s),a_{L\lambda }\cdot u(e^{2L\lambda n}(s_{n}\!-\!s))\cdot a_{L\lambda n}\cdot\widetilde{\varphi}(s_{n})\big)\leq \sigma.\]
Since $a_{L\lambda n}\varphi(s)$ and $a_{L\lambda}a_{L\lambda n}\varphi(s)$ belong to $\pi(\Delta_{2\sigma})$ (cf.~\eqref{indeltasigma}), by the first assertion of Lemma~\ref{lem:rownosckocykli},
it follows that
\begin{equation}\label{eq:eqcoc2}
A^p\big(a_{L\lambda},a_{L\lambda
n}\varphi(s)\big)=A^p\big(a_{L\lambda},u(e^{2L\lambda
n}(s_{n}\!-\!s)) a_{L\lambda n}{\varphi}(s_{n})\big).
\end{equation}
In view of \eqref{assum:b4}, we have
\begin{equation}\label{in:b4}
d \big(a_{L\lambda n}\widetilde{\varphi}(s_{n}),u(e^{2L\lambda
n}(s_{n}\!-\!s))a_{L\lambda n}\widetilde{\varphi}(s_{n})\big)\leq
|e^{2L\lambda n}(s_{n}\!-\!s)|\leq \sigma.
\end{equation}
Since $a_{L\lambda n}\varphi(s_n)\in\pi(\Delta_{2\sigma})$ (cf.~\eqref{eq:s0in}), by the second assertion  of
Lemma~\ref{lem:rownosckocykli}, it follows that
\[A^p\big(u(e^{2L\lambda n}(s_{n}\!-\!s)),a_{L\lambda
n}\varphi(s_{n})\big)=Id.\] Therefore, by \eqref{eq:eqcoc2},
\begin{align}\label{eq:eqcoc3}
\begin{aligned}
A^p&\big(a_{L\lambda}u(e^{2L\lambda n}(s_{n}\!-\!s)),a_{L\lambda n}{\varphi}(s_{n})\big)\\
&=A^p\big(a_{L\lambda},u(e^{2L\lambda n}(s_{n}\!-\!s)) a_{L\lambda n}{\varphi}(s_{n})\big)
A^p\big(u(e^{2L\lambda n}(s_{n}\!-\!s)),a_{L\lambda n}\varphi(s_{n})\big)\\
&= A^p\big(a_{L\lambda},a_{L\lambda n}\varphi(s)\big).
\end{aligned}
\end{align}
In view of  \eqref{eq:dlan} and \eqref{in:b4}, we also have
\begin{align*}
d \big(a_{L\lambda n}\widetilde{\varphi}(s),a_{L\lambda
n}\widetilde{\varphi}(s_{n})\big)\leq & d \big(a_{L\lambda
n}\widetilde{\varphi}(s),u(e^{2L\lambda n}(s_{n}\!-\!s))a_{L\lambda
n}\widetilde{\varphi}(s_{n})\big)\\ &+d \big(a_{L\lambda
n}\widetilde{\varphi}(s_{n}),u(e^{2L\lambda
n}(s_{n}\!-\!s))a_{L\lambda n}\widetilde{\varphi}(s_{n})\big) \leq
2\sigma.
\end{align*}
Since $\tilde \varphi(s), \tilde \varphi(s_{n})\in \Delta$
and $a_{L\lambda
n}\varphi(s)\in\pi(\Delta_{2\sigma})$, by the first assertion of
Lemma~\ref{lem:rownosckocykli}, it follows that
\begin{equation}\label{eq:eqcoc4}
A^p\big(a_{L\lambda
n},\varphi(s)\big)=A^p\big(a_{L\lambda n},{\varphi}(s_{n})\big).
\end{equation}
In view of \eqref{eq:eqcoc3}, this gives
\begin{align}\label{eq:eqcoc5}
\begin{aligned}
A^p&\big(a_{L\lambda}u(e^{2L\lambda n}(s_{n}\!-\!s))a_{L\lambda n},{\varphi}(s_{n})\big)\\
&=A^p\big(a_{L\lambda}u(e^{2L\lambda n}(s_{n}\!-\!s)),a_{L\lambda n}{\varphi}(s_{n})\big)A^p\big(a_{L\lambda n},{\varphi}(s_{n})\big)\\
&= A^p\big(a_{L\lambda},a_{L\lambda n}\varphi(s)\big)A^p\big(a_{L\lambda n},\varphi(s)\big)= A^p\big(a_{L\lambda (n+1)}\varphi(s)\big).
\end{aligned}
\end{align}
In summary, \eqref{anj1} combined with \eqref{eq:eqcoc3}, \eqref{eq:eqcoc4} and \eqref{eq:eqcoc5} yields that if $s\in I$ satisfies \eqref{eq:densprop} then for every
or every $n\in D(s)$ there exists $s_n\in I$ with $|s_n-s|\leq\sigma e^{-nL\lambda}$ so that
\begin{align}\label{neq:fundav}
\begin{aligned}
\lambda_{\mathcal{V}_i}-\frac{\vep}{4}&\leq\frac{1}{L\lambda}\log\frac{\|A^p\big(a_{L\lambda (n+1)},\varphi(s)\big)v_i(s_{n})\|}{\|A^p\big(a_{L\lambda n},\varphi(s)\big)v_i(s_{n})\|}\\
&\leq\frac{1}{L\lambda}\log\|A^p\big(a_{L\lambda},a_{L\lambda n}\varphi(s)\big)|_{\mathcal V_i(a_{L\lambda n}\varphi(s_{n}))}\|\leq\lambda_{\mathcal{V}_i}+\frac{\vep}{4}
\end{aligned}
\end{align}
for every $1\leq i\leq k$.
\medskip

\noindent \textbf{Step 7: Upper bound.}
Suppose that $s\in I$ satisfies \eqref{eq:densprop}.
In view of \eqref{eq:normcoc}, for every $j\geq 0$ we have
\[\frac{1}{L\lambda}\big|\log\|A^p(a_{L\lambda},a_{L\lambda j}\varphi(s))\|\big|\leq \frac{\log C+NL\lambda}{L\lambda}\leq C+N\quad\text{for any}\quad j\geq 0.\]
Therefore, by \eqref{neq:fundav} and \eqref{eq:ds}, for any $n\geq n_0(s)$ we have
\begin{align*}
\frac{1}{L\lambda n}&\log\|A^p(a_{L\lambda n},\varphi(s))\|\leq\frac{1}{n}\sum_{j=0}^{n-1}\frac{1}{L\lambda}\log\|A^p(a_{L\lambda},a_{L\lambda j}\varphi(s))\|\\
&\leq \frac{1}{n}\sum_{j\in D(s)\cap[0,n-1]}\frac{1}{L\lambda}\max_{1\leq i\leq k}\log C'\|A^p(a_{L\lambda},a_{L\lambda j}\varphi(s))|_{\mathcal V_i(a_{L\lambda
j}\varphi(s_{j}))}\|\\
&\quad+\frac{1}{n}\sum_{j\in[0,n-1]\setminus D(s)}\frac{1}{L\lambda}\log\|A^p(a_{L\lambda},a_{L\lambda j}\varphi(s))\|\\
&\leq \Big(\lambda_{top}+\frac{\vep}{4}+\frac{\log C'}{L\lambda}\Big)\frac{\#(D(s)\cap[0,n-1])}{n}
+\frac{\#([0,n-1]\setminus D(s))}{n}(C+N)\\
&\leq \lambda_{top}+\frac{\vep}{2}+\vep(C+N)\leq
\lambda_{top}+\vep(1+C+N).
\end{align*}
It follows that  for every $0<\vep<1$ there exist $\lambda>0$ and $L\in\N$ such that for a.e.\ $s\in I$ we have
\begin{equation}\label{neq:fundtop}
\limsup_{n\to\infty}\frac{1}{L\lambda n}\log\|A^p(a_{L\lambda n},\varphi(s))\|\leq\lambda_{top}+\vep(1+C+N).
\end{equation}
\medskip

\noindent \textbf{Step 8: Lower bound.}
Finally, we show that for a.e.\ $s\in I$
\begin{equation}\label{neq:fundbottom}
\liminf_{n\to\infty}\frac{1}{L\lambda n}\log\|A^p(a_{L\lambda n},\varphi(s))\|\geq\lambda_{top}-\vep(\lambda_{top}+1+C+N).
\end{equation}
We stress that the proof of \eqref{neq:fundbottom} is the first and only part which needs the assumption $\lambda_{top}<1$.

Let us consider the cocycle $((A^p)^{-1})^{tr}(g,x)=(A^p(g,x)^{-1})^{tr}$. The Lyapunov exponents of $((A^p)^{-1})^{tr}$ (with respect to $\mu$)
coincide with the additive inverse of the Lyapunov exponents of $A^p$. As the last set is symmetric, the Lyapunov exponents of $((A^p)^{-1})^{tr}$ and $A^p$ are the same.
Moreover, $((A^p)^{-1})^{tr}$ meets all properties of $A^p$ used in the current part of the proof, more precisely, $((A^p)^{-1})^{tr}$ is the composition of $A$ with a
group endomorphism so that the results of \cite{Filip} can be applied. It follows that for every $0<\vep<1$ there exist $\lambda>0$ and $L\in\N$ such that
for a.e.\ $s\in I$ we have
\begin{equation}\label{neq:fundtoptr}
\limsup_{n\to\infty}\frac{1}{L\lambda n}\log\|((A^p)^{-1})^{tr}(a_{L\lambda n},\varphi(s))\|\leq\lambda_{top}+\vep(1+C+N)
\end{equation}
and \eqref{neq:fundtop} holds.

Take any $0<\vep<(1-\lambda_{top})/(1+C+N)$. Since $\lambda_{top}+\vep(1+C+N)<1-\delta<1$ for some $\delta>0$, by \eqref{neq:fundtop} and \eqref{neq:fundtoptr},
for a.e.\ $s\in I$ there exists $n_1(s)\geq 1/\delta$ such that
\begin{equation}\label{neq:top}
\|A^p(a_{L\lambda j},\varphi(s))\|, \|A^p(a_{L\lambda j},\varphi(s))^{-1}\|\leq e^{(1-\delta)L\lambda j}\quad\text{for}\quad j\geq n_1(s).
\end{equation}

Let $s\in I$ so that \eqref{eq:densprop} and \eqref{neq:top} are valid. Recall that a.e.\ $s\in I$ satisfies these conditions.
Assume that $j\in D(s)$ and $j\geq n_1(s)$. Then there exists $J\in\mathcal{I}_j$ and $s_j\in J$ such that \eqref{neq:fundav} holds (for $n=j$).
As $s,s_{j}\in J$ and  $|J|=\sigma e^{-2L\lambda j}$, we have $|s-s_j|\leq \sigma e^{-2L\lambda j}$. Therefore, by \eqref{lipforv}, if $m=j$ or $m=j+1$ then for every $1\leq i\leq k$ we have
\begin{align*}
\Big|\log&\frac{\|A^p\big(a_{L\lambda m},\varphi(s)\big)v_i(s)\|}{\|A^p\big(a_{L\lambda m},\varphi(s)\big)v_i(s_{j})\|}\Big|\\
&\leq\|A^p\big(a_{L\lambda m},\varphi(s)\big)\|\|A^p\big(a_{L\lambda m},\varphi(s)\big)^{-1}\|\|v_i(s)-v_i(s_{j})\|\\
& \leq e^{2(1-\delta)L\lambda m} l \sigma e^{-2L\lambda j}= l\sigma e^{2(1-\delta(j+1))L\lambda}\leq l\sigma.
\end{align*}
In view of \eqref{neq:fundav} and \eqref{condL}, it follows that for every $j\in D(s)$ with $j\geq n_1(s)$ we have
\begin{equation}\label{neq:unique}
\frac{1}{L\lambda}\log\frac{\|A^p\big(a_{L\lambda (j+1)},\varphi(s)\big)v_i(s)\|}{\|A^p\big(a_{L\lambda j},\varphi(s)\big)v_i(s)\|}\geq\lambda_{\mathcal{V}_i}-\frac{\vep}{4}-\frac{2 l\sigma}{L\lambda}\geq\lambda_{\mathcal{V}_i}-\frac{\vep}{2}
\end{equation}
for every $1\leq i\leq k$. Note that \eqref{neq:unique} is the only inequality such that its proof uses the assumption $\lambda_{top}<1$
and it is applied to show \eqref{neq:fundbottom}.

Choose a subbundle $\mathcal{V}_i$ so that $\lambda_{\mathcal{V}_i}=\lambda_{top}$.
In view of \eqref{eq:diflogcoc2}, for every $j\geq 0$ we have
\begin{equation}\label{neq:genfr}
\frac{1}{L\lambda}\Big|\log\frac{\|A^p\big(a_{L\lambda (j+1)},\varphi(s)\big)v_i(s)\|}{\|A^p\big(a_{L\lambda j},\varphi(s)\big)v_i(s)\|}\Big|\leq
\frac{\log C+NL\lambda}{L\lambda}\leq  C+N.
\end{equation}
In view of \eqref{neq:unique}, \eqref{neq:genfr} and  \eqref{eq:ds}, for any $n\geq n_0(s)$ we have
\begin{align*}
\frac{1}{L\lambda n}&\log\|A^p\big(a_{L\lambda n},\varphi(s)\big)\|\geq  \frac{1}{L\lambda n}\log\|A^p\big(a_{L\lambda n},\varphi(s)\big)v_i(s)\|\\
&=\frac{1}{n}\sum_{j=0}^{n-1}\frac{1}{L\lambda}\log\frac{\|A^p\big(a_{L\lambda (j+1)},\varphi(s)\big)v_i(s)\|}{\|A^p\big(a_{L\lambda j},\varphi(s)\big)v_i(s)\|}\\
&=\frac{1}{n}\sum_{j\in [n_1(s),n-1]}\frac{1}{L\lambda}\log\frac{\|A^p\big(a_{L\lambda (j+1)},\varphi(s)\big)v_i(s)\|}{\|A^p\big(a_{L\lambda j},\varphi(s)\big)v_i(s)\|}\\
&\quad+\frac{1}{n}\sum_{j\in([0,n-1]\setminus D(s))\cup[0,n_1(s)-1]}\frac{1}{L\lambda}\log\frac{\|A^p\big(a_{L\lambda (j+1)},\varphi(s)\big)v_i(s)\|}{\|A^p\big(a_{L\lambda j},\varphi(s)\big)v_i(s)\|}\\
&\geq \Big(\lambda_{\mathcal{V}_i}-\frac{\vep}{2}\Big)\frac{\#\big(D(s)\cap[n_1(s),n-1]\big)}{n}
-\frac{n_1(s)+\#\big([0,n-1]\setminus D(s)\big)}{n}(C+N)\\
&\geq \Big(\lambda_{top}-\frac{\vep}{2}\Big)\Big(1-\vep-\frac{n_1(s)}{n}\Big)-\Big(\vep+\frac{n_1(s)}{n}\Big)(C+N)\\
&\geq \lambda_{top}-\Big(\vep+\frac{n_1(s)}{n}\Big)(\lambda_{top}+1+C+N).
\end{align*}
This implies \eqref{neq:fundbottom}.
\medskip

\noindent \textbf{Step 9: Final arguments.}
In view of \eqref{neq:fundtop} and \eqref{neq:fundbottom}, for a.e.\ $s\in I$ we have
\begin{align*}
\lambda_{top}-\vep(\lambda_{top}+1+C+N)&\leq \liminf_{n\to\infty}\frac{1}{L\lambda n}\log\|A^p(a_{L\lambda n},\varphi(s))\|\\
\leq \limsup_{n\to\infty}& \frac{1}{L\lambda
n}\log\|A^p(a_{L\lambda
n},\varphi(s))\|\leq\lambda_{top}+\vep(1+C+N).
\end{align*}
Suppose that $L\lambda n\leq t<L\lambda (n+1)$. Since $0\leq t-L\lambda n<L\lambda$, by \eqref{eq:diflogcoc1},
\[\big|\log\|A^p(a_{t},\varphi(s))\|-\log\|A^p(a_{L\lambda n},\varphi(s))\|\big|\leq \log C+NL\lambda\leq L\lambda(C+N).\]
Therefore,
\begin{align*}
\frac{\log\|A^p(a_{L\lambda n},\varphi(s))\|}{L\lambda(n+1)}-\frac{C+N}{n+1}& \leq \frac{\log\|A^p(a_{t},\varphi(s))\|}{t}\\
&\leq \frac{\log\|A^p(a_{L\lambda
n},\varphi(s))\|}{L\lambda n}+\frac{C+N}{n}.
\end{align*}
It follows that for every $0<\vep<(1-\lambda_{top})/(1+C+N)$ we have
\begin{align*}
\lambda_{top}-\vep&(\lambda_{top}+1+C+N)\leq \liminf_{t\to+\infty}\frac{1}{t}\log\|A^p(a_{t},\varphi(s))\|\\
&\leq \limsup_{t\to+\infty}
\frac{1}{t}\log\|A^p(a_{t},\varphi(s))\|\leq\lambda_{top}+\vep(1+C+N)
\end{align*}
for a.e.\ $s\in I$.
Since $\vep>0$ can be chosen arbitrary small, we obtain
\[\lim_{t\to+\infty} \frac{1}{t}\log\|A^p(a_{t},\varphi(s))\|=\lambda_{top}\text{ for a.e.\ }s\in I.\]
\end{proof}

\subsection{Proof of Theorem~\ref{thm:Oseledetsgen} and applications}\label{sec:no_assumption}
In this subsection we explain how the proof of Theorem~\ref{thm:Oseledets}, in virtue of the results proved in \cite{EFW}, can be modified in order to prove Theorem~\ref{thm:Oseledetsgen} in the introduction (which has no assumption on the sum of Lyapunov exponents). We then deduce a result for loci of translation surfaces which are ramified covers of another surface  (see Theorem \ref{thm:no_sum_assumption} 
at the end of this subsection).
\smallskip
Let us consider the splitting $\bigoplus_{i=1}^k\mathcal{V}_i$ of the bundle
$\bigwedge^p \mathcal{W}\to \mathcal{M}$ into strongly irreducible invariant subbundles $\mathcal{V}_i$.
\begin{thm}\label{thm:Oseledets_Apgen}
Assume that $\lambda_{top}>0$ is the top Lyapunov exponent $A^{p}:SL_2(\R)\times\mathcal{M}\to GL(\bigwedge^p
W)$ with respect to the measure $\mu$. Suppose that all subbundles $\mathcal{V}_i\to\mathcal{M}$, $1\leq i\leq k$ are all locally constant.
Suppose that $\varphi:I\to\mathcal{M}$
is a $C^1$-curve which is well approximated by horocycles  and  $\varphi(s)\in\mathcal{M}$ is Birkhoff
generic with respect to  $(\mathcal{M},\mu,a_t)$  for a.e.\
$s\in I$. Then for a.e.\ $s\in I$ we
have
\[\lim_{t\to\infty}\frac{1}{t}\|A^p(a_t,\varphi(s))\|=\lambda_{top}.\]
\end{thm}

\begin{proof}
The proof is essentially the same  as that of Theorem~\ref{thm:Oseledets_red}. As we have already observed,
the only part of the proof of Theorem~\ref{thm:Oseledets_red} where the assumption $\lambda_{top}<1$ is used is Step 8, for the the proof  the inequality \eqref{neq:fundbottom}.   More precisely, the  assumption on the Lyapunov exponent is only used in Step 8 to show the inequality \eqref{neq:unique}.

If all subbundles $\mathcal{V}_i$, $1\leq i\leq k$ are all locally constant then we can choose the curves $s\mapsto v_i(s)$, $1\leq i\leq k$
to be constant. Then \eqref{neq:unique} follows directly from \eqref{neq:fundav}, since $v_i(s_n)=v_i(s)$. This completes the proof.
\end{proof}

Filip in \cite{Filip} showed that the subbundles $\mathcal{V}_i\to\mathcal{M}$, $1\leq i\leq k$ locally vary polynomially in the period coordinates.
In view of \cite{EFW}, one can deduce  that $\mathcal{V}_i\to\mathcal{M}$, $1\leq i\leq k$ are indeed locally constant 
whenever the space $W\subset H_1(M,\R)$ satisfies the assumptions of Theorem~\ref{thm:Oseledetsgen}, as we now explain.

Let $\mathcal{M}\subset \mathcal{M}_1(M,\Sigma)$ be an orbit closure and let $\mu$ be the corresponding probability affine measure. Since $\mathcal{M}$ is an affine submanifold its tangent  bundle $T\mathcal{M}$,
in period coordinates (see \cite{AEM}) is determined by a subspace $T_{\C}\mathcal{M}\subset H^1(M,\Sigma,\C)$ such that
\[T_{\C}\mathcal{M}=\C\otimes T_{\R}\mathcal{M}\quad\text{ where }\quad T_{\R}\mathcal{M}\subset H^1(M,\Sigma,\R).\]
  Let $p:H^1(M,\Sigma,\R)\to H^1(M,\R)$ be the natural projection. Let us consider the space $p(T_{\R}\mathcal{M})\subset H^1(M,\R)$.
Since $H^1(M,\R)$ and $H_1(M,\R)$ can be identified by the Poincar\'e duality, we will also denote by  $p(T_{\R}\mathcal{M})$ the corresponding
subspace of $H_1(M,\R)$. Avila-Eskin-M\"oller, in \cite{AEM}, proved that $p(T_{\R}\mathcal{M})$ is an $SL_2(\R)$-invariant symplectic subspace.

\begin{thm}[\cite{EFW}]\label{thm:EFW}
Suppose that $W\subset H_1(M,\R)$ is an $SL_2(\R)$-invariant symplectic subspace which
is symplectic orthogonal to $p(T_{\R}\mathcal{M})$.
Let $\mathcal{W}\to\mathcal{M}$ be the corresponding bundle.
Then for every $p\geq 1$ the bundle $\bigwedge^p \mathcal{W}\to\mathcal{M}$ has a splitting into
\emph{locally constant} strongly irreducible $SL_2(\R)$-invariant subbundles.
\end{thm}

We can now conclude the proof of Theorem~\ref{thm:Oseledetsgen}.
\begin{proof}[Proof of Theorem~\ref{thm:Oseledetsgen}]
After the reduction described in Remark~\ref{rem:reduct}, Theorem~\ref{thm:Oseledetsgen} follows directly form
Theorem~\ref{thm:Oseledets_Apgen} combined with Theorem~\ref{thm:EFW}.
\end{proof}

Finally, let us remark that the key assumption in Theorem~\ref{thm:Oseledetsgen}, namely that $W$ is symplectic orthogonal to $p(T_{\R}\mathcal{M})$, holds for some natural classes of  invariant subspaces when $M$ is a ramified
cover of another surface.
Suppose that $q:M\to N$ is a covering map,
ramified over a finite set $\Sigma_q\subset N$. Let us consider a stratum $\mathcal{M}({\alpha})\subset \mathcal{M}(N,\Sigma_q)$.
Denote by  $\widetilde{\mathcal{M}}({\alpha})$ the moduli space of all translation surfaces of the form $(M,q^*\nu)$ for $\nu\in \mathcal{M}({\alpha})$.
Denote by $D$ the deck group of the ramified cover $q:M\to N$. Let us consider the subspace $H_1^{+}(M,\R)\subset H_1(M,\R)$ of all $D$-invariant elements.
Then $H_1^{+}(M,\R)$ is a symplectic $SL_2(\R)$-invariant subspace which is naturally identified with $H_1(N,\R)$. The
symplectic orthocomplement
\[
H_1^{-}(M,\R):=H_1^{+}(M,\R)^{\perp}
\]
plays a crucial role in billiard or lenses applications.
Also the space  $H_1^{-}(M,\R)$ is clearly symplectic  and $SL_2(\R)$-invariant.
\begin{lem}\label{lem:orthogonality}
\begin{sloppypar}
If an orbit closure $\mathcal{M}$ is contained in $\widetilde{\mathcal{M}}({\alpha})$ then
$ p(T_{\R}\mathcal{M})\subset H_1^{+}(M,\R)$. Consequently, $H_1^{-}(M,\R)$ is symplectic orthogonal to $p(T_{\R}\mathcal{M})$.
\end{sloppypar}
\end{lem}
\begin{proof}
Let us consider the induced action of the group $D$ on $\mathcal{M}(M,\Sigma)$. By assumption,
$\mathcal{M}$ is a subset of the set $\Fix_{G}\mathcal{M}(M,\Sigma)$ of fixed points  of this action.
Next, let us consider the induced action of $D$ on the tangent bundle $T\mathcal{M}(M,\Sigma)$.
Then \[T\mathcal{M}\subset \Fix_{D}T\mathcal{M}(M,\Sigma).\]
Since each tangent space of $T\mathcal{M}(M,\Sigma)$
is identified with $H^1(M,\Sigma,\C)$ and the identification is $D$-equivariant, we have
\[T_\C\mathcal{M}\subset \Fix_{D}H^1(M,\Sigma,\C)\]
and hence
\[T_\R\mathcal{M}\subset \Fix_{D}H^1(M,\Sigma,\R).\]
It follows that
\[p(T_\R\mathcal{M})\subset \Fix_{D}H^1(M,\R).\]
Finally, using the Poincar\'e duality we have $p(T_{\R}\mathcal{M})\subset H_1^{+}(M,\R)$.
\end{proof}
Hence, we can now  conclude the present subsection by formulating a  result which we hope will be useful for further mathematical physics applications.

Recall that $\widetilde{\mathcal{M}}({\alpha})$ denotes the moduli space of all translation surfaces which are ramified covers $q:M\to N$ of the form $(M,q^*\nu)$ of $\nu\in \mathcal{M}({\alpha})$.
\begin{thm}\label{thm:no_sum_assumption}
Let $\mathcal{M}$ be an $SL_2(\mathbb{R})$-orbit closure contained in $\widetilde{\mathcal{M}}({\alpha})$ and let $\mu$ be the probability affine measure of $\mathcal{M}$.
Suppose that $\varphi:I\to\mathcal{M}$ is a curve  well
approximated by horocycles   such that $\varphi(s)\in\mathcal{M}$ is Birkhoff generic with respect to  $(\mathcal M,\mu, a_t)$ for a.e.\ $s\in I$.
Assume that $W\subset H_1^-(M,\R)$ is an $SL_2(\R)$-invariant  symplectic  subspace.
Then
$\varphi(s)\in\mathcal{M}$ is Oseledets generic with respect to  $(\mathcal M, \mu,  a_t, A_W^{KZ})$ for
a.e.\ $s\in I$.
\end{thm}

\section*{Acknowledgments}
We are indebted to J.~Marklof, both for asking  to R.~Shi the question on gaps distributions (answered in Theorem~\ref{thm;g;gap}) and for showing to C.~Ulcigrai the paper \cite{Dra-Ra}
by Dragovi\'c and Radnovi\'c (without knowing that he was suggesting two applications of the same result). We are grateful to J.~Chaika and A.~Eskin
 for explaining to us their work in \cite{Es-Ch} and patiently answering  all our questions on it. We thank
Y.~Benoist, Y.~Cheung, A.~Mohammadi, A.~Zorich and B.~Weiss
 for useful discussions and comments.
We would also like to thank  the
anonymous referee for suggestions to improve the clarity of the paper.

We acknowledge both the \emph{Max Planck Institute} in Bonn and the  \emph{Mathematical Science Research Institute} in Berkeley, CA for their hospitality when part of this collaborative work was done, respectively when K.~Fr\k{a}czek and C.~Ulcigrai were visitors during the program "\emph{Dynamics and Numbers}" in Bonn  and when R.~Shi and C.~Ulcigrai were research members of the program "\emph{Geometric and Arithmetic Aspects of Homogeneous Dynamics}" at MSRI.  We also  thank
 the \emph{University of Bristol} for hosting R.~Shi for  a seminar and research visit which  started this collaboration.  Finally, part of this work was done under NSF 0932078 000 while R.~Shi and C.~Ulcigrai were visiting  MSRI during Spring 2015.

K.~Fr\k{a}czek  is partially supported by the National Science Centre (Poland) grant 2014/13/B/ST1/03153.
R.~Shi is supported by NSFC (11201388) and NSFC (11271278).
C.~Ulcigrai is currently supported by the ERC Starting Grant ChaParDyn and the research leading to these results has received funding from the European Research Council under the European Union's Seventh Framework Programme (FP/2007-2013) / ERC Grant Agreement n.\ 335989. C.~Ulcigrai also gratefully acknowledges the support of the \emph{Leverhulme Trust} via a \emph{Leverhulme Prize} and of the Royal Society and the Wolfson Foundation via a \emph{Royal Society Wolfson Research Merit Award}.

\end{document}